\documentclass[draft, oneside, reqno]{amsart}
\usepackage{amsmath,amsthm}
\usepackage{amssymb,latexsym, eucal}
\usepackage{amscd} %For commutative diagrams.
\usepackage{verbatim}  %For comments while editing.
\usepackage{epic} %For enhanced pictures.
\usepackage{caption} %For captions outside of floats
\newtheorem{theorem}{Theorem}[subsection]  %Number theorems with subsections.
\newtheorem{lemma}[theorem]{Lemma}
\newtheorem{proposition}[theorem]{Proposition}
\newtheorem{corollary}[theorem]{Corollary}

\theoremstyle{definition}
\newtheorem{definition}[theorem]{Definition}

\newtheorem{example}[theorem]{Example}

\newtheorem{remark}[theorem]{Remark}

\newtheorem{specialcases}[theorem]{Special Cases}

\DeclareMathOperator\ad{ad}

\DeclareMathOperator\Aut{Aut}

\DeclareMathOperator{\Centrd}{C}

\DeclareMathOperator\diag{diag}
\DeclareMathOperator\End{End}
\DeclareMathOperator\FSk{FSkew}

\DeclareMathOperator\GL{GL}

\DeclareMathOperator\id{id}

\DeclareMathOperator\Hom{Hom}

\DeclareMathOperator\Pic{Pic}

\DeclareMathOperator\Spec{Spec}

\DeclareMathOperator\supp{supp}
\DeclareMathOperator\tr{tr}
\DeclareMathOperator\Sk{Skew}
\DeclareMathOperator\spann{span}

\newcommand\al{\alpha}
\newcommand\ep{\varepsilon}

\newcommand\De{\Delta}
\newcommand\lm{\lambda}

\newcommand\ph{\varphi}
\newcommand\sg{\sigma}

\newcommand\bbC{\mathbb{C}}

\newcommand\bbF{\mathbb{F}}

\newcommand\bbT{\mathbb{T}}
\newcommand\bbZ{\mathbb{Z}}

\newcommand\cE{\mathcal{E}}
\newcommand\cF{\mathcal{F}}
\newcommand\cG{\mathcal{G}}
\newcommand\cH{\mathcal{H}}

\newcommand\cP{\mathcal{P}}
\newcommand\cS{\mathcal{S}}
\newcommand\cT{\mathcal{T}}

\newcommand\bP{{\breve P}}
\newcommand\brv{\breve{\hphantom{a}}} %Empty breve

\newcommand\fg{\mathfrak{g}}

\newcommand\fp{\mathfrak{p}}
\newcommand\Kan{\mathfrak{K}}

\newcommand\fS{\mathfrak{S}}

\newcommand\tE{\tilde{\cE}}
\newcommand\ttheta{\tilde{\theta}}
\newcommand\tS{\tilde{\cS}}

\newcommand\rL{L}  %For general Lie algebras.
\newcommand\bR{\mathbb{K}} %For the base ring.

\newcommand \andd{\quad\text{and}\quad}

\newcommand\mat[1]{\left[\begin{matrix} #1 \end{matrix}\right]}  %Macro for matrices
\newcommand\smat[1]{\left[\begin{smallmatrix} #1 \end{smallmatrix}\right]}  %Macro for small matrices
\newcommand \suchthat { : }
\newcommand \set[1]{\{#1 \}}

\newcommand\type[2]{\rm{#1}_{#2}}  %For the type of a root system

\newcommand\BCone{\type{BC}{1}}
\newcommand\BCtwo{\type{BC}{2}}
\newcommand\Esix{\type{E}{6}}
\newcommand\Kalg{\bR \mathrm{-alg}}
\newcommand\msg{{-\sg}}
\newcommand\op{{\textrm{op}}}
\newcommand\p[1]{{}^{#1}}
\newcommand\tprod{{   \{\, ,\, ,\, \} }}
\newcommand\tripF{(M^-_\bbF,M^+_\bbF,g_\bbF)}
\newcommand\fo{{\mathfrak o}}
\newcommand\ffo{{\mathfrak{fo}}}
\newcommand\tg{{\tilde g}}
\newcommand\TLP{{T_L(P)}}
\newcommand\tV{\tilde V}
\newcommand\Vm{V^-}

\newcommand\Vp{V^+}
\newcommand\Vs{V^\sg}
\newcommand\Wg{{\textstyle \bigwedge}}
\newcommand\Xn{\type{X}{n}}
\DeclareRobustCommand\Wgref{\Wg_3\hspace{-4pt}\breve{}\hspace{4pt}}
\newcommand\Lbp{\overline{L'}}

\newcommand\Sec{Section }
\newcommand\Subsec{Subsection }

\begin{document}
\title[Weyl images of Kantor pairs]{Weyl images of Kantor pairs}
%***********
\author{Bruce Allison}
\address[Bruce Allison]{Department of Mathematical and Statistical Sciences \\ University
of Alberta \\ Edmonton, Alberta, Canada T6G 2G1 }
\email{ballison@ualberta.ca}
\author{John Faulkner}
\address[John Faulkner]
{Department of Mathematics\\
University of Virginia\\
Kerchof Hall, P.O.~Box 400137\\
Charlottesville VA 22904-4137 USA}
\email{jrf@virginia.edu}
\author{Oleg Smirnov}
\address[Oleg Smirnov]
{Department of Mathematics\\
College of Charleston\\
Charleston, SC, USA 29424-0001}
\email{smirnov@cofc.edu}

\let\thefootnote\relax\footnotetext{To be published in the Canadian Journal of Mathematics, DOI 10.4153/CJM-2016-047-1,  \textcopyright~2016 Canadian Mathematical Society}

%\date{July 8, 2015}
\begin{abstract} Kantor pairs arise naturally in the study of $5$-graded Lie algebras.  In this article, we introduce
and study Kantor pairs with short Peirce gradings and relate them to Lie algebras
graded by the root system of type $\BCtwo$.  This relationship allows us to define so called Weyl images
of short Peirce graded Kantor pairs.  We use Weyl images to construct new examples of Kantor pairs, including a  class of infinite  dimensional
central simple Kantor pairs over a field of characteristic $\ne 2$ or $3$, as well as a family of forms of a split
Kantor pair of type  $\type{E}{6}$.  \end{abstract}
\maketitle

Assume for simplicity  in this introduction that  $\bR$ is a ring of scalars containing
$\frac 16$ (although we will relax this
assumption in a few parts of the paper).  A \emph{Kantor pair} over $\bR$ is a pair $P =  (P^- , P^+)$ of $\bR$-modules
together with two trilinear
products $\{\ ,\ ,\ \}^\sg : P^\sg \times P^\msg \times P^\sg \to P^\sg$, $\sg = \pm$, satisfying
two $5$-linear identities (K1) and (K2)   (see \Subsec \ref{subsec:KPs} or \cite{AF1}).
These structures arise naturally in the study of \emph{$5$-graded Lie algebras}, by which we mean
$\bbZ$-graded Lie algebras of the form $L = L_{-2}\oplus L_{-1}\oplus L_0\oplus L_1\oplus L_2$.
Indeed, if $L$ is $5$-graded Lie algebra, then the pair $(L_{-1} , L_1)$ has the structure of a Kantor pair,
called \emph{the Kantor pair  enveloped by $L$}, where the two products are restrictions of the product
$[[x,y],z]$ on $L$.
Conversely, given a Kantor pair $P$ there exists a $5$-graded Lie algebra
(not in general unique) that  envelops $P$.
This relationship between
$5$-graded Lie algebras and Kantor pairs,
which we view schematically as
\begin{equation} \label{eq:detscheme} \text{$5$-graded Lie algebras} \ \ \rightsquigarrow \ \ \text{Kantor pairs},
\end{equation}
is an important tool in the study of each of these structures,
and it generalizes the well-known relationship between $3$-graded Lie algebras and  Jordan pairs
\cite[\S 1.5]{N2}.

To describe  some background, we note that Isai Kantor studied a class of triple systems, which we call
\emph{Kantor triple systems}, in his foundational paper \cite{K1}.  He developed the relationship of
Kantor triple systems with $5$-graded Lie algebras that possess grade-reversing  period 2 automorphisms.
He used this relationship to  obtain a classification of  finite dimensional nonpolarized
(see \Subsec \ref{subsec:trilinear}) simple Kantor
triple systems over an algebraically closed field of characteristic $0$.

Kantor triple systems constitute one of the largest classes of nonassociative objects for which such a
classification result has been obtained. The class includes  Jordan  triple systems  as well as triple
systems constructed from associative algebras, alternative algebras,   Jordan algebras and many other
interesting exceptional objects.

Given a Kantor triple system one can, as in the Jordan case, construct a Kantor pair by doubling
(see \Subsec \ref{subsec:trilinear}), but not every Kantor pair arises in this way.
So in this sense Kantor pairs are generalizations of Kantor triple systems. Moreover,
pairs are  more natural objects to consider from the viewpoint  of graded Lie algebras,
since $5$-graded Lie algebras need not possess grade-reversing  period 2 automorphisms
 (see Remark \ref{rem:KP}(iii)).

Kantor pairs also arise using signed doubling of some analogs of
Kantor triple systems
called $(1,1)$-Freudenthal-Kantor triple systems, including Freudenthal triple systems (with a suitably modified product). (See Example \ref{ex:symplectic}.)  Freudenthal triple systems
have been studied mathematically  by many authors and have appeared recently in several important physical models
(see \cite{G}, \cite{BDDRE}, \cite{MQSTZ}, and the references   therein).

In this paper, we  introduce  short Peirce  gradings (SP-gradings) of Kantor pairs.  We describe a relationship  between Lie algebras graded by the root system $\Delta$ of type $\BCtwo$  and SP-graded Kantor pairs.
Using this relationship we  define a Weyl image $\p{u}P$ of $P$ for each SP-graded Kantor pair $P$ and each element
$u$ of the Weyl group of $\Delta$.  We develop the properties
of Weyl images and use them to construct new examples of Kantor pairs.

Although we obtain our results and examples for the most part without assuming that $\bR$ is a field, this article
is the beginning of an investigation of
central simple Kantor pairs  over a field, so we have particular interest in that case.
Our results here will be used
in a paper \cite{AS} by the first and third authors  that contains a
structure theorem for central simple Kantor pairs over a field of characteristic  $\ne 2$, $3$ or $5$.  The
theorem asserts that
these pairs occur in four classes:  a class of Jordan pairs, the class of
(finite dimensional)
forms  of split Kantor pairs of exceptional type (see Subsection \ref{subsec:splitKP}), a new class of Kantor
pairs constructed from hermitian forms, and a new class that we introduce
in \Sec \ref{sec:skew}  using Weyl  images.

We conclude this introduction by briefly outlining the contents of the paper.
After some preliminaries on root graded Lie algebras  in \Sec \ref{sec:rootgraded} and trilinear pairs
in \Sec \ref{sec:trillinear}, we recall or prove some basic properties of Kantor pairs
in \Sec \ref{sec:KPs}. One such property is that  the relationship
\eqref{eq:detscheme} restricts to a \emph{one-to-one} correspondence between central simple
$5$-graded Lie algebras up to graded isomorphism and central simple Kantor pairs up to isomorphism.

In \Sec \ref{sec:SP}, we introduce  SP-graded Kantor pairs and
$\BCtwo$-graded Lie algebras.
An \emph{SP-grading}  of a Kantor pair $P$ is a $\bbZ$-grading of $P$ whose
support is contained in $\set{0,1}$; whereas a \emph{$\BCtwo$-graded Lie algebra}
is a Lie algebra graded by the root lattice of the root system $\Delta$ of type $\BCtwo$
with support contained in $\Delta\cup\set0$.  (The latter definition is convenient
for us but not standard.  See \Subsec \ref{subsec:rootgraded}.)
We  establish the   relationship mentioned above between
$\BCtwo$-graded Lie algebras and SP-graded Kantor pairs, and deduce some of its properties.
(It can be viewed as the rank two version of \eqref{eq:detscheme}, since
$5$-graded Lie algebras are precisely the same as $\BCone$-graded Lie algebras.)

In Section \ref{sec:Weyl}, we define and study Weyl images of SP-graded Kantor pairs.
To describe these briefly, let $P$ be an SP-graded Kantor pair and
let $u$ be an element of the Weyl group $W_\Delta$ of the root system $\Delta$ of type $\BCtwo$. Then $P$
is  enveloped by a
$\BCtwo$-graded Lie algebra $L$; and we  use $u$ to adjust the grading
of $L$ in an evident fashion to obtain a $\BCtwo$-graded Lie algebra
$\p{u}L$, which in turn  envelops a Kantor pair $\p{u}P$, called a \emph{Weyl image} of $P$.  This gives us a well-defined action of the group $W_\Delta$
on the class of SP-graded Kantor pairs, and one  sees that
Weyl images of central simple SP-graded Kantor pairs are central simple.
 A particularly interesting case
occurs when $u$ is the reflection corresponding to the short basic root,
in which case we denote $\p{u}P$
by $\bP$ and call it simply the \emph{reflection} of $P$.

The reader may  initially suspect that the reflection $\bP$ of an SP-graded Kantor pair $P$ is just $P$ with a different SP-grading.
However, it turns out that $P$ and $\bP$  are not in general isomorphic \emph{even as ungraded pairs}.  This suggests a strategy for
giving new constructions of Kantor pairs:   Start with a Kantor pair $P$, choose an appropriate SP-grading of $P$ and
form the reflection $\bP$ of $P$ with that grading.
In the last two sections we look in detail at two examples of this strategy where we obtain a pair $\bP$ with quite different
properties   than $P$.

First, in \Sec \ref{sec:skew}, we
start with a nondegenerate   bilinear form  $g : V^- \times V^+ \to \bR$.  Let $\tg$ be the symmetric
bilinear form on $\tV = V^- \oplus V^+$ that extends $g$ and is zero
on  $V^\sg \times V^\sg$, $\sg = \pm$; and let  $\ffo(\tg)$ be Lie algebra
spanned by endomorphisms of $\tV$ of the form $x \mapsto g(x,w)v - g(x,v)w$,
where $u,v\in V$.
If there exists  $e = (e^-,e^+)\in V^-\times V^+$ such that
$g(e^-,e^+) = 1$, then $\ffo(\tg)$ has a natural $\BCtwo$-grading,
so it  envelops an SP-graded Kantor pair $\FSk(g)$.
The pair $\FSk(g)$ is known to be Jordan \cite{LB}, but its reflection $\FSk(g)\brv$ is not in general.
Moreover, in the case when $\bR$ is a field and $\dim(V^\sg) \ge 3$, Kantor pairs of the form $\FSk(g)\brv$ make up the
fourth class of central simple Kantor pairs appearing in the structure theorem  mentioned  above.

In  \Sec \ref{sec:KPE6}, we use a non-singular bilinear form
$g : M^- \times M^+ \to \bR$, where each $M^\sg$ is a finitely generated projective
module over $\bR$ of rank $6$.
Following the approach in \cite{F}, we
use the exterior algebras $\Wg(V^\sg)$, $\sg = \pm$, to construct a form $\cE = \cE(M^-,M^+,g)$
of the split Lie algebra of type  $\Esix$.
In the case when $\bR = \bbC$, this is a basis free version, with full proofs, of the construction of the complex Lie algebra $\Esix$
given by \'Elie Cartan in his 1894 thesis \cite[\S V.18, pp.89--90]{C}.
If there exists  $e = (e^-,e^+)\in M^-\times M^+$ such that
$g(e^-,e^+) = 1$, then
$\cE$ has a natural $\BCtwo$-grading,  which is strikingly similar to the
grading of $\ffo(\tg)$ arising in \Sec \ref{sec:skew}.  The
SP-graded Kantor pair  enveloped by $\cE$ has the form
$\Wg_3 = (\Wg_3(V^-),\Wg_3(V^+))$ with an easily remembered basis
free product and a natural SP-grading, which we use to
construct the reflection $\Wgref$.  As an example, we see
that when $\bR$ is a
Dedekind domain, the set of isomorphism classes of Kantor pairs of the form
$\Wg_3$ (resp.~$\Wgref$) is  parameterized by the Picard group of  $\bR$.
Suppose finally that $\bR$ is a field, in which case $\Wg_3$ is central simple.
Although the pair $\Wg_3$ is not Jordan,
it is close to Jordan  in a sense that we make   precise in \Subsec \ref{subsec:obstruct}.
In contrast,  $\Wgref$  is not close to Jordan; and it turns out that it
is a Kantor pair of particular interest in the  theory of finite dimensional central
simple Kantor pairs (see Remark \ref{rem:Wg3refb}).
The pair $\Wgref$ is the double of a Kantor triple system $C_{55}^2$ that was constructed originally by Kantor
using tensors $a_{ijk}$ which are  skew-symmetric with respect to $i,j$,
where $i, j = 1,\dots, 5$, $k =1,2$ \cite[\S4]{K2}.
Our approach using reflection gives a simple new construction of this interesting Kantor pair.

\smallskip \noindent\textbf{Acknowledgements.}  The authors thank Erhard Neher and the referee who each made suggestions that significantly improved
this article.

\smallskip
\noindent\textbf{Assumptions and notation.} \emph{Throughout the rest of the article,  we  assume that $\bR$ is a unital commutative associative
ring of scalars}.
In much of the article, we will also assume that $\frac 16\in \bR$ (and clearly state this assumption).  We do this because Kantor pairs have not even been defined without the assumption that $\frac 16\in \bR$, and it is not yet clear  what the definition should be.  However, one  place where we do not assume $\frac 16 \in \bR$ is in Section \ref{sec:KPE6}, where we think the Lie algebra constructions are of independent interest without restriction on~$\bR$.

We shall require a  $\bR$-module to be unital;
i.e., $1x = x$.
Unless indicated otherwise,  by  a module (resp.~an algebra) we will mean a module (resp.~an algebra) over~$\bR$.
 If $V$ and $W$ are modules, we will often abbreviate $\Hom_\bR(V,W)$ and $\End_\bR(V)$
by $\Hom(V,W)$ and $\End(V)$
respectively.
 Then, as usual,  $\End(V)$ is an associative algebra under composition and a Lie algebra under the commutator product (and it will always be clear which is being considered).
If $V$ is a module, we use the notation
$V^* = \Hom(V,\bR)$ for the \emph{dual module} of $V$.
If $\bR$ is a field, we often abbreviate $\dim_\bR(V)$ by $\dim(V)$.

If $V$ and $W$ are modules and $g: V\times W \to \bR$ is a bilinear form,
we say that $g$ is \emph{non-degenerate} (resp.~\emph{non-singular}) if the maps
$v \rightarrow g(v,\  \ )$ from $V$ into $W^*$
and  $w \rightarrow g(\ \ ,w)$ from $W$ into $V^*$ are injective (resp.~bijective).

Recall  that a module $W$ is said to be \emph{flat} (resp.~faithfully flat)  if for a exact sequence $V'\rightarrow V\rightarrow V''$ of modules to be exact
it is necessary (resp.~necessary and sufficient) that the induced sequence
$W\otimes_{\bR}V'\rightarrow W\otimes_{\bR}V\rightarrow
W\otimes_{\bR}V''$ is exact \cite[I.2 and I.3]{B2}.

Let $\Kalg$ denote  the category of unital commutative associative $\bR$-algebras.
We say that $\bbF\in\Kalg$ is \emph{flat} (resp.~\emph{faithfully flat}) if $\bbF$ is a flat
(resp.~faithfully flat) $\bR$-module.  Note that if $\bR$ is a field,
then any $\bbF\in\Kalg$ is non-trivial and free and hence faithfully flat.

If $V$ is
a module and $\bbF \in \Kalg$,  we write  $V_\bbF := \bbF \otimes_\bR V$.
If $V$ is  a $\bR$-algebra, then $V_\bbF$ is naturally an $\bbF$-algebra.  If $\ph : V \to W$
is a homomorphism of modules, we denote the induced homomorphism of $\bbF$-modules by $\ph_\bbF : V_\bbF \to W_\bbF$.
If $g : V \times W\to \bR$ is a bilinear form, we
have a unique $\bbF$-bilinear form $g_\bbF : V_\bbF \times W_\bbF \to \bbF$, which
we say is \emph{induced} by $g$,  such that  $g_\bbF(1 \otimes x,1 \otimes y) = g(x,y)$ for $x\in V$, $y\in W$.

Finally, if $X = \bigoplus_{i} X_i$ is a direct sum of modules
 and $\bbF\in \Kalg$, then there is a canonical identification
of $(X_i)_\bbF$ as an $\bbF$-submodule of $X_\bbF$ so that
$X_\bbF = \bigoplus_{i} (X_i)_\bbF$.  In this way
a $G$-graded algebra $X = \bigoplus_{g\in G} X_g$ , where $G$ is an abelian group, yields a $G$-graded $\bbF$-algebra
$X_\bbF = \bigoplus_{g\in G} (X_g)_\bbF$.

\section{Root graded  Lie algebras}
\label{sec:rootgraded}

\subsection{Root systems}

In this
\label{subsec:rootsystems}
paper, a \emph{root system} will mean a finite root system $\De$ in a finite
dimensional real Euclidean space $E_\Delta$ as
described for example in  \cite[VI.3]{B3}.  We  use the notation $Q_\De := \spann_\bbZ(\Delta)$
for the \emph{root lattice} of $\Delta$.
The \emph{automorphism group of $\De$}, denoted by
$\Aut(\De)$, is the stabilizer of $\De$ in $\GL(E_\De)$.  Using the restriction map,
we often identify  $\Aut(\De)$ with the stabilizer of $\De$ in $\Aut(Q_\De)$.
The \emph{Weyl group}  of $\De$, which is a subgroup of $\Aut(\De)$, will be denoted by
$W_\De$.
A root system $\De$ is said to be \emph{reduced} if $\De\cap(2\De) = \emptyset$. Recall that for each
rank $n \ge 1$, there exists a unique irreducible non-reduced root system of rank $n$ up to
isomorphism   \cite[VI.1.4, Prop.~14]{B3}.  This root system is said to have \emph{type}~$\type{BC}n$.

\subsection{Root graded Lie algebras}
\label{subsec:rootgraded}
Let $\De$ be a root system.
A \emph{$\De$-grading} of a  Lie algebra $\rL$ is a $Q_\De$-grading of $\rL$
such that $\supp_{Q_{\Delta}}(\rL) \subseteq \Delta \cup \set0$,
where  $\supp_{Q_{\Delta}}(\rL)$ denotes the \emph{support} of $\rL$ in $Q_\De$.
In that case we call $\rL$ together with the $\De$-grading a
\emph{$\De$-graded Lie algebra}.
(We note that this definition is less restrictive
than the one used in \cite{ABG}, \cite{BS} and  several earlier papers,
since we do not assume the existence of a grading subalgebra.  Our usage is natural here---see in particular Section 3---and will not cause the reader any confusion.)  If
$\Delta$ is irreducible of type $\Xn$, we often say that $\rL$ is \emph{$\Xn$-graded}; and we often refer to a $Q_{\Delta}$-graded isomorphism of $\Xn$-graded Lie algebras as  an
\emph{$\Xn$-graded isomorphism}.

If we fix a base $\Gamma= \set{\al_1,\dots,\al_n}$ for $\De$, we can identify $Q_\De$
with $\bbZ^n$ using the $\bbZ$-basis $\Gamma$ for $Q_\De$.
With this identification, every $\De$-graded Lie algebra is
a $\bbZ^n$-graded Lie algebra  (but not conversely of course).

\subsection{Images of $\De$-graded Lie algebras under the left action of \protect $\Aut(\De)$ \protect}
\label{subsec:Weylactionroot}
Suppose that $\rL$ is  a $\De$-graded Lie algebra.  If $\theta\in \Aut(\De)$, we  let $\p{\theta}\rL$ be the
$\De$-graded Lie algebra  such that $\p{\theta}\rL = \rL$ as Lie algebras and
\begin{equation}
\label{eq:imagedef}
(\p{\theta}\rL)_\al = \rL_{\theta^{-1}\al}
\end{equation}
for $\al\in Q_\De$.  We
call $\p{\theta}\rL$ the \emph{$\theta$-image} of $\rL$, and if $\theta\in W_\De$
we call   $\p{\theta}\rL$
a \emph{Weyl image} of $\rL$.
Clearly
\begin{equation} \label{eq:actrootgraded}
\p1 \rL = \rL \andd \p{\theta_1}(\p{\theta_2}\rL) = \p{\theta_1\theta_2}\rL
\end{equation}
for $\theta_1,\theta_2\in \Aut(\De)$,
so we have a left action of $\Aut(\De)$ on the class of $\De$-graded  Lie algebras.

\section{Trilinear pairs}
\label{sec:trillinear}

\emph{Unless mentioned to the contrary we  will assume throughout Sections \ref{sec:trillinear}--\ref{sec:skew} that  $\bR$ contains $\frac 16$.}

\subsection{Terminology}
A \emph{trilinear pair} is a pair $P = (P^-,P^+)$ of
\label{subsec:trilinear} modules together with two trilinear
maps $\tprod^\sg : P^\sg \times P^\msg \times P^\sg \to P^\sg$, $\sg = \pm$,
which we call the products on $P$.
If needed, we will call $\tprod^\sg$ the \emph{$\sg$-product} on $P$.
We define the  \emph{D-operator}
$D^\sg(x^\sg,y^\msg)\in \End(P^\sg)$ for  $x^\sg\in P^\sg$, $y^\msg\in P^\msg$ by
\[D^\sg(x^\sg,y^\msg)z^\sg  = \{x^\sg,y^\msg,z^\sg\}^\sg;\]
and we define the  \emph{K-operator}
$K^\sg(x^\sg,z^\sg)\in \Hom(P^\msg,P^\sg)$ for $x^\sg,z^\sg\in P^\sg$
by
\[K^\sg(x^\sg,z^\sg)y^\msg = \{x^\sg,y^\msg,z^\sg\}^\sg-\{z^\sg,y^\msg,x^\sg\}^\sg.\]
When no confusion will arise, we usually write
$\tprod^\sg$, $D^\sg(x^\sg,y^\msg)$ and $K^\sg (x^\sg,z^\sg)$ simply as
$\tprod$, $D(x^\sg,y^\msg)$ and $K (x^\sg,z^\sg)$ respectively; and we sometimes
also omit superscripts in our notation for elements in~$P^\sg$.

A \emph{homomorphism} from a trilinear pair $P$ into a trilinear pair $P'$ is a pair $\omega = (\omega^-,\omega^+)$
of linear maps such that
$\omega^\sg\{x^\sg,y^\msg,z^\sg\} = \{\omega^\sg x^\sg,\omega^\msg y^\msg,\omega^\sg z^\sg\}$
for $x^\sg,z^\sg\in P^\sg$, $y^\msg \in P^\msg$, $\sg = \pm$.
A homomorphism $\omega$ is  called an \emph{isomorphism}
if each $\omega^\sg$ is bijective.

If $P$ is a trilinear pair and  $Q = (Q^-,Q^+)$, where
 $Q^\sg$ is a submodule of $P^\sg$ for $\sg = \pm$, then
 $Q$ is called a \emph{subpair}, \emph{ideal} or \emph{left ideal} of $P$  if
 $\{Q^\sg,Q^\msg,Q^\sg\} \subseteq Q^\sg$,
 $\{P^\sg,P^\msg,Q^\sg\}+ \{P^\sg,Q^\msg,P^\sg\}+ \{Q^\sg,P^\msg,P^\sg\}\subseteq Q^\sg$
 or ~$\{P^\sg,P^\msg,Q^\sg\}\subseteq Q^\sg$ respectively for $\sg = \pm$.
A trilinear pair $P$ is said to be \emph{simple} if $\{P^\sg,P^\msg,P^\sg\} \ne \{0\}$
for $\sg = +$ or $\sg = -$  and the only ideals of $P$ are $P$
and $\{0\}$.

There are evident notions of \emph{direct sum} and \emph{quotient}   for trilinear pairs.

If $\bR$ is a field we say that $P$ is \emph{finite dimensional}
if each $P^\sg$ is finite dimensional, and we call
$(\dim(P^-) ,\dim(P^+))$ the \emph{dimension}
 of $P$.
If $d = \dim(P^-) =\dim(P^+)$, we say that
$P$ has \emph{balanced dimension} $d$.

The   \emph{centroid} of a trilinear pair $P$ is the subalgebra
$\Centrd(P)$ of the  associative algebra $\End(P^-)\oplus \End(P^+)$
consisting of the pairs of
maps $(\omega^-, \omega^+)\in \End(P^-)\oplus \End(P^+)$ such that
$\omega^\sg(\{x^\sg,y^\msg,z^\sg\})=
\{\omega^\sg(x^\sg),y^\msg,z^\sg\}=
\{x^\sg,\omega^\msg(y^\msg),z^\sg\}=
\{x^\sg,y^\msg,\omega^\sg(z^\sg)\}$.
We say  $P$ is  \emph{central} if the homomorphism
$a \mapsto a (\id_{P^-},\id_{P^+})$ from $\bR$ into $\Centrd(P)$ is an isomorphism.
If $P$ is simple, then $\Centrd(P)$  is a field.

The  \emph{opposite} of a trilinear pair $P = (P^-,P^+)$
is the trilinear pair $P^\op = (P^+,P^-)$ whose $\sg$-product
is the $\msg$-product of $P$ for $\sg = \pm$.

If $\bbF \in \Kalg$ and $P$ is a trilinear  pair,
then the products on $P$ canonically induce $\bbF$-trilinear products  on
$P_\bbF := (P^-_\bbF,P^+_\bbF)$ so that
$P_\bbF$ is a trilinear pair over $\bbF$.

Suppose that $P$ is a trilinear pair, $G$ is an abelian group  (written additively), and  $P^\sg =
\bigoplus_{g\in G}P_g^\sg$ for $\sg = \pm$, where
$P_g^\sg$ is a submodule of $P^\sg$ for $g\in G$, $\sg = \pm$.
We  say that $P =  (\bigoplus_{g\in G}P_g^-, \bigoplus_{g\in G}P_g^+)$
is a  $G$-\emph{grading} of $P$ if
\[\{P_g^\sg,P_k^\msg,P_\ell^\sg\} \subseteq P_{g-k+\ell}^\sg\]
for $g,k,\ell\in G$, $\sg = \pm$.
(Here we follow the terminology in  \cite[\S8.1]{LN}.  This notion of grading
is equivalent to the usual one
if we replace
$P_g^\sg $ by $P_{\sg g}^\sg$ for each $\sg$ and $g$.)    We often then write
\begin{equation} \label{eq:SP} \textstyle
P =  \bigoplus_{g\in G}P_g,
\end{equation}
where $P_g := (P_g^-,P_g^+)$ for $g\in G$.
Note that each $P_g$ is a subpair of $P$; however the sum \eqref{eq:SP}
is not in general a direct sum of trilinear pairs, since the
subpairs $P_g$ need not be ideals.  The $G$-\emph{support}
of $P$ is defined to be $\supp_G(P) = \set{g\in G \suchthat P_g^\sg \ne 0 \text{ for } \sg = + \text{ or } \sg = -}$.

Finally  suppose that $X$ is a \emph{triple system},  by which we mean a module with a trilinear
product $\{\ ,\ ,\ \} : X \times X \times X \to X$.  A \emph{polarization} of $X$ is a module decomposition
$X = X^-\oplus X^+$  such that
$\{X^\sg,X^\msg,X^\sg\} \subseteq X^\sg$,  $\{X^\sg,X^\sg,X\} = 0$ and
$\{X,X^\sg,X^\sg\} = 0$ for $\sg = \pm$; and we say that
$X$ is \emph{non-polarized} if it has no polarizations.

If $X$ is a triple system, the trilinear pair $(X,X)$ with products  defined by
$\{x,y,z\}^\sg = \{x,y,z\}$ (resp.~$\{x,y,z\}^\sg = \sg \{x,y,z\}$) is called the \emph{double} (resp.~the \emph{signed double}) of $X$.
It is easy to check (and well known) that the double (resp.~the signed double) of $X$ is simple if and only if $X$ is simple and   non-polarized.

\begin{remark}
In the rest of the paper, we will  often discuss simplicity and isomorphism
of graded algebras and graded  trilinear pairs.  To be clear,
\emph{the terms simple and isomorphism will be used in the
ungraded sense} as defined above, unless we specify to the contrary.
\end{remark}

\section{Kantor pairs and  $5$-graded ($\BCone$-graded) Lie algebras}
\label{sec:KPs}

\emph{Throughout this section, we assume that}
\[\De = \set{-2\al_1,-\al_1,\al_1,2\al_1}\]
\emph{is the irreducible root system of type $\BCone$ with base  $\Gamma = \set{\al_1}$}.
We identify $Q_\De = \bbZ$ using the $\bbZ$-basis $\Gamma$ for $Q_\De$ (as in \Subsec \ref{subsec:rootgraded}).
Then a $\BCone$-grading of a Lie algebra $L$ is merely
 a \emph{$5$-grading} of $L$.
(Recall  that if $m\ge 1$, a \emph{$2m+1$-grading of $L$}
is a $\bbZ$-grading $L = \bigoplus_{i\in\bbZ} L_i$ with $L_i = 0$ for $\mid i \mid > m$.)

In this section, we  recall the definition of a Kantor pair and how Kantor pairs are
related to $5$-graded Lie algebras.

\subsection{Kantor pairs}
\label{subsec:KPs}
A \emph{Kantor pair} is a trilinear pair $P$ such that the following identities  hold
\begin{gather}
\tag{K1}
[D(x^\sg,y^\msg\!),D(z^\sg,w^\msg\!)]=
D(D(x^\sg,y^\msg\!)z^\sg,w^\msg\!)-
D(z^\sg,D(y^\msg,x^\sg)w^\msg\!),\\
\tag{K2}
K(x^\sg,z^\sg)D(w^\msg,u^\sg)+ D(u^\sg,w^\msg)K(x^\sg,z^\sg) =
K(K(x^\sg,z^\sg)w^\msg,u^\sg)
\end{gather}
for $x^\sg,z^\sg,u^\sg  \in P^\sg$, $y^\msg,w^\msg \in P^\msg$, $\sg = \pm$.

It is clear that the  opposite  of a Kantor pair  is a Kantor pair.

\begin{specialcases}\label{ex:speccase}
(i) A  \emph{Jordan pair} $P$ is a Kantor pair satisfying $K(P^\sg,P^\sg) = 0$
for $\sg = \pm$.
The structure theory of Jordan pairs
is developed in detail in \cite{L},
where Jordan pairs are defined in a different way  using quadratic operators. (However, since
$\frac 1 6\in \bR$, the two definitions are equivalent \cite[Prop.~2.2]{L}.)

(ii)   Suppose  that $X$ is a triple system. Then $X$ is called
a \emph{Kantor triple system} if its double is a Kantor pair.
Kantor triple systems were introduced by Kantor in \cite{K1,K2}, where they were called generalized
Jordan triple systems of the second  order, and where numerous examples can be found.
In the literature,  Kantor triple systems are  also often called
\emph{$(-1,1)$-Freudenthal--Kantor triple systems}.   (See  \cite{YO}, as well as
the recent papers \cite{EO,EKO} and their references, for information about
\emph{$(\epsilon,\delta)$-Freudenthal--Kantor triple systems}, where
$\epsilon, \delta = \pm 1$.)

(iii)
Analogously,    a triple system $X$ with product $\{x,y,z\}$ is called a
$(1,1)$-\emph{Freudenthal-Kantor triple
system} if  its signed double is a Kantor pair.

(iv)  A
 \emph{structurable algebra} is a unital algebra with involution $(A,-)$ such that
$A$ is a Kantor triple system under the product
$\{x,y,z\} = 2((x\bar y) z + (z\bar y)x - (z\bar x)y)$ on $A$.
(Involution here means a period 2 anti-automorphism.  Also, the $2$ in the expression for $\tprod$ is unimportant;  it is included for
compatibility with the Jordan algebra case.)
The double $(A,A)$ of this Kantor triple system is
also called the \emph{double} of the
structurable algebra $A$.
See \cite{A1} and \cite{Sm} for  many examples of structurable algebras, including all
unital Jordan algebras and  all unital alternative algebras with  involution.
\end{specialcases}

\subsection{The Kantor pair enveloped by a 5-graded Lie algebra}
\label{subsec:5grading}
If
$\rL = \bigoplus_{i\in \bbZ} \rL_i$
is a  $5$-graded Lie algebra, then $P = (\rL_{-1},\rL_1)$ is a Kantor pair with products defined by
\begin{equation*}
\label{eq:KPproduct}
\{x^\sg,y^\msg,z^\sg\} = [[x^\sg,y^\msg],z^\sg]
\end{equation*}
for $x^\sg,z^\sg\in \rL_{\sg 1}$, $y^\msg\in \rL_{\msg 1}$, $\sg = \pm$.  (See \cite[Thm.~7]{AF1}, where
$-[[x^\sg,y^\msg],z^\sg]$ is used instead of $[[x^\sg,y^\msg],z^\sg]$.)
We call  $P$ the \emph{Kantor pair  enveloped by the $5$-graded Lie algebra
$\rL$}, and we say that \emph{the $5$-graded Lie algebra $L$  envelops $P$}.
(In a similar situation, Jacobson uses  the term enveloping Lie algebra \cite[\S3.1]{J1}.)

If  $L$ is $3$-graded, we see using the Jacobi identity that the pair $(\rL_{-1},\rL_1)$ enveloped by $L$  is in fact Jordan.

If  $P$ is the Kantor pair  enveloped by a $5$-graded Lie algebra $L$, we let
\[\TLP := P^- \oplus P^+ = L_{-1} \oplus L_1 \quad \text{in $L$.}\]

The following  Lemma (as well as Lemma \ref{lem:TLP} below) is easily checked.

\begin{lemma} \label{lem:LTS}
Suppose that $P$ is the Kantor pair  enveloped by a $5$-graded Lie algebra $L$.
Then $\TLP$ is a triple system under the trilinear product
$[[x,y],z]$.  Moreover,  this triple system depends only on $P$;
specifically, if $x^\tau,y^\tau,z^\tau \in P^\tau$ for $\tau = \pm$, we have
\begin{equation}
\label{eq:LTS}
\begin{aligned}
{}[[ x^\sg,y^\sg],z^\sg] &= 0,
&[[x^\sg,y^\msg],z^\sg] &= \{x^\sg,y^\msg,z^\sg\},\\
[[x^\msg,y^\sg],z^\sg] &= -\{y^\sg,x^\msg,z^\sg\},
&[[x^\sg,y^\sg],z^\msg] &= K(x^\sg,y^\sg)z^\msg.
\end{aligned}
\end{equation}
\end{lemma}

\begin{remark} \label{rem:LTS}
(i) Despite  the conclusion in Lemma \ref{lem:LTS}, we include
the subscript $L$ in the notation for $\TLP$ to emphasize that we are regarding $\TLP$ as a submodule of $L$.

(ii) The triple system $\TLP$ is a Lie triple system.
Moreover, this triple system is sign-graded, which means that it is $\bbZ$-graded with
support contained in $\{-1,1\}$.  This is the point of view
taken in \cite[\S3--4]{AF1} (and in special cases elsewhere), but for
simplicity we won't make use of Lie triple systems in this work.
\end{remark}

\begin{lemma} \label{lem:TLP}
Suppose that  $P$ is the Kantor pair  enveloped by a $5$-graded Lie algebra $L$.
Then the subalgebra $\langle \TLP \rangle_\text{alg}$ of $L$ that is generated
by $T_L(P)$ is a $5$-graded ideal of $L$ that  envelops $P$. Moreover
$\langle \TLP \rangle_\text{alg} = [\TLP,\TLP] \oplus \TLP$
and
$[\TLP,\TLP] = [L_{-1},L_{-1}]  \oplus [L_{-1},L_1] \oplus [L_{-1},L_1]$.
\end{lemma}

\begin{definition} \label{def:tight}  Suppose  that $L$ is a $5$-graded Lie algebra and $P$ is a Kantor pair.  We say that
$L$ \emph{tightly  envelops} $P$ if $L$  envelops $P$,
\begin{equation}\label{eq:tight}
\langle \TLP \rangle_\text{alg} = L \andd Z(L) \cap [\TLP,\TLP] = 0,
\end{equation}
where (here and subsequently) $Z(L)$ denotes the \emph{centre} of the Lie algebra $L$.
\end{definition}

\begin{remark} \label{rem:tight} If $L$ is a $5$-graded Lie algebra that  envelops a Kantor pair $P$,  it follows easily from
Lemma \ref{lem:TLP} that we can replace
$L$ by $L' = \langle \TLP \rangle_\text{alg}$ and then replace $L'$ by
$\overline{L'} = L'/\left(Z(L')\cap [T_{L'}(P),T_{L'}(P)]\right)$  to get a $5$-graded Lie algebra $\overline{L'}$
that tightly  envelops $P$ (with the evident identifications of $P^-$ and $P^+$ in  $\overline{L'}$).
\end{remark}

\begin{remark}\label{rem:WeylBC1}
Suppose that $P$ is the Kantor pair  enveloped by a $5$-graded Lie algebra $L$.
Since  $\Aut(\De) = W_\De = \set{1,-1}$,
we can form the Weyl images
$\p1\rL$ and $\p{-1}\rL$ of~$\rL$.  Clearly the $5$-graded Lie
algebra $\p1\rL = \rL$  envelops $P$, whereas
the $5$-graded Lie algebra $\p{-1}\rL$  envelops the Kantor pair~$P^\op$,
which in general in not isomorphic to $P$.
In \Sec \ref{sec:Weyl}, we will look at this phenomenon for $\BCtwo$-graded Lie algebras, where
the supply of Weyl images is richer.
\end{remark}

\subsection{The Kantor construction}
\label{subsec:KanLA}
To see that any Kantor pair is  enveloped by a  $5$-graded Lie algebra,  we now recall from
\cite[\S3--4]{AF1} the
construction of a $5$-graded Lie algebra $\Kan(P)$ from a Kantor pair  $P$.

Let $P$ be a Kantor pair.  Let $\smat{P^-\\ P^+}$ be the module of
column vectors with entries as indicated, and canonically identify
$\End \smat{P^-\\ P^+}=\smat{\End(P^-)&\Hom(P^+,P^-)\\
\Hom(P^-,P^+)&\End(P^+)}$
so that the action of  $\End \smat{P^-\\ P^+}$ on $\smat{P^-\\ P^+}$ is by matrix multiplication.
Then
\[
\fS(P) := \spann_\bR\{\mat{D(x^-,x^+)&K(y^-,z^-)\\ K(y^+,z^+)&-D(x^+,x^-)}
: x^\sg,y^\sg,z^\sg\in P^\sg,\ \sg = \pm 1
\}
\]
is a subalgebra of the Lie algebra $\End\! \smat{P^-\\ P^+}$ under the commutator product.
Also
\[\Kan(P) := \fS(P) \oplus \mat{P^-\\ P^+},\]
is a Lie algebra under the anti-commutative product $[\, ,\,]$ satisfying:
\begin{gather*}
[A,B] = AB-BA, \quad
[A,\mat{x^-\\x^+}]=A\mat{x^-\\x^+},\\
[\mat{x^-\\x^+},
\mat{y^-\\y^+}]=
\mat{D(x^-,y^+)-D(y^-,x^+)&K(x^-,y^-)\\
K(x^+,y^+)&-D(y^+,x^-)+D(x^+,y^-)}
\end{gather*}
for $A,B \in \fS(P)$, $x^\sg,y^\sg\in P^\sg$, $\sg = \pm$.
We call  $\Kan(P)$ the \emph{Kantor Lie algebra} of  $P$.

The Lie algebra $\Kan(P)$ is $5$-graded with
\begin{equation}
\label{eq:gr5}
\begin{gathered}
\Kan(P)_{-2} =
\mat{0& K(P^-,P^-)\\
0&0},\quad
\Kan(P)_{-1}=  \mat{P^-\\0},\quad\\
\Kan(P)_0 =
\spann_\bR\{\mat{ D(x^-,x^+)&0\\0&-D(x^+,x^-)} : x^-\in P^-,\ x^+\in P^+\},\\
\Kan(P)_1 = \mat{0\\ P^+},\quad
\Kan(P)_2 = \mat{ 0& 0\\ K(P^+ ,P^+)&0}.
\end{gathered}
\end{equation}
We call this $5$-grading the \emph{standard $5$-grading}
of $\Kan(P)$.  \emph{Unless mentioned otherwise we will regard $\Kan(P)$ as a $5$-graded
algebra with its standard $5$-grading.}

If $P$ is a Kantor pair, \emph{we identify $P^-$ with $\Kan(P)_{-1}= \smat{P^-\\0}$
and $P^+$ with $\Kan(P)_1= \smat{0\\P^+}$}  in the evident
fashion. With this identification the following is clear:

\begin{proposition}
\label{prop:Kan5} If $P$ is a Kantor pair, then $\Kan(P)$ with its standard $5$-grading tightly  envelops $P$.
\end{proposition}

We will   see in Corollary \ref{cor:Kanchar} that
$\Kan(P)$ is the unique $5$-graded Lie algebra that tightly envelops the Kantor pair $P$.

\begin{remark} \label{rem:KP} Suppose that $P$ is a Kantor pair.

(i) $P$ is Jordan if and only if $\Kan(P)_{-2} = \Kan(P)_2 = 0$,
in which case  the $3$-graded ($=\type A1$-graded) Lie algebra $\Kan(P) = \Kan(P)_{-1}\oplus \Kan(P)_0 \oplus \Kan(P)_1$
is (graded-isomorphic to) the derived algebra of the \emph{Tits-Kantor-Koecher Lie algebra} of $P$
\cite[\S9.1]{LN}.

(ii) $P$ is finitely spanned  (as a module) if and only if $\Kan(P)$ has the same property.

(iii) Suppose  $P = (X,X)$ is the double of a Kantor triple system $X$.  Then
$\Kan(P)$ is the Lie algebra constructed by Kantor from $X$ in \cite{K1,K2},
and it is easy to check that there is a unique grade-reversing  period 2 automorphism
of $\Kan(P)$ which
maps $\smat{x\\0}$ to    $\smat{0\\x}$ for $x\in X$.
\end{remark}

\begin{lemma}  \label{lem:extendKan}  Let $\bbF \in \Kalg$ and let $P$ be a Kantor pair.   Assume
that either $\bbF$ is a projective $\bR$-module (which holds for example if $\bR$ is a field);
or that $\bbF$ is flat   and
each $P^\sg$ is a finitely generated  module.  Then there is a canonical  $5$-graded $\bbF$-algebra isomorphism from
$\Kan(P)_\bbF$ onto   $\Kan(P_\bbF)$.
\end{lemma}

\begin{proof}  Now $\Kan(P) = \fS(P) \oplus  \smat{P^-\\ P^+}$,
so $\Kan(P)_\bbF = \fS(P)_\bbF \oplus \smat{P_\bbF^-\\ P_\bbF^+}$,
whereas $\Kan(P_\bbF) = \fS(P_\bbF) \oplus \smat{P_\bbF^-\\ P_\bbF^+}$.
Our isomorphism $\omega : \Kan(P)_\bbF \to \Kan(P_\bbF)$, is the direct sum
$\omega'\oplus \omega''$, where $\omega''$ is the identify map and
$\omega'$ is the composition
\begin{equation}
\label{eq:extendKan}
\fS(P)_\bbF \to \End(\mat{P^-\\ P^+})_\bbF \to \End_\bbF (\mat{P_\bbF^-\\ P_\bbF^+}).
\end{equation}
Here the first map in \eqref{eq:extendKan} is induced by inclusion and is injective
since $\bbF$ is flat; whereas the second map in \eqref{eq:extendKan}
is the canonical homomorphism, which is injective because of our assumptions
on $\bbF$ and $P$ (see \cite[II.5.3, Prop.~7]{B1} and
\cite[I.2.10, Prop.~11]{B2}).  It is easy to check that the image of $\omega'$
is in fact $\fS(P_\bbF)$, and that $\omega$ is a graded $\bbF$-algebra homomorphism.
\end{proof}

\subsection{Simplicity and centrality} \label{subsec:simcent}

The following proposition is proved in  \cite[Prop.~2.7(iii)]{GLN}.

\label{subsec:simplicity}
\begin{proposition}  \label{prop:simpleKan} If $P$ is a Kantor pair, then
$P$ is simple if and only if $\Kan(P)$ is  simple.
\end{proposition}

Recall that the \emph{centroid} of
an algebra $\rL$ is the subalgebra $\Centrd(\rL)$ of the associative algebra  $\End(\rL)$
consisting of the endomorphisms of $\rL$ that commute with all
left and right multiplication operators. We say that $\rL$ is \emph{central}
if the homomorphism  $a \mapsto a\id_{\rL}$ from $\bR$ into $\Centrd(L)$ is an isomorphism.
If $\rL$ is simple,  then $\Centrd(\rL)$ is a field.
If $\rL$ is $G$-graded, where $G$ is an abelian group,
then $\Centrd(\rL,G) :=
\set{\chi\in \Centrd(\rL) \suchthat
\chi(\rL_g) \subseteq \rL_g \text{ for } g \in G}$ is a subalgebra of $\Centrd(\rL)$.

\begin{lemma}
\label{lem:centroid}  Suppose that $P$ is a Kantor pair.  Then the restriction map
$\chi \mapsto (\chi \mid_{P^-},\chi \mid_{P^+})$ is an
isomorphism of
$\Centrd(\Kan(P),\bbZ)$ onto the centroid $\Centrd(P)$ of $P$.
\end{lemma}

\begin{proof} All but surjectivity is clear.
For surjectivity, suppose that $\omega = (\omega^-, \omega^+)\in \Centrd(P)$.
Define $\chi : \Kan(P) \to \Kan(P)$ by
$\chi(X) = \smat{\omega^- & 0 \\ 0 & \omega^+}X$.
Then one checks easily that
$\chi\in \Centrd(\Kan(P),\bbZ)$; and clearly $\chi$ restricts to~$\omega$.
\end{proof}

\begin{proposition}
\label{prop:centroid}  Suppose
that $P$ is a simple Kantor pair over~$\bR$. Then
we have $\Centrd(\Kan(P),\bbZ) = \Centrd(\Kan(P))$.  Furthermore, the restriction
map is an isomorphism of $\Centrd(\Kan(P))$ onto
$\Centrd(P)$, so $P$ is central if and only if $\Kan(P)$ is central.
\end{proposition}

\begin{proof}
 In view  of Lemma \ref{lem:centroid}, it is enough to show the first statement.
This follows from  \cite[Lemma 1.6(a)]{Z2} when $\bR$ is a field.  In general, note that
$\Kan(P)$ is simple by Proposition \ref{prop:simpleKan},
so $\Kan(P)$ is a finitely generated module for its multiplication
algebra.  Therefore by \cite[(2.15)]{BN}, $\Centrd(\Kan(P))$ is naturally $\bbZ$-graded with $\Centrd(\Kan(P))_0 = \Centrd(\Kan(P),\bbZ)$. But this grading is trivial since $\Centrd(\Kan(P))$ is a field.
\end{proof}

\begin{corollary}  \label{cor:centroid} If $P$  is a Kantor pair, then
$P$ is central simple if and only if $\Kan(P)$ is  central simple.
\end{corollary}

The next proposition lists facts about Kantor pairs that are
analogues of well-known facts for algebras.
The first two of these
tell us that the study of simple Kantor pairs  over a field $\bR$
is reduced to the study of central simple Kantor pairs over extension fields of $\bR$.

\begin{proposition}  \label{prop:extension}  Suppose that $P$ is a Kantor pair over $\bR$.
\begin{itemize}
  \item [(i)] If $P$ is simple, then  $P$
is a central simple Kantor pair over the field $\Centrd(P)$.
  \item [(ii)] If  $\bR$ is a field  and $P$ is a central simple Kantor pair over a field $\bbF$ containing $\bR$, then $P$ is a simple Kantor pair over $\bR$ with centroid $\bbF$.
  \item  [(iii)] If  $\bR$ is a field, then $P$ is a central simple Kantor pair over $\bR$ if and only if
$P_\bbF$ is simple over $\bbF$ for all fields $\bbF$ containing $\bR$.
\end{itemize}
\end{proposition}

\begin{proof}  Using Propositions
\ref{prop:simpleKan}, Corollary \ref{cor:centroid} and Lemma \ref{lem:extendKan}, all these statements follow from the corresponding statements for Lie algebras.  The statement corresponding to (i)
is \cite[Thm.~II.1.6.3(2)]{Mc}; the statement corresponding to (ii) follows from the second part of
\cite[Thm.~X.3]{J3} (with $\Gamma = \Delta = \bbF$ and $\Phi = \bR$);
and the statement corresponding to (iii) is \cite[Thm.~II.1.6.3(2)]{Mc}.
\end{proof}

\subsection{$5$-graded Lie algebras enveloping a Kantor pair.}
\label{subsec:LdetP}
In this   subsection,   we use as usual the standard $5$-grading on each  Kantor Lie algebra.

\begin{lemma}  \label{lem:LdetP} Suppose that $P$ and $P'$ are Kantor pairs and $L$ is a $5$-graded Lie algebra
that tightly  envelops $P$.  Let $\chi : P \to P'$ be an surjective homomorphism of Kantor pairs.
Then there exists a unique $5$-graded algebra homomorphism $\varphi: L \to \Kan(P')$ that extends
$\tilde \chi := \chi^-\oplus \chi^+: \TLP \to T_{\Kan(P')}(P')$. Furthermore, $\varphi$ is surjective and
\begin{equation}\label{eq:induceKer}
\ker (\varphi)=\Big\{d\in [\TLP,\TLP]:[d,\TLP]\subseteq \ker(\tilde \chi)\Big\}+\ker(\tilde \chi).
\end{equation}

\end{lemma}

\begin{proof}  Let $T = T_L(P)$. By assumption,  $L = \langle T \rangle_\text{alg}$; so uniqueness in the lemma is clear and, by Lemma \ref{lem:TLP},
we have  $L = [T,T]\oplus T$,
 Next, let $L' = \Kan(P')$, $T' = T_{L'}(P')$.
One sees using Lemma \ref{lem:LTS} that
$\tilde \chi([[x,y],z]) = [[\tilde \chi(x),\tilde \chi(y)],\tilde \chi(z)]$ for $x,y,z\in T$.

Since $L=[T,T]\oplus T$ and $\varphi|_T=\tilde \chi$, we only need to define $\varphi$ on $[T,T]$. So, we consider $\sum_i [x_i,y_i]\in [T,T]$ for $x_i,y_i\in T$ and set
\begin{equation}\label{eq:induce} \textstyle
\varphi(\sum_i [x_i,y_i])=\sum_i [\tilde \chi(x_i),\tilde \chi(y_i)].
\end{equation}
If $\sum_i [x_i,y_i]=0$, then
$[\sum_i [\tilde \chi(x_i),\tilde \chi(y_i)],\tilde \chi(z)]=\tilde \chi([\sum_i[x_i,y_i],z])=0$ for any $z\in T$.
It follows that $\sum_i [\tilde \chi(x_i),\tilde \chi(y_i)]\in Z(L')\cap [T',T']=0$ since $\tilde \chi$ is surjective and $T'$ generates $L'$. Thus $\varphi$ is well-defined.

It is clear that $\varphi$ is $\bbZ$-graded, and  one checks directly that $\varphi$ is a homomorphism of Lie algebras and that (\ref{eq:induceKer}) holds.
\end{proof}

Applying Lemma \ref{lem:LdetP} with  $P' = P$ and $\chi = (\id_{P^-},\id_{P^+})$, we obtain the following corollary:

\begin{corollary}
\label{cor:Kanchar}
If  $L$ is a $5$-graded Lie algebra that tightly  envelops a Kantor pair $P$
then there exists a unique $5$-graded algebra  isomorphism $\varphi : L \to \Kan(P)$
that restricts to the identity map on $\TLP$.
\end{corollary}

Also, applying Lemma \ref{lem:LdetP} with  $L= \Kan(P)$,  we obtain:

\begin{corollary}\label{cor:funcKP} Suppose that
$P$ and $P'$ are Kantor pairs.  If $\varphi: \Kan(P) \to \Kan(P')$ is a $5$-graded isomorphism, then
$(\varphi|_{P^-}, \varphi|_{P^+})$ is an isomorphism of $P$ onto $P'$.  Conversely, if
$\chi = (\chi^-,\chi^+)$ is an isomorphism of $P$ onto $P'$,
there exists
a unique  $5$-graded isomorphism $\varphi: \Kan(P) \to \Kan(P')$ that extends
$\chi^-\oplus \chi^+: \TLP \to T_{\Kan(P')}(P')$.
\end{corollary}
\begin{proof}  The first statement is clear.  The converse follows from
Corollary \ref{prop:Kan5} and  Lemma \ref{lem:LdetP}
(with $L = \Kan(P)$).
\end{proof}

\begin{proposition}  \label{prop:simpletight}
Suppose $P$  is a nonzero Kantor pair and  $L$ is a simple $5$-graded Lie algebra that  envelops
$P$.  Then $L$ tightly envelops $P$.
\end{proposition}

\begin{proof}  By Lemma \ref{lem:TLP},
$\langle \TLP \rangle_\text{alg}$ is
an ideal of $L$, which is nonzero since $P\ne 0$.  Also $Z(L)\cap [\TLP,\TLP]$ is a ideal of $L$ which is proper  since $P\ne~0$.
\end{proof}

The  next   theorem will be among the basic tools in our study of simple
Kantor pairs in this paper and in \cite{AF2} and \cite{AS}.  It tells   us in particular that
each simple Kantor pair is  enveloped by a unique
simple $5$-graded Lie algebra.

\begin{theorem}
\label{thm:LdetP}
Suppose
that $P$ is a nonzero Kantor pair.  Then
\begin{itemize}
\item[(i)]
The following statements are equivalent:
\begin{itemize}
  \item[(a)]  $P$ is simple
  \item[(b)]  There exists a simple $5$-graded Lie algebra $L$ that  envelops $P$.
\end{itemize}

\item[(ii)]  If $L$ is  simple $5$-graded Lie algebra that  envelops $P$,
then there exists a unique $5$-graded isomorphism of $L$ onto $\Kan(P)$ that extends the identity on~$\TLP$.
\item[(iii)] Statements (a) and (b)  in (i) with simple replaced by central simple are equivalent.
\end{itemize}
\end{theorem}

\begin{proof}  If (a) holds in (i), we know  from
Propositions \ref{prop:Kan5}  and \ref{prop:simpleKan}
that $\Kan(P)$ is a simple $5$-graded Lie algebra that  envelops $P$.
Conversely, suppose (b) holds. Then by Proposition \ref{prop:simpletight} and Corollary \ref{cor:Kanchar} we have a unique $5$-graded isomorphism as indicated in (ii).  So $P$ is simple by Proposition \ref{prop:simpleKan}.  This proves (i) and (ii); and
(iii) is proved by the same argument using
Corollary   \ref{cor:centroid}.
\end{proof}

\begin{remark} \label{rem:fdclass} Suppose that $\bR$ is an algebraically closed field
of characteristic~0.  Here we sketch an argument due to Kantor \cite{K1} for the classification
 of finite dimensional simple Kantor pairs in terms of weighted Dynkin diagrams. (Kantor worked with Kantor triple systems.)
We fix a finite dimensional simple Lie algebra $\cG$
of type $\type{X}{n}$ with Cartan subalgebra $\cH$, root system $\Sigma$, root spaces $\cG_\alpha$ for $\alpha \in Q_\Sigma$, base $\Pi = \set{\mu_1,\dots,\mu_n}$ for $\Sigma$, and highest root $\mu^+$.  We view $\Pi$ as usual as a Dynkin diagram.  If ${\mathbf p} =(p_1,\dots,p_n)$ is an $n$-tuple of non-negative integers, we call $(\Pi,{\mathbf p})$ a \emph{weighted Dynkin diagram} and we define
$\chi_{{\mathbf p}} : Q_\Sigma \to \bbZ$ by
$\chi_{{\mathbf p}}(\sum_{j = 1}^n k_j \mu_j) = \sum_{j = 1}^n k_j p_j$.

If $(\Pi,\mathbf p)$ is a weighted Dynkin diagram, then $\cG = \oplus_{i\in \bbZ} \cG_i$ is a $\bbZ$-grading, where $\cG_i$ is the sum of all $\cG_\mu$ for $\mu \in Q_\Sigma$ with
$\chi_{{\mathbf p}}(\mu) = i$.   Moreover it is well-known that any $\bbZ$-graded
finite dimensional simple Lie algebra of type $\type{X}{n}$ is graded isomorphic to one obtained in this way (see \cite[Prop.~12]{K1}, \cite[Thm.~1]{D} or \cite[\S~3.3.5]{GOV}).  Clearly this grading
$\cG = \oplus_{i\in \bbZ} \cG_i$ is a $5$-grading with $(\cG_{-1},\cG_{1}) \ne 0$ if and only~if
\begin{equation} \label{eq:pcond}
\chi_{\mathbf p}(\mu^+) \le 2 \andd p_j = 1 \text{ for some $j$}.
\end{equation}
We see, using Theorem \ref{thm:LdetP}(i) and (ii), that if
if \eqref{eq:pcond} holds then $(\cG_{-1},\cG_{1})$ is a simple Kantor pair whose Kantor Lie algebra has type $\type{X}{n}$; and conversely any such Kantor pair is isomorphic to one that arises in this way for some ${\mathbf p}$ satisfying \eqref{eq:pcond}.  We will discuss two examples for type $\type{E}{6}$ in Remark \ref{rem:E6example}(i).

In \cite{AF2}, we will extend the method just described to  study other types of root gradings of $\cG$ (in particular $\BCtwo$-gradings), and use this to obtain a Dynkin diagram interpretation of  reflection.
\end{remark}

\subsection{Split Lie algebras of type $\type Xn$ and their forms}
\label{subsec:forms}
\emph{In this subsection, we do not assume  that $\bR$ contains~$\frac 16$.
Let $X_{n}$ be one of the
types of an irreducible reduced root system.
}

\begin{lemma}
\label{lem:form}Suppose  $L$ is a Lie algebra that is finitely generated as a module,
and suppose $\bbF\in \Kalg$ is flat.
 Then $Z(L)_\bbF \simeq Z(L_\bbF)$ and
$(L/Z(L))_\bbF\simeq L_\bbF/Z(L_\bbF)$
 as Lie algebras over $\bbF$.
\end{lemma}

\begin{proof}
Let $x_1,\ldots,x_{\ell}$ generate $L$ as a module.
The sequence
$0\rightarrow Z(L)\overset{\iota}\rightarrow L\overset{\eta
}\rightarrow L^{\ell}$
is exact, where $\iota$ is the inclusion map and $\eta(y)=([y,x_1],\ldots,[y,x_{\ell}])$.  So
$0\rightarrow Z(L)_\bbF\overset{\iota_\bbF}
\rightarrow L_\bbF \overset{\eta_\bbF}\rightarrow L_\bbF^{\ell}$
is exact.
Thus $\iota_\bbF$ is an isomorphism of $Z(L)_\bbF$ onto
$Z(L_\bbF)$ (as $\bbF$-modules and hence as $\bbF$-algebras).
Moreover, the sequence
$0 \to Z(L) \overset{\iota} \rightarrow L \overset{\pi} \rightarrow L/Z(L)\to 0$ is exact and thus so is
$0 \to Z(L)_\bbF  \overset{\iota_\bbF} \rightarrow L_\bbF
\overset{\pi_\bbF} \rightarrow(L/Z(L))_\bbF \to 0$.
Hence $(L/Z(L))_\bbF\simeq L_\bbF/\iota_\bbF(Z(L)_\bbF) = L_\bbF/Z(L_\bbF)$ as $\bbF$-algebras.
\end{proof}

\begin{definition}
Let $\fg(\mathbb{C)}$ be
a finite dimensional simple Lie algebra of type $X_{n}$ over the complex field
$\bbC $, and choose a Chevalley basis $B$ for $\fg(\mathbb{C)}$
\cite[\S 25]{H}. Then the $\bbZ$-span $\mathfrak{g(}\bbZ)$ of $B$
is a Lie algebra over $\bbZ$ which depends up to isomorphism only on the
type $X_{n}$.  A Lie algebra isomorphic to $\fg(\bR):=\fg(\mathbb{Z)}_{\bR}$
is called the \textit{Chevalley algebra of
type} $X_{n}$ over $\bR$, and a Lie algebra isomorphic to the
quotient algebra $\fg(\mathbb{K)}/Z(\fg(\mathbb{K)})$ is called the
\textit{split Lie algebra of type} $X_{n}$ over $\bR$.
\end{definition}

\begin{remark}
If $\bR$ is a field of characteristic $\ne 2$ or $3$,  the (finite dimensional) split Lie algebra of type $\type{X}{n}$ is   central simple and studied in detail in
 \cite{Sel}, where it is called the \emph{classical simple Lie algebra}
of type $\type{X}{n}$. Furthermore, if $\bR$ is a field of characteristic $0$, the split Lie algebra of type $\type{X}{n}$ is the \emph{split simple Lie algebra of type  $\type{X}{n}$} defined and studied for example in  \cite[Chap.~IV]{J3}.
\end{remark}

\begin{definition} Suppose that $L$ is a Lie algebra.  We say that $L$ is a \emph{form
of the Chevalley algebra of type $\type{X}{n}$} (resp.~ a \emph{form of the
split Lie algebra of type $\type{X}{n}$)}
if for some faithfully flat $\bbF\in \Kalg$, $L_\bbF$ is the
Chevalley algebra of type $\type{X}{n}$ (resp.~the split Lie algebra of type $\type{X}{n}$) over $\bbF$.
\end{definition}

\begin{remark} \label{rem:chevalleyform} If $L$ is a form of the Chevalley algebra of type $\type{X}{n}$,
then by Lemma \ref{lem:form}, $L/Z(L)$ is a form of the split simple Lie algebra of type $\type{X}{n}$.
\end{remark}

\begin{remark} \label{rem:chevalleyF}  (i)
 If $L$ is the Chevalley algebra of type $\Xn$ and $\bbF\in \Kalg$, then $L_\bbF$ is the Chevalley algebra of type $\Xn$ over $\bbF$.
Hence, by Lemma \ref{lem:form}, if $\bbF$ is flat and $L$ is the split
Lie algebra of type $\Xn$, then $L_\bbF$ is the split
Lie algebra of type $\Xn$ over $\bbF$.

(ii) If $\bR$ is a field and $L$ is a form
of the Chevalley algebra of type $\type{X}{n}$ (resp.~ a form of the
split Lie algebra of type $\type{X}{n}$), then it is easy to see using (i) that  $L_\bbF$ is the
Chevalley algebra of type $\type{X}{n}$ (resp.~the split Lie algebra of type $\type{X}{n}$) over $\bbF$
\emph{for some field $\bbF$ containing $\bR$}.
\end{remark}

\subsection{Split Kantor pairs of type $\type Xn$ and their forms}
\label{subsec:splitKP}
\emph{We return to our assumption that $\frac 16 \in \bR$.}

\begin{definition}   Suppose $P$ is a Kantor pair.  We say $P$
is a \emph{split Kantor pair of type $\Xn$} if $\Kan(P)$  is the
split Lie algebra of type $\Xn$; and we say that $P$ is a
\emph{form of a split Kantor pair of type $\Xn$} if for some
faithfully flat $\bbF\in \Kalg$, $P_\bbF$ is a split Kantor pair of type $\Xn$ over $\bbF$.
\end{definition}

The reader should keep in mind that, unlike the situation for Lie algebras, there can be non-isomorphic split Kantor pairs of type $\Xn$. This is already true for Jordan pairs; and the same    type can even encompass Jordan as well as non-Jordan Kantor pairs (see Example~\ref{ex:Dnfd}).

\begin{lemma}  \label{lem:splitform} Suppose $P$ is a Kantor pair.
Then $P$ is a form of a split Kantor pair of type $\Xn$ if and only if
$\Kan(P)$ is a form of the split Lie algebra of type $\Xn$.
\end{lemma}

\begin{proof} Let $\bbF \in \Kalg$ be faithfully flat and let $L$
be the split Lie algebra of type $\Xn$ over $\bbF$.    It is sufficient to show that
$\Kan(P_\bbF) \simeq L$  if and only if $\Kan(P)_\bbF \simeq L$.

If  $\Kan(P_\bbF) \simeq L$, then $\Kan(P_\bbF)$ is a finitely generated $\bbF$-module
and hence so is each $P^\sg_\bbF$ (see Remark \ref{rem:KP}(ii)).  Thus each $P^\sg$ is a finitely generated $\bR$-module
\cite[I.3.6, Prop.~11]{B2}, so $\Kan(P_\bbF) \simeq \Kan(P)_\bbF$ by Lemma \ref{lem:extendKan}.

Conversely, if  $\Kan(P)_\bbF \simeq L$, then $\Kan(P)_\bbF$ is a finitely generated $\bbF$-module.
Hence $\Kan(P)$ is a finitely generated $\bR$-module \cite[I.3.6, Prop.~11]{B2}, and therefore so is each $P^\sg$.
Again $\Kan(P_\bbF) \simeq \Kan(P)_\bbF$ by Lemma \ref{lem:extendKan}.
\end{proof}

\begin{remark} \label{rem:KPF} In view of Remark \ref{rem:chevalleyF} and Lemma \ref{lem:splitform}, we see:

(i)  If $P$ is a split Kantor pair of type $\Xn$ and $\bbF\in \Kalg$, then $P_\bbF$ is a split Kantor pair of type $\Xn$ over $\bbF$.

(ii) If $\bR$ is a field and $P$ is a form
of a Kantor pair type $\type{X}{n}$, then   $P_\bbF$ is a
split Kantor pair algebra of type $\type{X}{n}$ over $\bbF$ \emph{for some field $\bbF$ containing $\bR$}.
\end{remark}

It turns out that forms of split Kantor pairs of type $\type D4$, $\type E6$, $\type E7$, $\type E8$, $\type G2$ and $\type F4$
make up one of the four classes of central simple Kantor pairs that appear in the structure theorem in \cite{AS} mentioned in the introduction.

\subsection{The Jordan obstruction and the $2$-dimension of a Kantor pair}  \label{subsec:obstruct}  Let $P$ be a Kantor pair.

Note first that $\Kan(P)_{-2} \oplus \Kan(P)_0 \oplus \Kan(P)_{2}$ is a $3$-graded Lie algebra (with the grading scaled in the obvious way),  so
\[J(P) :=(\Kan(P)_{-2},\Kan(P)_{2})\]
is a Jordan pair with products
$[[x,y],z]$ calculated in $\Kan(P)$.  We call $J(P)$
the \emph{Jordan obstruction} of $P$.
We use this term because
\emph{$J(P)$ is trivial if and only if $P$ is Jordan} (see Remark \ref{rem:KP}(i)).

If $L$ is a $5$-graded Lie algebra that tightly envelops
$P$, then by Corollary \ref{cor:Kanchar},
\begin{equation} \label{eq:JPL}
J(P) \simeq (L_{-2},L_2)
\end{equation}
with the products
$[[x,y],z]$ calculated in $L$.

Suppose  that $\bR$  is a field.  If
$J(P)$ is finite dimensional, we call its dimension the $2$-\emph{dimension} of $P$.
Further, if $J(P)$ has balanced dimension $k$, we say that $P$ has
\emph{balanced $2$-dimension} $k$.  So
$2$-dimension and balanced $2$-dimension (when they are defined) can be
viewed as numerical measures of the  distance of the pair $P$ from Jordan theory.
In particular, Kantor pairs of balanced $2$-dimension $1$
can be thought of as being \emph{close to Jordan};
and they are, after Jordan pairs, the most studied and best understood Kantor pairs in the literature (see the following example).

\begin{example} \label{ex:symplectic}   Suppose  $\bR$ is a field.
A \emph{symplectic} triple system is defined as a triple system $\cT$ (with product $[x,y,z]$) together with a non-zero
skew-symmetric bilinear form $(.\mid.)$ on $\cT$ satisfying a list of axioms \cite[\S 4]{E2} \cite[\S 6.4]{EK}, \cite[\S 2]{YA}.    These structures
are a variation on Freudenthal triple systems, which have been studied by many authors (see for example \cite{M,FF,KS}) going back to the work of Freudenthal on exceptional Lie algebras.  Indeed given a symplectic triple $\cT$, the product
  $xyz = [x,y,z] - (x\vert z)y - (y\vert z) x$ endows $\cT$
with the structure of a \emph{Freudenthal triple system}, and this process can be reversed \cite[Thm.~4.7]{E2}.

Given a symplectic triple system $\cT$, one can construct a $5$-graded Lie algebra $\frak g(\cT)$ with $\dim(\frak g(\cT)_{\pm 2}) = 1$ \cite[\S 4]{E2} \cite[\S 6.4]{EK}, \cite[\S 2]{YA}.  It turns out that
the Kantor pair $\cP(\cT)$ enveloped by $\frak g(\cT)$ is the signed double of a $(1,1)$-Freudenthal-Kantor triple system (see \ref{ex:speccase}(iii))
with product $\{x,y,z\} = [x,y,z] - (x \vert y)z$.  Moreover, one easily checks that $\frak g(\cT)$ tightly envelops $\cP(\cT)$
and hence, using Corollary \ref{cor:Kanchar}, $\cP(\cT)$ is close to Jordan in the above sense.
\end{example}

\section{Short Peirce gradings and  \protect $\BCtwo$-gradings \protect}
\label{sec:SP}

\emph{In  this section, we assume that $\Delta$ is a irreducible root system of type $\BCtwo$.}  We realize $\Delta$ in the standard way
as
$\Delta = \set{ \ell_1 \ep_1 + \ell_2 \ep_2 \suchthat \ell_i\in \bbZ,\ \lvert  \ell_1 \rvert  + \lvert \ell_2 \rvert \le 2}\setminus\set{0}$, where
$\set{\ep_1,\ep_2}$ is an orthonormal basis in a 2-dimensional real Euclidean space $E_\Delta$.  (See \cite[VI.4.14]{B3}, although there is an evident typo in the last line there.)  We \emph{fix the base}
$\Gamma = \set{\al_1,\al_2}$ for $\Delta$, where
 $\al_1 = \ep_1$, $\al_2 =-\ep_1+ \ep_2$, in which case
 \begin{equation}
\label{eq:BC2}
\Delta = \{\pm \al_1, \pm \al_2, \pm(\al_1+\al_2), \pm 2\al_1, \pm(2\al_1+\al_2), \pm(2\al_1+2\al_2).\}
\end{equation}
We identify $Q_\De = \bbZ^2$  using the
$\bbZ$-basis $\Gamma$ for $Q_\De$, so that
 $(\ell_1,\ell_2) = \ell_1\al_1 + \ell_2\al_2$ for $\ell_i\in \bbZ$.
(\emph{Caution}:  We are not using the $\bbZ$-basis $\set{\ep_1,\ep_2}$ of $Q_\Delta$ for this identification.) So, as in \Subsec \ref{subsec:rootgraded},  any $\BCtwo$-graded Lie algebra
is a $\bbZ^2$-graded Lie algebra.

In this section, we introduce Kantor pairs with
short Peirce gradings and see how they are related to $\BCtwo$-graded Lie algebras.

\subsection{Short Peirce gradings}
\label{subsec:SPgradings}

\begin{definition} \label{def:SP} A \emph{short Peirce grading} (or \emph{SP-grading})
of a Kantor pair $P$ is a $\bbZ$-grading
$P =  \bigoplus_{i\in\bbZ} P_i$
of $P$ such that $\supp_\bbZ(P) \subseteq \set{0,1}$. In that case we have
$P = P_0 \oplus P_1$.  A Kantor pair $P$ together with a short Peirce grading of $P$
is called a \emph{short Peirce graded} (or \emph{SP-graded}) Kantor pair.
\end{definition}

Note that if $P$ is an SP-graded Kantor pair, then
\begin{equation}
\label{eq:ini}
\{P_i^\sg,P_{1-i}^\msg,P_i^\sg\} = 0.
\end{equation}
for $i= 0,1$ and $\sg = \pm$.

Every Kantor pair $P$ has at least two SP-gradings, the \emph{zero SP-grading}
$P = P_0$ with $P_1 = 0$, and the
\emph{one SP-grading} $P = P_1$ with $P_0 = 0$. We call these two SP-gradings \emph{trivial}.

The following example  explains our use of the term short Peirce grading.

\begin{example}[The Jordan case] \label{ex:JordSP} Suppose that $P$ is a Jordan pair.
Neher in \cite{N1}  has defined a \emph{Peirce grading} of $P$ to be a
$\bbZ$-grading  $P =  \bigoplus_{i\in\bbZ}P_i$
of $P$ such that $\supp_\bbZ(P) \subseteq \set{0,1,2}$ and
\begin{equation} \label{eq:PeirceB}
\{ P_2^\sg,  P_0^\msg,  P_0^\sg \} = \{ P_0^\sg,  P_2^\msg,  P_2^\sg \} = 0
\end{equation}
for $\sg = \pm$. (See   also \cite[\S8]{LN}), where Peirce gradings are used in the study of groups associated
with Jordan  pairs.)
Clearly SP-gradings of $P$ are precisely the same as Peirce gradings of $P$ satisfying $P_2 = 0$.
The motivating example in \cite{N1} of a Peirce grading  is the Peirce
decomposition $P = P_0 \oplus P_1 \oplus P_2$ relative to an idempotent $c$ in
$P$ \cite[Thm.~5.4]{L}. This Peirce grading is never short (if $c\ne 0$) since $c\in P_2$.
However, a very important
case in the structure theory of Jordan pairs occurs when $P_0 = 0$ \cite[Thm. 8.2]{L}.
In that case, one obtains an SP-grading
$P = \tilde P_0 \oplus \tilde P_1$, where $\tilde P_i = P_{2-i}$ for $i\in \bbZ$.
\end{example}

\begin{remark}  \label{rem:oppshift}  Let  $P$ be an SP-graded Kantor pair.  There
are two simple procedures for modifying $P$ to obtain another SP-graded Kantor pair.

(i)  Let $P^\op = (P^+,P^-)$ be the opposite pair of $P$
and set $(P^\op)_i^\sg = P_i^\msg$ for $\sg = \pm$, $i\in\bbZ$.
Then  $P^\op$ is an SP-graded Kantor pair called the \emph{opposite} of $P$ (as an SP-graded Kantor pair).

(ii)  Let $\bar P = P$ as a Kantor pair and
 define $\bar P_i = P_{1-i}$ for $i\in \bbZ$.
Then  $\bar P$ is an SP-graded Kantor pair called the \emph{shift} of $P$.
(We use this terminology since, if we view degrees modulo 2, we have shifted
the SP-grading of $P$ by $1$ to obtain~$\bar P$.)  Note that if $P$
has the zero SP-grading then $\bar P$ has the one SP-grading  and vice-versa.
\end{remark}

We will see in \Subsec  \ref{subsec:Weyl} that the SP-graded Kantor pairs
$P^\op$ and $\bar P$ are examples of   Weyl images of $P$.

\subsection{Component gradings}
\label{subsec:component}
If  $\rL=\bigoplus_{(i,j)\in \bbZ^2}\rL_{(i,j)}$ is a $\bbZ^2$-graded module, we often write
$\rL_{(i,j)}$ as $\rL_{i,j}$ for brevity.
Then the \textit{first component grading} of $\rL$ is the $\bbZ$-grading
\[\rL=\textstyle \bigoplus_{i\in \bbZ}\rL_{i,*}, \qquad
\text{ where } \rL_{i,*}=\bigoplus_{j\in \bbZ}\rL_{i,j}.\]
Similarly, we have the \textit{second component} grading $\rL=\bigoplus_{j\in \bbZ}\rL_{*,j}$
with $\rL_{*,j}=\bigoplus_{i\in \bbZ}\rL_{i,j}$.  Of course, if $\rL=\bigoplus_{(i,j)\in \bbZ^2}\rL_{i,j}$
is an algebra grading, then so are its component  gradings.

\subsection{The SP-graded Kantor pair enveloped by a \protect $\BCtwo$-graded \protect Lie algebra}
\label{subsec:BC2}
Suppose   that $\rL$ is a $\BCtwo$-graded Lie algebra.
Then, by \eqref{eq:BC2},  the first component grading
$\rL=\bigoplus_{i\in \bbZ}\rL_{i,*}$ of $\rL$ is a $5$-grading, and we have
\begin{equation*} \label{eq:BC2BC1}
\rL_{-1,*} = \rL_{-1,0} \oplus \rL_{- 1, - 1} \andd \rL_{ 1,*} = \rL_{ 1,0} \oplus \rL_{  1,  1}.
\end{equation*}
Let $P$ be the Kantor pair  enveloped by  $L$ with this  $5$-grading.
So
\[P = (\rL_{-1,*},\rL_{1,*}) =  ( \rL_{-1,0} \oplus \rL_{- 1, - 1} , \rL_{ 1,0} \oplus \rL_{  1,  1} )\]
with products $\{x^\sg,y^\msg,z^\sg\} = [[x^\sg,y^\msg],z^\sg]$.
For $i\in \bbZ$, let $P_i = (P_i^-,P_i^+)$,   where
\begin{equation}
\label{eq:grBC2a}
P_i^\sg = L_{\sg 1,\sg i}
\end{equation}
for $\sg = \pm$.   Then
$P = P_0 \oplus P_1$ is an  SP-grading of $P$ since
\[\{P_i^\sg,P_j^\msg,P_k^\sg\} = [[\rL_{\sg1,\sg i},\rL_{\msg 1,\msg j}],
\rL_{\sg 1,\sg k}] \subseteq L_{\sg 1,\sg (i-j+k)} = P_{i-j+k}^\sg\]
for $\sg = \pm$, $i,j,k\in\bbZ$. We call $P$ together with this
SP-grading the  \emph{SP-graded Kantor pair  enveloped by the $\BCtwo$-graded Lie algebra
$\rL$}, and we say that \emph{the $\BCtwo$-graded Lie algebra $L$  envelops the SP-graded Kantor pair $P$}.
If in addition \eqref{eq:tight} holds, we say that the $\BCtwo$-graded Lie algebra $L$ \emph{tightly envelops} $P$.

\begin{remark}  \label{rem:BC2vs5gr}  It  follows that every
$\BCtwo$-graded Lie algebra enveloping an SP-graded Kantor pair $P$ is also a 5-graded Lie algebra,  using the first component grading,
enveloping the pair $P$ when considered without its SP-grading. The former is tight if and only if the latter is  tight.
\end{remark}

\subsection{The standard $\BCtwo$-grading of $\Kan(P)$}
\label{subsec:KanBC2}
We now  see that any SP-graded Kantor pair is enveloped by some $\BCtwo$-graded Lie algebra.

\begin{proposition}  \label{prop:KanBC2} Suppose that $P$ is an SP-graded Kantor pair.  Then
there exists a unique $\BCtwo$-grading of $\Kan(P)$, which we call the \emph{standard $\BCtwo$-grading
of} $\Kan(P)$,
such that
\begin{equation}
\label{eq:grBC2b}
\Kan(P)_{\sg 1,\sg i} = P_i^\sg
\end{equation}
for $\sg = \pm$, $i\in\bbZ$.  Moreover,   for $\sg = \pm$,
\begin{equation} \label{eq:KPBC2}
\begin{gathered}
\Kan(P)_{\sg 2,\sg 2i} = [P^\sg_i,  P^\sg_i] \text{ for } i=0,1, \quad  \Kan(P)_{\sg 2,\sg 1} = [P^\sg_0,  P^\sg_1],\\
\textstyle \Kan(P)_{0,\sg 1} =  [P_1^\sg,P_0^\msg]
\andd \Kan(P)_{0,0} = \sum_{j=0,1} [P_j^\sg,P_j^\msg].
\end{gathered}
\end{equation}
Finally,  the first component grading of the standard $\BCtwo$-grading is the standard $5$-grading of $\Kan(P)$, and $\Kan(P)$ with the standard $\BCtwo$-grading
tightly  envelops  $P$.
\end{proposition}

\begin{proof} Let
$\Kan = \Kan(P)$ and $T = T_\Kan(P)$. Since $\Kan$
is generated by $T$ (by  Proposition \ref{prop:Kan5}),
uniqueness in the first statement is clear.
For existence, define a $\bbZ^2$-grading  of the module $T$ by setting
\[T_{i,j} =
\begin{cases}
	P_{\sg j}^\sg&\text{if } i=\sg 1 \text{ with } \sg = \pm,\\
	0&\text{otherwise}.
\end{cases}\]
Then one checks directly using \eqref{eq:LTS} that  the trilinear product
$[[x,y],z]$ on $T$ is $\bbZ^2$-graded.  Also, since the $\bbZ^2$-grading of $T$ has finite support,
it induces a natural $\bbZ^2$-grading of the Lie algebra $\End(T)$ under the  commutator product.
Further, since $\fS(P) = [T,T]$, we see that $\fS(P)$ is a $\bbZ^2$-graded
subalgebra of this Lie algebra.
Next, we give  $\Kan = \fS(P) \oplus T$ the direct sum $\bbZ^2$-grading.  It then follows that
$\Kan$ is a $\bbZ^2$-graded Lie algebra.
But, since $\Kan = [T,T]\oplus T$, we have
\begin{equation} \label{eq:suppKan}
\supp_{\bbZ^2}(\Kan) \ \subseteq  (\supp_{\bbZ^2}(T) + \supp_{\bbZ^2}(T)) \cup \supp_{\bbZ^2}(T).
\end{equation}
So, since  $\supp_{\bbZ^2}(T) \subseteq \set{\pm(1,0), \pm (1,1)}$, we have
$\supp_{\bbZ^2}(\Kan) \subseteq \De$. Thus $\Kan$ is
$\De$-graded.  Moreover, the union in \eqref{eq:suppKan} is disjoint, so
we obtain \eqref{eq:grBC2b}.
The second statement follows  immediately from this proof, and the last statement  is clear using  Proposition \ref{prop:Kan5} and Remark   \ref{rem:BC2vs5gr}.
\end{proof}

If   $P$ is an SP-graded Kantor pair, then,  unless mentioned to the contrary, \emph{we will regard  $\Kan(P)$ as $\BCtwo$-graded with its standard $\BCtwo$-grading}.

Using Remark \ref{rem:BC2vs5gr}, it is easy to write down $\BCtwo$-graded versions of many of the $5$-graded results shown in Section \ref{sec:KPs}.  We content ourselves now with recording the results of this type that we will need in this paper
or in \cite{AF2}.

\begin{proposition}  \label{prop:5grvsBC2} Let
$P$ and $P'$ be SP-graded Kantor pairs.
\begin{itemize}
\item[(i)] $P$ and $P'$ are SP-graded isomorphic if and only if $\Kan(P)$ and
$\Kan(P')$ are $\BCtwo$-graded isomorphic.
  \item[(ii)]  If $L$ is a $\BCtwo$-graded Lie algebra that tightly envelops $P$,
then there exists a unique $\BCtwo$-graded Lie algebra  isomorphism $\varphi : L \to \Kan(P)$ that restricts to the identity map on $\TLP$.
\item[(iii)] If $P\ne 0$ and $L$ is a simple $\BCtwo$-graded Lie algebra that envelops $P$, then $L$ tightly envelops $P$ and so we have the conclusion in (ii).
\end{itemize}
\end{proposition}

\begin{proof}   (i):  Since $T_{\Kan(P)}(P)$  generates the algebra $\Kan(P)$, the implication ``$\Rightarrow$" follows from Corollary \ref{cor:funcKP}.  Moreover, the reverse implication is clear.

(ii) By Remark \ref{rem:BC2vs5gr} and Corollary \ref{cor:Kanchar},
there exists a unique $5$-graded Lie algebra  isomorphism $\varphi : L \to \Kan(P)$
that restricts to the identity map on $\TLP$.  If we use $\varphi$ to transfer the $\BCtwo$-grading
from $L$ to $\Kan(P)$, we obtain a grading that must coincide with the standard $\BCtwo$-grading by uniqueness in Proposition  \ref{prop:KanBC2}.

(iii):  This follows from (ii) using  Remark \ref{rem:BC2vs5gr} and Proposition~\ref{prop:simpletight}.
\end{proof}

\section{Weyl images of SP-graded Kantor pairs}
\label{sec:Weyl}

In this section   (except in Subsection \ref{subsec:rank2}), we  continue with the assumptions of \Sec \ref{sec:SP}. In particular,
$\Delta$ is the irreducible root system of type $\BCtwo$ with base $\Gamma = \set{\al_1,\al_2}$,
where $\al_1$ is the short basic root and $\al_2$ is the long basic root.
Let $s_\al\in W_\De$ be the reflection through the hyperplane
orthogonal to $\al$ for $\al \in \De$, and put $s_i = s_{\al_i}$ for $i=1,2$.
The generators $s_1$ and $s_2$ of $W_\De$ satisfy $s_1s_2s_1s_2 = s_2s_1s_2s_1 = -1$,
and
\[\Aut(\De) = W_\De = \set{1,\ s_1,\ s_2, \ s_2s_1,\ -1,\ -s_1,\ -s_2, \ -s_2s_1}\]
is the dihedral group of order $8$.
In particular, since $\Aut(\De) = W_\De$, all images of a $\BCtwo$-graded Lie algebra
are Weyl images.  (See \Subsec \ref{subsec:Weylactionroot}.)

\subsection{Weyl images}
\label{subsec:Weyl}
\begin{definition}
\label{def:wP}
Suppose that $P$ is an SP-graded Kantor pair and  $u\in W_\De$.  Choose a
$\BCtwo$-graded Lie algebra $\rL$ that  envelops $P$.   Then $\p{u}\rL$
is a $\BCtwo$-graded Lie algebra, which therefore  envelops an SP-graded
Kantor pair $\p{u}P$ (see \Subsec \ref{subsec:BC2}).
We call $\p{u}P$ the \emph{$u$-image} (or a \emph{Weyl image}) of $P$.
\end{definition}

In parts (i)--(iii) of the next proposition, we give an internal characterization of the Weyl image $\p{u}P$.  It follows in part (iv) that
$\p{u}P$ is well-defined.

\begin{proposition}
\label{prop:compwP}
Suppose  that $P$ is an SP-graded Kantor pair, $u\in W_\De$, $\rL$ is a
$\BCtwo$-graded Lie algebra that  envelops  $P$, and $\p{u}P$ is defined as above.  Then
\begin{itemize}
\item[(i)]  $T_{\p{u}L}(\p{u}P) = T_L(P)$;  so if
$L$ tightly envelopes~$P$
    then $\p{u}L$ tightly envelopes $\p{u}P$.
\item[(ii)] $u^{-1}(\al_1) = \pi(\al_1+a\al_2)$ and $u^{-1}(\al_1+\al_2) = \rho(\al_1+b\al_2)$
for some $\pi = \pm$, $\rho = \pm$, $a,b\in \set{0,1}$, in which case
\[(\p{u}P)_0^\sg = P_a^{\pi\sg}\andd (\p{u}P)_1^\sg = P_b^{\rho\sg}.\]
\item[(iii)]  If $\sg = \pm$ and $i,j,k\in \set{0,1}$, the $\sg$-product on $\p{u}P$
restricted to $(\p{u}P)_i^\sg \times (\p{u}P)_j^\msg\times (\p{u}P)_k^\sg$ is given by
$[[x,y],z]$ in  $\rL$, which can be expressed in terms of products in $P$ using \emph{(ii)}  and \eqref{eq:LTS}.
\item[(iv)]  $\p{u}P$ does not depend on the choice of $\rL$. (This justifies our notation $\p{u}P$.)
%\item[(v)]  $\Kan(\p{u}P)$  with its standard $\BCtwo$-grading
%is graded isomorphic to the $u$-image $\p{u}\Kan(P)$ of $\Kan(P)$ with %its standard $\BCtwo$-grading.
%So $\Kan(\p{u}P)$ is isomorphic to $\Kan(P)$ (as algebras).
\end{itemize}
\end{proposition}

\begin{proof} First
$T_L(P) = \sum_{\al \in S^-\cup S^+} L_\al$, where $S^\sg = \set{\sg\al_1, \sg(\al_1+\al_2)}$,
and similarly $T_{\p{u}L}(\p{u}P) = \sum_{\al \in S^-\cup S^+} (\p{u}L)_\al  =
\sum_{\al \in u^{-1}(S^-\cup S^+)} L_\al$.   But $S^-\cup S^+$ is the set of short roots of $\De$ and hence this set is stabilized by $W$.  So we have (i) and  the expressions
for $u^{-1}(\al_1)$ and $u^{-1}(\al_1+\al_2)$ in (ii).
Further $(\p{u}P)_0^\sg = (\p{u} \rL)_{\sg 1,0} = (\p{u} \rL)_{\sg \al_1} = \rL_{u^{-1}(\sg \al_1)}
= \rL_{\pi\sg(\al_1+a\al_2) } = \rL_{\pi\sg 1, \pi\sg a} =  P_a^{\pi\sg}$, and similarly
$(\p{u}P)_1^\sg = P_b^{\rho\sg}$, so we have (ii).  Next (iii) follows from (ii); and (iv) follows from (ii) and (iii).
\end{proof}

\begin{corollary}
\label{cor:compwP}  If  $P$ is an SP-graded Kantor pair and $u\in W_\Delta$, then
$\Kan(\p{u}P)$  with its standard $\BCtwo$-grading
is graded isomorphic to the $u$-image $\p{u}\Kan(P)$ of $\Kan(P)$ with its standard $\BCtwo$-grading.
So $\Kan(\p{u}P)$ is isomorphic to $\Kan(P)$ (as algebras).
\end{corollary}
\begin{proof} By Proposition \ref{prop:KanBC2}, $\Kan(P)$ with its standard $\BCtwo$-grading tightly  envelops $P$.  So
by Proposition \ref{prop:compwP}(i), the $\BCtwo$-graded Lie algebra $\p{u}\Kan(P)$ tightly  envelops~$\p{u}P$, which completes the proof by Proposition \ref{prop:5grvsBC2}(ii).
\end{proof}

If $P$ is an SP-graded Kantor pair, we see using Proposition \ref{prop:compwP}(iv) and \eqref{eq:actrootgraded} that
\begin{equation}\label{eq:actSPKP}
\p1P = P \andd \p{u_1}(\p{u_2} P) = \p{u_1u_2} P
\end{equation}
for $u_1,u_2\in W_\De$. In other words, we have a left action of $W_\De$ on the class of SP-graded Kantor pairs.

\begin{proposition} \label{prop:simplereflect}  Let $P$ be an SP-graded Kantor pair
and $u\in W_\De$.  Then, $P$ is simple (resp.~central simple) if and only if
$\p{u}P$ is simple (resp.~central simple).  Also
$P$ is a  split Kantor pair  of type $\type Xn$ (resp.~a form of a
split Kantor pair  of type $\type Xn$) if and only if the same is
true for $\p{u}P$.
\end{proposition}

\begin{proof}  In view of Corollary \ref{cor:compwP}, the first statement follows
from Proposition \ref{prop:simpleKan} and Corollary \ref{cor:centroid}, while the second statement
follows from  Lemma \ref{lem:splitform}.
\end{proof}

If $P$  is an SP-graded Kantor pair, two of the Weyl images of $P$ (besides $\p1P$)
are already familiar to us.  Indeed, using the notation of  Remark \ref{rem:oppshift}, we have
\begin{equation} \label{eq:oppshift}
\p{-1}P = P^\op \andd \p{s_2}P = \bar P.
\end{equation}
The first of these follows immediately from (ii) and (iii) in Proposition \ref{prop:compwP},
as does the second, since $s_2$ exchanges $\al_1$ and $\al_1+\al_2$.

\subsection{Reflection}
\label{subsec:basic}
If $P$ is an SP-graded Kantor pair, the Weyl image   $\p{s_1}P$ of $P$
is of particular interest in the theory.  For convenience we introduce the following notation:
\begin{equation}  \label{eq:brevedef}
\bP := \p{s_1}P.
\end{equation}
Since $s_1$ is the reflection in $W_\De$ corresponding to the short basic root $\al_1$,
we call $\bP$ the \emph{reflection of $P$ corresponding to the short basic root},
or more simply the \emph{reflection} of~$P$.  We have the following explicit description of $\bP$:

\begin{proposition} \label{prop:reflect}
Suppose that $P$ is an SP-graded Kantor pair.  Then
$\bP_0 = (P_0)^\op$ and $\bP_1 = P_1$ as Kantor pairs; and the $\sg$-product
$\{\ ,\ ,\  \}\brv$
on $\bP$ is given by
\begin{multline*}
\{a_0^\msg+a_1^\sg,b_0^\sg+b_1^\msg,c_0^\msg+c_1^\sg\}\brv =
\{a_0^\msg,b_0^\sg,c_0^\msg\}
-\{b_1^\msg,a_1^\sg,c_0^\msg\}
+K(a_0^\msg,b_1^\msg)c_1^\sg \\
+\{a_1^\sg,b_1^\msg,c_1^\sg\}
-\{b_0^\sg,a_0^\msg,c_1^\sg\}
+ K(a_1^\sg,b_0^\sg)c_0^\msg,
\end{multline*}
where  $a_i^{\tau}, b_i^{\tau}, c_i^{\tau}\in P_i^{\tau}$ in each case.
\end{proposition}

\begin{proof}  Since $s_1(\al_1) = -\al_1$ and $s_1 (\al_1+\al_2) = \al_1+\al_2$,
the conclusions follow easily using parts
(ii) and (iii) of  Proposition \ref{prop:compwP} and \eqref{eq:LTS}.
\end{proof}

\begin{corollary}  \label{cor:reflecttriv} Suppose $P$ is an SP-graded pair.
If the SP-grading on $P$ is the one SP-grading,
then  $\bP = P$  as SP-graded Kantor pairs.   On the other hand, if
the SP-grading on $P$ is the zero SP-grading,
then  $\bP = P^\op$ as SP-graded Kantor~pairs.
\end{corollary}

Suppose that $P$ is an SP-graded Kantor pair. By \eqref{eq:actSPKP}, \eqref{eq:oppshift} and
\eqref{eq:brevedef},
 the eight elements $1,\ s_1, s_2,  s_2s_1, -1, -s_1, -s_2,  -s_2s_1$ of $W_\De$
yield in order the following eight   SP-graded Weyl images of $P$:
\begin{equation}
\label{eq:Weylimages}
P,\ \ \bP,\ \ \bar P,\ \ \bar \bP,\ P^\op,\ \ (\bP)^\op,\ \ (\bar P)^\op,\ \ (\bar \bP)^\op.
\end{equation}
Since shifting does not change the underlying  ungraded Kantor pair,  our list includes
in general four non-isomorphic Kantor pairs:
\[P,\ \ \bP,\ \  P^\op,\ \ (\bP)^\op.\]
This  suggests the central role that reflection plays in our study of Weyl images.

\subsection{A geometric interpretation of $\bP$}
\label{subsec:geometric}
First (using the notation at the beginning of Section \ref{sec:SP})
$\set{\ep_1,\ep_2}$ and $\set{\al_1,\al_2}$ are
$\bbZ$-bases of $Q_\Delta$, with $\al_1 = \ep_1$ and
$\al_2 = -\ep_1 + \ep_2$. To avoid conflict with
our ongoing notation $(\ell_1, \ell_2) = \ell_1 \al_1 + \ell_2 \al_2$, we set
$\langle \ell_1, \ell_2\rangle = \ell_1 \ep_1 + \ell_2 \ep_2$ for $\ell_1.\ell_2\in \bbZ$.  Note that
$(k_1,k_2) = \langle k_1 - k_2, k_2\rangle$ and  $s_1(\langle \ell_1, \ell_2\rangle) = \langle -\ell_1, \ell_2\rangle$.
In the figures below, we will use the coordinates $\langle \ell_1,\ell_2 \rangle$ to plots points in $Q_\Delta$

Suppose now that $L$ is a $\BCtwo$-graded Lie algebra enveloping a SP-graded Kantor pair $P$. Then the first component grading of $L$ is
$\rL=\bigoplus_{i\in \bbZ}\rL_{i,*}$, where
\begin{equation}
\label{eq:BCtwoEuclid}
\rL_{i,*}= \textstyle \bigoplus_{\ell_2\in \bbZ}\rL_{i,\ell_2} = \bigoplus_{\ell_2\in \bbZ}\rL_{\langle i-\ell_2,\ell_2 \rangle}
= \bigoplus_{\ell_1,\ell_2\in \bbZ,\ \ell_1 + \ell_2 = i }  \rL_{\langle \ell_1,\ell_2 \rangle}.
\end{equation}
Thus, the first component grading for $L$ is pictured by the dashed lines in Figure \ref{fig:one} below. (More precisely the support of $L_{i,*}$ in $Q_\Delta$ is contained in the set of points labelled by  circles (filled or unfilled) on the dashed line labelled as $i$.)  Hence the Kantor pair  $P$ is pictured by the dashed lines labelled as  $-1$ and $1$ in Figure \ref{fig:one}. (In this case we have filled the support circles for emphasis.)  In fact,
using \eqref{eq:grBC2a}, we see that
\[P_0^\sg = L_{\langle \sg 1, 0\rangle} \andd P_1^\sg = L_{\langle 0, \sg 1\rangle},\]
so the SP-grading of $P$ is pictured as well.

%%%%%%%Figures for reflection
\setlength{\unitlength}{20pt}
\newcommand\edot{\circle{.18}}  %Empty dot
\newcommand\fdot{\circle*{.18}}  %Filled dot
\begin{minipage}{.49\textwidth}
\centering
\begin{picture}(5,6)(-2.5,-2.5)
%****Grading***********
%%%%L_2
\put(0,2){\edot}
\put(1,1){\edot}
\put(2,0){\edot}
\dashline{.2}(-.5,2.5)(2.5,-.5)
%L_1
\put(0,1){\fdot}
\put(1,0){\fdot}
\dashline{.2}(-1,2)(2,-1)
%%%%L_0
\put(-1,1){\edot}
\put(0,0){\edot}
\put(1,-1){\edot}
\dashline{.2}(-1.5,1.5)(1.5,-1.5)
%%%%L_{-1}
\put(-1,0){\fdot}
\put(0,-1){\fdot}
\dashline{.2}(-2,1)(1,-2)
%%%%L_{-2}
\put(-2,0){\edot}
\put(-1,-1){\edot}
\put(0,-2){\edot}
\dashline{.2}(-2.5,.5)(.5,-2.5)
%Draw degrees
\put(.55,-2.8){$-2$}
\put(1.05,-2.3){$-1$}
\put(1.55,-1.8){$0$}
\put(2.05,-1.3){$1$}
\put(2.55,-.8){$2$}
%%%%%Draw vertical axis
\put(0,-2.5){\vector(0,1){5.3}}
%%%%%Draw horizontal axis
\put(-2.5,0){\vector(1, 0){5.3}}
%Draw axis labels
\put(2.9,-.3){$\ell_1$}
\put(.1,2.7){$\ell_2$}
\end{picture}
\captionof{figure}{$L$ and $P$}
\label{fig:one}
\end{minipage}%
%%%%%%%%Reflected figure%%%%%%%%%
%%%%%%%%%%%%%%%%%%%%%%%%%%%%%%%%%
\begin{minipage}{.45\textwidth}
\centering
\begin{picture}(5,6)(-2.5,-2.5)
%****Grading***********
%%%%L_2
\put(0,2){\edot}
\put(1,1){\edot}
\put(2,0){\edot}
\dashline{.2}(.5,2.5)(-2.5,-.5) %For \breve L_{2}
%L_1
\put(0,1){\fdot}
\put(1,0){\fdot}
\dashline{.2}(1,2)(-2,-1) %For \breve L_{1}
%%%%L_0
\put(-1,1){\edot}
\put(0,0){\edot}
\put(1,-1){\edot}
\dashline{.2}(1.5,1.5)(-1.5,-1.5) %For \breve L_{0}
%%%%L_{-1}
\put(-1,0){\fdot}
\put(0,-1){\fdot}
\dashline{.2}(2,1)(-1,-2) %For \breve L_{-1}
%%%%L_{-2}
\put(-2,0){\edot}
\put(-1,-1){\edot}
\put(0,-2){\edot}
\dashline{.2}(2.5,.5)(-.5,-2.5) %For \breve L_{-2}
%Draw degrees
\put(.7,2.4){$2$}
\put(1.2,1.9){$1$}
\put(1.7,1.4){$0$}
\put(2.2,.9){$-1$}
\put(2.7,.4){$-2$}
%%%%%Draw vertical axis
\put(0,-2.5){\vector(0,1){5.3}}
%%%%%Draw horizontal axis
\put(-2.5,0){\vector(1, 0){5.3}}
%Draw axis labels
\put(2.9,-.3){$\ell_1$}
\put(.1,2.7){$\ell_2$}
\end{picture}
\captionof{figure}{$\breve L$ and $\bP$}
\label{fig:two}
\end{minipage}

On the other hand, setting  $\breve L = \p{s_1}L$, the
first component grading of $\breve L$ is
$\breve \rL=\bigoplus_{i\in \bbZ}\breve \rL_{i,*}$, where, using
\eqref{eq:BCtwoEuclid} applied to $\breve L$, we have
\[\breve \rL_{i,*}=  \textstyle
\bigoplus_{\ell_1,\ell_2\in \bbZ,\ \ell_1 + \ell_2 = i }  \breve \rL_{\langle \ell_1,\ell_2 \rangle}  =
\bigoplus_{\ell_1,\ell_2\in \bbZ,\ -\ell_1 + \ell_2 = i }
 \rL_{\langle \ell_1,\ell_2 \rangle} .\]
Thus the first component grading for $\breve L$ is pictured by the dashed lines in Figure \ref{fig:two}.
Next, by definition
$\bP$ is enveloped by $\breve L$,
so the Kantor pair  $\breve P$ is pictured by the dashed  lines labelled as  $-1$ and $1$ in Figure 2.  Also
we have
\[\breve P_0^\sg = L_{\langle -\sg 1,\ 0\rangle} \andd
\breve P_1^\sg = L_{\langle0, \sg 1\rangle},\]
so the SP-grading of $\breve P$ is pictured as well.

Evidently, Figure \ref{fig:two} is obtained from Figure  \ref{fig:one} by orthogonal reflection through the vertical   axis.

\subsection{The Jordan obstruction of a Weyl image} \label{subsec:balWeyl}

Suppose  $P$ is an SP-graded Kantor pair.

Recall from Subsection \ref{subsec:obstruct} that the  Jordan obstruction $J(P)$ of $P$ is the Jordan pair
$(\Kan(P)_{-2,*},\Kan(P)_{2,*})$.  So
\begin{equation*}
\label{eq:JPK}
J(P) = (\Kan(P)_{-2,0} \oplus \Kan(P)_{-2,-1} \oplus \Kan(P)_{-2,-2},
\Kan(P)_{2,0} \oplus \Kan(P)_{2,1} \oplus \Kan(P)_{2,2}).
\end{equation*}
It is interesting to note that
$J(P) = \oplus_{i=0}^2 J(P)_i$ is Peirce
graded (see subsection \ref{subsec:obstruct}) with
$J(P)^\sg_i = \Kan(P)_{\sg 2, \sg i}$, since
$\Kan(P)_{0,\sg 2} = 0$.
However, we will view $J(P)$ as an ungraded Jordan pair.

If $L$ is a $\BCtwo$-graded Lie algebra that tightly envelops $P$
then
\begin{equation} \label{eq:JPLBC}
J(P) \simeq (L_{-2,*},L_{2,*}) =
(L_{-2,0} \oplus L_{-2,-1} \oplus L_{-2,-2},
L_{2,0} \oplus L_{2,1} \oplus L_{2,2})
\end{equation}
 by Proposition \ref{prop:5grvsBC2}(ii).

In view of \eqref{eq:Weylimages}, the following proposition tells us how to compute the Jordan obstruction of any Weyl image of $P$.

\begin{proposition}  \label{prop:obstruct} If $P$ is an SP-graded Kantor pair, and
 $L$ is a $\BCtwo$-graded Lie algebra that tightly envelops $P$, then
$J(\bar P) \simeq J(P)$, $J(P^\op) \simeq J(P)^\op$ and  \begin{equation}
\label{eq:JPbrv} J(\bP) \simeq ( L_{ 2,0} \oplus L_{0,-1} \oplus   L_{-2,-2},
 L_{-2,0} \oplus L_{0,1} \oplus   L_{2,2})
 \end{equation}
 under the products $[[x,y],x]$ calculated in $L$.
\end{proposition}

\begin{proof}
If $u\in W$, the $\BCtwo$-graded Lie algebra $\p{u}L$ tightly envelopes
$\p{u}P$, by Proposition \ref{prop:compwP}(i).  Thus, by \eqref{eq:JPLBC},  $J(\p{u}P) \simeq ((\p{u}L)_{-2,*},(\p{u}L)_{2,*})$.
If $u = s_2$, $u = -1$ or $u= s_1$, then $u^{-1} = u$ maps
$\set{2 \al_1,2\al_1+\al_2,2\al_1+2\al_2}$ onto
$\set{2 \al_1,2\al_1+\al_2,2\al_1+2\al_2}$,
$\set{-2 \al_1,-2\al_1-\al_2,-2\al_1-2\al_2}$
and $\set{-2 \al_1,\al_2,2\al_1+2\al_2}$ respectively.
\end{proof}

If follows from Proposition \ref{prop:obstruct} (or from the definitions) that if $P$ is Jordan then so are  $\bar P$ and $P^\op$.    We now see that this fails for the reflection $\bP$ of~$P$.

\begin{lemma} \label{lem:Jordan}
$\bP$ is a Jordan pair if and only if
$P_0$ and $P_1$ are left ideals of $P$ that are Jordan pairs.
Suppose further that $P$ is Jordan.
Then $\bP$ is  Jordan  if and only if $P_0$ and $P_1$ are ideals of $P$, in which case
$P = P_0 \oplus P_1$ and $\breve P = (P_0)^\op \oplus P_1$ are direct sums of ideals.
\end{lemma}

\begin{proof}  Now
$\bP$ is Jordan if and only $J(\bP) = 0$, and,  by \eqref{eq:JPbrv} (with $L = \Kan(P)$) and  \eqref{eq:KPBC2}, this holds if and only if
$[P^\sg_i,  P^\sg_i] = 0$  and $[P_1^\sg,P_0^\msg] = 0$ in $\Kan(P)$ for $\sg = \pm$ and $i = 0,1$.
Moreover, this holds if and only if
$K(P_i^\sg, P_i^\sg)P^\msg = 0$ and
$\{P_i^\sg,P_{1-i}^\msg,P^\sg\} = 0$ for $\sg = \pm$, $i = 0,1$.
Then, using \eqref{eq:ini}, we see that $\bP$ is Jordan if and only if each
$P_i$ is Jordan and $\{P_{1-i}^\sg,P_i^\msg,P_i^\sg\} = 0$ for $\sg = \pm$, $i = 0,1$.   But
one checks easily using \eqref{eq:ini}  that this last condition holds  if and only if
$P_i$ is a left ideal of $P$ for $i=0,1$.

Suppose next that $P$ is Jordan.  If $\bP$ is Jordan, then we know
that $P_i$ is a left and right ideal of $P$ for $i=0,1$; and hence, by  \eqref{eq:ini},
$P_i$ is an ideal for $i=0,1$.  Conversely, suppose that each $P_i$ is an ideal of $P$.
Then $P = P_0 \oplus P_1$ is a direct sum of ideals; so by Proposition
\ref{prop:reflect}, $\breve P = (P_0)^\op \oplus P_1$ is a direct
sum of ideals and hence $\bP$ is Jordan.
\end{proof}

\begin{corollary}  \label{cor:notJordan}
Suppose that $P$ is a simple Jordan pair with non-trivial SP-grading.  Then $\bP$ is not Jordan.
\end{corollary}

There are many well-understood examples of simple Jordan pairs with
non-trivial SP-gradings  (for example simple Jordan pairs with a maximal non-invertible idempotent).
In this way,  reflection produces many examples of  simple SP-graded Kantor pairs that are not Jordan.
We  consider one such example in Subsection \ref{subsec:KPskew}.

In the  same spirit, we will give an example in
Subsection \ref{subsec:Wg3} of an SP-graded Kantor pair  of balanced $2$-dimension
$1$ whose reflection has balanced $2$-dimension $5$.  So reflection produces a Kantor pair that is far from Jordan starting from a pair that is close to Jordan.

\subsection{Remarks on Weyl images using other rank 2 root systems}\label{subsec:rank2}

Up   to this point,  we have discussed how $\BCtwo$-graded Lie algebras
arise in the theory of SP-graded Kantor pairs.
With this in mind, it is natural to wonder how Lie algebras graded by other irreducible rank 2 root systems fit into this picture.
We discuss this briefly in this subsection (with few details), and explain why we have focused on the $\BCtwo$-case.

Suppose that $\Delta$ is an irreducible root system with base $\Gamma = \set{\al_1,\al_2}$, and denote the type of $\Delta$ by
$\type X2$.
Let $m_1\al_1 + m_2\al_2$ be the highest root of $\Delta$ and suppose that $m_1 = 1$ or $2$; and let $\al_1 + n_2\al_2$ be the highest
root in $(\al_1 + \bbZ \al_2)\cap \Delta$.  Then  we have the following possible cases:
\begin{equation*}
\begin{aligned}
(1)\ &\type X2 = \type A2, m_1 = 1, n_2 = 1;
&(2)\ &\type X2 = \type B2, \al_1 \textrm{ long}, m_1 = 1, n_2 = 2; \\
(3)\  &\type X2 = \type B2, \al_1 \textrm{ short}, m_1 = 2, n_2 = 1;
&(4)\ &\type X2 = \type G2, \al_1 \textrm{ long}, m_1 = 2, n_2 = 3; \\
(5)\  &\type X2 = \type{BC}2, \al_1 \textrm{ long}, m_1 = 2, n_2 = 2;
&(6)\  &\type X2 = \type{BC}2, \al_1 \textrm{ short}, m_1 = 2, n_2 = 1.
\end{aligned}
\end{equation*}

Suppose   that $\rL$ is a $\Delta$-graded Lie algebra.
Then, since $m_1 \le 2$,  the first component grading
$\rL=\bigoplus_{i\in \bbZ}\rL_{i,*}$ of $\rL$ is a $5$-grading.
Let $P = (\rL_{-1,*},\rL_{1,*})$ be the Kantor pair  enveloped by  $L$ with this  $5$-grading.  Then
$P =  \oplus_{j\in \bbZ} P_j = \oplus_{j=0}^{n_2}P_j$, where
$P_j^\sg = L_{\sg 1,\sg j}$,  is a $\bbZ$-graded Kantor pair which we say is
\emph{enveloped by the $\Delta$-graded Lie algebra $L$}.  Further in each case  we have
$[P^\sg_i,P^\msg_j] = 0$ in $L$ if $(i-j)\al_2 \notin \Delta\cup\set0$ and
$[P^\sg_i,P^\sg_j] = 0$ in $L$ if $2\al_1 + (i+j)\al_2 \notin \Delta\cup\set0$;
and these facts translate into relations in $P$ which we call \emph{Peirce} relations.  For example the Peirce relations in Case 2 are the
relations \eqref{eq:PeirceB} in Example \ref{ex:JordSP} as well as  the relations $K(P_i^\sg,P_j^\sg) = 0$ for all $i,j$.
In Case 6, the set of relations is empty.

If we are in Case $\ell$, where $1\le \ell \le 6$, we are led to introduce
the class $\mathfrak C(\ell)$ of $\bbZ$-graded Kantor pairs that by definition
have support contained in $\set{0,\dots,n_2}$,
and satisfy the Peirce relations mentioned above.

For example $\mathfrak C(1)$, $\mathfrak C(2)$ and $\mathfrak C(3)$ are respectively
the class of SP-graded Jordan pairs, the class of Peirce graded Jordan pairs (see Example \ref{ex:JordSP}), and
the class of SP-graded Kantor pairs with $P_0$ and $P_1$ Jordan.
We leave it to the interested reader to work out  $\mathfrak C(4)$ and $\mathfrak C(5)$.
Of course  $\mathfrak C(6)$ is the class of SP-graded Kantor pairs that we have been studying.

So in each case, any $\Delta$-graded Lie algebra envelops a $\bbZ$-graded
Kantor pair in  $\mathfrak C(\ell)$; and conversely one sees exactly as in
Proposition \ref{prop:KanBC2} that any $\bbZ$-graded Kantor pair in
$\mathfrak C(\ell)$  is enveloped by a $\Delta$-graded Lie algebra.

If $P$ is in $\mathfrak C(\ell)$  and $u \in W_\Delta$, one can define
the Weyl image $\p{u}P$ in $\mathfrak C(\ell)$ as in
Definition \ref{def:wP}.  However,  to see that $\p{u}P$
is well-defined as in Proposition \ref{prop:compwP},
one needs the fact that $S^-\cup S^+$ is
invariant under $W_\Delta$,
where $S^\sg : = \set{\sg\al_1 + \sg j \al_2}_{j=1}^{n_2}$.
Since this is not true in general, it is natural to consider
Weyl images $\p{u}P$ only for $u$ in the stabilizer $V$ of $S^- \cup S^+$ in $W_\Delta$.

Now because of our work in this article we understand Weyl images in $\mathfrak C(6)$; and we
therefore also understand them in $\mathfrak C(3)$, since $\mathfrak C(3)$ is a subclass
of $\mathfrak C(6)$ that is closed under Weyl images in $\mathfrak C(6)$.
In the remaining four Cases 1, 2, 4 and 5, it turns out that we
have $u(S^+) = S^+$ or $S^-$ for $u\in V$ and hence
$\p{u}P \simeq P$ or $P^\op$ \emph{as ungraded Kantor pairs}
for $P\in\mathfrak C(\ell)$ and $u\in V$. So in those four cases,
Weyl images cannot be used to produce new ungraded Kantor pairs.
For this reason, we have only considered Case 6 in this paper.

Nevertheless, it seems that all of the classes $\mathfrak C(\ell)$, as well as analogous classes
for higher rank root systems, are of interest in their own right
and may play a role in the development of a theory of
grids for Kantor pairs following the lead of Neher in the Jordan case
(see \cite{N2} and its  references).

\section{Kantor pairs of skew transformations}
\label{sec:skew}

In this section,
\emph{let $\Vm$ and $\Vp$ be  modules and
let  $g:\Vm\times \Vp\to \bR$ be a non-degenerate bilinear form}.
If  $v^+\in \Vp$ and $v^-\in \Vm$, we
set $g(v^+,v^-)=g(v^-,v^+)$ for convenience.

\subsection{$3$-graded Lie algebras of skew transformations}
\label{subsec:Lskew}
Let $\tV = \Vm \oplus \Vp$, and define a nondegenerate symmetric
bilinear form $\tg : \tV \times \tV \to \bR$~by
\[\tg(v^- + v^+, w^- + w^+) = g(v^-,w^+) +  g(v^+,w^-).\]

For $\sg,\tau = \pm$, set $\End(\tV)^{\sg,\tau} = \set{A\in \End(\tV) \suchthat AV^{-\tau} = 0,\ AV^\tau \subseteq \Vs}$
and identify this module with $\Hom(V^\tau,\Vs)$ in the evident fashion.
Then we have
$\End(\tV) = \bigoplus_{\sg,\tau = \pm} \End(\tV)^{\sg,\tau}$ with
$ \End(\tV)^{\kappa,\lambda}\End(\tV)^{\sg,\tau} \subseteq \delta_{\lambda,\sg} \End(\tV)^{\kappa,\tau}$.
Hence, the associative algebra $\End(\tV) = \bigoplus_{i\in \bbZ} \End(\tV)_i$
 is $\bbZ$-graded with
\[\End(\tV)_{\sg 1} = \End(\tV)^{\sg,-\sg}, \quad
\End(\tV)_{0} = \End(\tV)^{-,-} \oplus \End(\tV)^{+,+}\]
and $\End(\tV)_i = 0$ otherwise.
So  $\End(\tV) = \bigoplus_{i\in \bbZ} \End(\tV)_i$  is a $3$-graded Lie algebra under the commutator product.

Let
\[\fo(\tg) = \set{A \in \End(\tV) \suchthat \tg(Av,w) + \tg (v,Aw) = 0 \text{ for } v,w\in \tV}.\]
Then $\fo(\tg)$ is a graded subalgebra of the Lie algebra $\End(\tV)$.
Thus $\fo(\tg) = \bigoplus_{i\in \bbZ}\fo(\fg)_i$ is a $3$-graded  Lie algebra  with $\fo(\tg)_i = \fo(\tg) \cap \End(\tV)_i$
for $i\in \bbZ$.  We call $\fo(\tg)$ the \emph{orthogonal Lie algebra of $\tg$}, or sometimes the
\emph{Lie algebra of skew transformations of $\tg$}, .

For $v,w\in \tV$, define $\zeta(v,w) \in \End(\tV)$ by
\[\zeta(v,w) x = \tilde g(x,w)v - \tilde g(x,v)w,\]
in which case $\zeta(v,w) = - \zeta(w,v)$.  Also $\zeta(v,w)\in \fo(\tg)$ and
\begin{equation} \label{eq:Azeta}
[A,\zeta(v,w)] = \zeta(Av,w) + \zeta(v,Aw) \quad\text{for } A\in \fo(\tg).
\end{equation}
Thus
\[\ffo(\tg) := \zeta(\tV,\tV) = \spann_\bR\set{\zeta(v,w) \suchthat v,w\in \tV}\]
is an ideal of $\fo(\tg)$ which is graded
since $\zeta(V^\sg,V^\sg) \subseteq \End(\tV)_{\sg 1}$
and $\zeta(V^+,V^-) \subseteq \End(\tV)_0$.  So
$\ffo(\tg) =  \bigoplus_{i\in \bbZ}\ffo(\fg)_i$ is a $3$-graded Lie algebra with
\[\ffo(\tg)_{\sg 1} = \zeta(V^\sg,V^\sg) \andd
\ffo(\tg)_{0} = \zeta(V^+,V^-).
 \]

If $\bR$ is a field, it is well known (and easy to check) that $\ffo(\tg)$ is the Lie algebra of all finite rank homomorphisms in $\fo(\tg)$.

 \begin{lemma} \label{lem:extendfo} Let $\bbF \in \Kalg$.  Suppose that
 either $\bbF$ is a projective-$\bR$-module, or else
 $\bbF$ is flat and each $V^\sg$ is a finitely generated $\bR$-module.
 Then
$g_\bbF : V^-_\bbF \times V^+_\bbF \to \bbF$ is nondegenerate, and
the canonical homomorphism $\End(\tV)_\bbF \to \End(\tV_\bbF)$ restricts to
a $3$-graded  $\bbF$-algebra  isomorphism  from $\ffo(\tg)_\bbF$ onto
$\ffo(\widetilde{g_\bbF})$.  (Here we can identify
$\ffo(\tg)_\bbF$ as an $\bbF$-submodule of $\End(\tV)_\bbF$ since
$\bbF$ is flat.)
\end{lemma}

\begin{proof}  As in Lemma \ref{lem:extendKan}, we know from our assumptions that the canonical $\bbF$-module homomorphisms $\Hom(V^\sg,\bR)_\bbF \to \Hom(V^\sg_\bbF,\bbF)$
and $\End(\tV)_\bbF \to \End_\bbF(\tV_\bbF)$ are injective.
Using this we can easily verify both statements.
\end{proof}
If    $\set{v_i^\sg}_{i\in I}$ is a basis for the module $W^\sg$
for $\sg = \pm$, we say that
 the bases $\set{v_i^-}_{i\in I}$ and $\set{v_i^+}_{i\in I}$ are
\emph{dual with respect to $g$} if
 $g(v_i^-,v_j^+) = \delta_{ij}$ for all $i,j\in I$.  Of course if $W^-$ or $W^+$ is not free, such dual bases cannot exist.
In fact even if $\bR$ is a field, such dual bases need not exist
\cite[\S IV.5]{J2}.
However, one can  say:

\begin{lemma}  \label{lem:fd}\ Let $\bR$ be a field.
\begin{itemize}
\item[(i)]
Suppose $W^\sg$ is a finite dimensional subspace of $\Vs$ for $\sg = \pm$ such that
$g |_{W^-\times W^+}$ is nondegenerate.  Then there exist dual bases
for $W^-$ and $W^+$ relative to $g|_{W^-\times W^+}$.
Moreover  $\Vs = W^\sg \oplus (W^\msg)^\perp$ for $\sg = \pm$,
where $(W^\msg)^\perp = \set{v^\sg \in \Vs \suchthat g(v^\sg,W^\msg) = 0}$.
\item[(ii)] Suppose that $X^\sg$ is a finite subset of $\Vs$ for
$\sg = \pm$.  Then there exists a finite dimensional subspace $W^\sg$ of $\Vs$
containing $X^\sg$ for $\sg = \pm$ such that $g |_{W^-\times W^+}$ is nondegenerate.
\end{itemize}
\end{lemma}

\begin{proof}  (i) is a well-known fact from elementary linear algebra.
(ii) is a special case of
\cite[Prop.~3.18]{LB}, or it can be checked, using (i), by induction
on  $\dim(\spann_\bR\set{X^-}) + \dim(\spann_\bR\set{X^+})$.
\end{proof}

\begin{proposition} \label{prop:orthosimple}
If $\bR$  is a field and $\dim(\Vs) \ge 3$, then $\ffo(\tg)$ is central simple.
\end{proposition}

\begin{proof}  Simplicity is well known.  One way to show
it is to reduce to the finite dimensional case using
Lemma \ref{lem:fd}, in which case simplicity is a classical fact which is easily checked.
Thus, by Lemma \ref{lem:extendfo},
$\ffo(\tg)_\bbF$ is an simple $\bbF$-algebra for each field $\bbF$ containing $\bR$.
So $\ffo(\tg)$ is central simple by \cite[Thm.~II.1.6.3(2)]{Mc}.
\end{proof}

\begin{remark} \label{rem:converseortho}
If
there exist dual bases for $V^-$ and $V^+$ relative to $g$, then the converse
of  Proposition \ref{prop:orthosimple} is true \cite[\S 3.8]{GN}.
The same remark applies to the corresponding Jordan and Kantor pair results
below.  (See Proposition \ref{prop:JPskew}(iii) and the first statement in Theorem \ref{thm:reflectskew}(iii).)
\end{remark}

\subsection{Jordan pairs of skew transformations}
\label{subsec:JPskew}

Recall that $H = (H^-,H^+)$, where
$H^\sg = \Hom(V^\msg,V^\sg)$,
is a Jordan pair under the products
$\{A^\sg,B^\msg ,C^\sg \} =A^\sg B^\msg C^\sg + C^\sg B^\msg A^\sg$.
Indeed, $H$ is the Jordan pair determined by the $3$-graded Lie
algebra $\End(\tV)$ discussed in Subsection \ref{subsec:Lskew}.

\begin{proposition} \label{prop:JPskew} \
\begin{itemize}
\item[(i)]
Let $\Sk(g) = (\Sk(g)^-,\Sk(g)^+)$, where
\begin{align*}
\Sk(g)^\sg &=  \set{ A^\sg \in H^\sg \suchthat g(A^\sg v^\msg ,w^\msg )+ g(v^\msg ,A^\sg w^\msg )=0}
\end{align*}
for $\sg = \pm$. Then $\Sk(g)$ is a subpair of
$H$, and $\Sk(g)$
is enveloped by the $3$-graded Lie algebra
$\fo(\tg)$.
\item[(ii)] Let $\FSk(g) = (\FSk(g)^-,\FSk(g)^+)$, where
\[\FSk(g)^\sg  = \zeta(\Vs,\Vs).\]
Then $\FSk(g)$ is an ideal of $\Sk(g)$, and
$\FSk(g)$ is enveloped by the $3$-graded Lie algebra
$\ffo(\tg)$.
\item[(iii)]
If $\bR$ is a field  and $\dim(\Vs) \ge 2$, the Jordan pair $\FSk(g)$ is central simple.
\end{itemize}
 \end{proposition}

\begin{proof}  (i) and (ii) follow from the discussion in
Subsection \ref{subsec:Lskew}. For (iii),
suppose $\bR$ is a field  and $\dim(\Vs) \ge 2$.
If $\dim(\Vs) = 2$, $\FSk(g)$ is one-dimensional with nontrivial  products,
so it is central simple.  If $\dim(\Vs) \ge 3$, then $\FSk(g)$ is central simple by  Theorem \ref{thm:LdetP}(iii) and Proposition  \ref{prop:orthosimple}.
\end{proof}

\begin{remark}\label{rem:LB}
The Jordan pairs $\Sk(g)$ and
$\FSk(g)$ are special cases of Jordan pairs studied in \cite{LB} and \cite{Z1}. Part (iii) in the proposition
 is a special case of
\cite[Thm.~3.9 and Prop.~4.1(3)]{LB} or \cite[Lemma 5(2)]{Z1}.
\end{remark}

\subsection{Kantor pairs of skew transformations}
\label{subsec:KPskew}
\emph{Assume now  that $e = (e^-,e^+)\in \Vm\times \Vp$ satisfies $g(e^-,e^+)=1$.}    We will use $e$ to construct
a $\BCtwo$-grading of $\ffo(\tg)$ and hence an SP-grading of
$\FSk(g)$.

Let $U^\sg = (e^\msg)^\perp$ in $\Vs$ relative to $g$, in which case we have
\[\Vs=\bR e^\sg\oplus U^\sg.\]

\begin{proposition} \label{prop:BC2skew}
$\ffo(\tg) = \bigoplus_{(i_1,i_2)\in \bbZ^2} \ffo(\tg)_{i_1,i_2}$
is a
$\BCtwo$-grading of the Lie  algebra $\ffo(\tg)$, where
\begin{equation*}
\begin{gathered}
\ffo(\tg)_{\sg 1,0} = \zeta(U^\sg,U^\sg), \quad \ffo(\tg)_{\sg 1,\sg 1} = \zeta(e^\sg,U^\sg),\\
  \ffo(\tg)_{0,0} = \zeta(U^-,U^+) + \bR \zeta(e^\msg,e^\sg), \quad \ffo(\tg)_{0,\sg 1} = \zeta(e^\sg,U^\msg)
\end{gathered}
\end{equation*}
for $\sg = \pm$, and $\ffo(\tg)_{i_1,i_2} = 0$ for all other $(i_1,i_2)$ in $\bbZ^2$.
Moreover, the first component grading of this grading is the $3$-grading of $\ffo(\tg)$ in Subsection \ref{subsec:Lskew}.
\end{proposition}

\begin{proof}
Note that it suffices to show that
$\ffo(\tg) = \bigoplus_{(i_1,i_2)\in \bbZ^2} \ffo(\tg)_{i_1,i_2}$
is a $\bbZ^2$-grading, as the rest is then clear.
There
are a number of ways to see this including a direct case-by-case check.  We give an argument that we can use again in Section~\ref{sec:KPE6}.

Suppose first that there exists $t\in \bR$ such that
\begin{equation} \label{eq:condt}
 (t^i-t^j)x = 0,\ x \in \ffo(\tg) \implies x = 0 \text{ or } i = j.
\end{equation}
For $\sg = \pm$, we define $\theta^\sg \in \GL(V^\sg)$ by
$\theta^\sg (e^\sg) = t^{\sg 1} e^\sg$ and
$\theta^\sg|_{U^\sg} = \id_{U^\sg}$.
Then $g(\theta^\sg x^\sg, \theta^\msg x^\msg) =  g(x^\sg, x^\msg)$,
so $\ttheta := \theta^-\oplus \theta^+ \in \GL(\tV)$ preserves the form $\tg$.  Hence
\begin{equation} \label{eq:thetazeta}
\ttheta \zeta (x,y) \ttheta^{-1} = \zeta(\ttheta x, \ttheta y)
\end{equation}
for $x,y\in \tV$, so we can define $\psi_\theta \in \Aut(\ffo(\tg))$ by
\[\psi_\theta(X) = \ttheta X \ttheta^{-1}\]
For $(i_1,i_2)\in\bbZ^2$, let
\[\cF_{i_1,i_2}=\{x\in\ffo(\tg)_{i_1}:
\psi_{\theta}(x)=t^{i_2} x\}.\]
(We will see that~$\cF_{i_1,i_2} = \ffo(\tg)_{i_1,i_2}$.)

Note  that the sum $\sum_{(i_1,i_2)\in \bbZ^2} \cF _{i_1,i_2}$ in $\ffo(\tg)$ is direct.
Indeed, to show this it suffices to show that $\sum_{i_2\in \bbZ} \cF _{i_1,i_2}$ is direct for $i_1\in \bbZ$, and this
is checked by a standard argument in linear algebra using \eqref{eq:condt}.
Note also that $\ffo(\tg)_{i_1,i_2} \subseteq \cF_{i_1,i_2}$ for each $(i_1,i_2)$, which follows from
\eqref{eq:thetazeta} and the definition of the modules $\cF_{i_1,i_2}$.

Since $\sum_{(i_1, i_2)\in \bbZ^2} \ffo(\tg)_{i_1,i_2} = \ffo(\tg)$
we see from the preceding paragraph that $\ffo(\tg)=\bigoplus_{(i_1,i_2)\in\bbZ^2}~\ffo(\tg)
_{i_1,i_2}$ as modules and that $\ffo(\tg)_{i_1,i_2} = \cF_{i_1,i_2}$ for each $(i_1,i_2)$.
Thus, since $\psi_{\theta}$ is an automorphism,
$\ffo(\tg)=\bigoplus_{(i_1,i_2)\in\bbZ^2}~\ffo(\tg)
_{i_1,i_2}$ is an algebra grading.

Finally, consider the general case (without assuming the existence of $t$ satisfying \eqref{eq:condt}).  Let $\bbT= \bR[t,t^{-1}]$ be the algebra of Laurent polynomials.  Then, by Lemma \ref{lem:extendfo}, we have a canonical $3$-graded $\bbT$-algebra isomorphism $\ph : \ffo(\tg)_\bbT \to \ffo(\tg_\bbT)$.  So $\ffo(\widetilde{g_\bbT})$ satisfies
\eqref{eq:condt} (over $\bbT$), and hence
$\ffo(\widetilde{g_\bbT}) = \bigoplus_{(i_1,i_2)\in \bbZ^2} \ffo(\widetilde{g_\bbT})_{i_1,i_2}$ is a $\bbZ^2$-grading.
Moreover, since $\bbT$ is flat, we
can identify  $(\ffo(\tg)_{i_1,i_2})_\bbT$ as a $\bbT$-submodule of
$\ffo(\tg)_\bbT$  and we see that
$\ph$ maps $(\ffo(\tg)_{i_1,i_2})_\bbT$ onto
$\ffo(\widetilde{g_\bbT})_{i_1,i_2}$ for each $(i_1,i_2)$.
So $\ffo(\tg)_\bbT = \bigoplus_{(i_1,i_2)\in \bbZ^2} (\ffo(\tg)_{i_1,i_2})_\bbT$ is a $\bbZ^2$-grading.
Since $\bbT$ is faithfully flat, this easily implies our result.
\end{proof}

\tolerance=900
\begin{corollary}  \label{cor:SPskew}
Let $\FSk(g)$ be the Kantor pair in Proposition \ref{prop:JPskew}(ii) and let
\begin{equation}
\label{eq:SPFskew}
\FSk(g)_0^\sg = \zeta(U^\sg,U^\sg) \andd \FSk(g)_1^\sg =\zeta(e^\sg,U^\sg).
\end{equation}
Then $\FSk(g) = \FSk(g)_0\oplus \FSk(g)_1$ is an SP-graded Jordan pair, which is enveloped
by the $\BCtwo$-graded Lie algebra in Proposition \ref{prop:BC2skew}.
\end{corollary}

\begin{remark}  In a similar fashion   one obtains a $\BCtwo$-grading of $\fo(\tg)$ and hence an
SP-grading  of $\Sk(g)$ with
$\Sk(g)_0^\sg =\{A^\sg \in \Sk(g)^\sg: A^\sg  e^\msg =0\}$
and
$\Sk(g)_1^\sg =\{A^\sg  \in \Sk(g)^\sg: A^\sg U^\msg\subseteq \bR e^\sg\}$.
We leave the details of this to the interested reader.
\end{remark}

\begin{theorem} \label{thm:reflectskew} Suppose $\bR$ is a unital commutative ring containing $\frac 16$,
$g : \Vm \times \Vp \to \bR$ is nondegenerate bilinear form,
 and
$e = (e^-,e^+)\in \Vm\times \Vp$ satisfies $g(e^-,e^+)=1$.  Let $\FSk(g)$ be the SP-graded Jordan pair  in  Corollary \ref{cor:SPskew}.
Then
\begin{itemize}
\item[(i)] $\FSk(g)\brv$ is an SP-graded Kantor pair.
\item[(ii)] The Jordan obstruction $J(\FSk(g)\brv)$ of $\FSk(g)\brv$ is isomorphic to $(U'/U'')^\op$, where $U = (U^-,U^+)$ is the Jordan pair with products
\[\{u^\sg,v^\msg,w^\sg\}^\sg = g(u^\sg,v^\msg) w^\sg + g(w^\sg,v^\msg) u^\sg,\]
$U'$ is the ideal of $U$ with $(U')^\sg =
\zeta(U^\sg,U^\sg)U^\msg$, and $U''$ is the ideal of $U'$ with $(U'')^\sg = \set{u^\sg \in (U')^\sg \suchthat
g(u^\sg,(U')^\msg) = 0}$.
\item[(iii)] If $\bR$ is a field and $\dim(\Vm) \ge 2$,  then $\FSk(g)\brv$ is a central simple SP-graded Kantor pair.  Moreover,     if $\bR$ is a field and $\dim(\Vm) \ge 3$, then $J(\FSk(g)\brv) \simeq U^\op$, so $\FSk(g)\brv$  is not Jordan.
    \end{itemize}
\end{theorem}

\begin{proof} (i) is immediate from Corollary \ref{cor:SPskew}.

(ii):  Let $P = \FSk(g)$ with the SP-grading in Corollary \ref{cor:SPskew} and let $L = \ffo(\tg)$ with the $\BCtwo$-grading
in Proposition \ref{prop:BC2skew}.
Let $L' = \langle \TLP \rangle_\text{alg}$,
$L'' = Z(L')\cap [T_{L'}(P),T_{L'}(P)]$ and
$\overline{L'} = L'/L''$.  Then $L'$ is a $\BCtwo$-graded ideal of $L$ and
$L''$ is a $\BCtwo$-graded ideal of $L'$, so $\Lbp$ is $\BCtwo$-graded.
Moreover, by Remarks \ref{rem:tight} and \ref{rem:BC2vs5gr}, $\overline{L'}$
tightly envelopes the SP-graded
Kantor pair $P$ (suitably identified in $\overline{L'}$).
So by  \eqref{eq:JPbrv} and   Proposition \ref{prop:BC2skew},  $J(\bP) \simeq (\Lbp_{0,-1},\Lbp_{0,1})$ under the products $[[x,y],z]$ in $\Lbp$.  Thus $J(\bP)^\op \simeq (\Lbp_{0,1},\Lbp_{0,-1}) \simeq (L'_{0,1},L'_{0,-1})/(L''_{0,1},L''_{0,-1})$.

We now calculate $(L'_{0,1},L'_{0,-1})$ and
$(L''_{0,1},L''_{0,-1})$ leaving some of the details to the reader.
First $L_{0,\sg 1} = \zeta(U^\sg,e^\msg)$, and we define
$\lm^\sg : U^\sg \to L_{0,\sg 1}$ by
$\lm^\sg(u^\sg) = \sg \zeta(u^\sg,e^\msg)$.
One checks that
$\lm = (\lm^-,\lm^+) : U \to (L_{0,1},L_{0,-1})$
is an isomorphism of Jordan pairs.
Next we have $L'_{0,-\sg 1} = [L_{\sg 1,0}, L_{-\sg 1,-\sg 1}] = [\zeta(U^\sg,U^\sg),\zeta(U^\msg,e^\msg) ] =
\zeta(\zeta(U^\sg,U^\sg)U^\msg,e^\msg)$.  So $\lm$ maps
$U'$ onto $(L'_{0,1},L'_{0,-1})$.
Finally $L''_{0,-\sg 1}$ is the centralizer of
$L_{-1,*}+ L_{1,*} = \zeta(V^-,V^-) + \zeta(V^+,V^+)$
in $L'_{0,-\sg 1}$, and one checks using this that
$\lm$ maps $U''$ onto $(L''_{0,1},L''_{0,-1})$.

(iii) The first statement is a consequence of Propositions \ref{prop:simplereflect} and \ref{prop:JPskew}(iii).
For the second statement, we know by Lemma \ref{lem:fd} that there exist $e^\sg_i \in U^\sg$ for $\sg = \pm$ and $i = 1,2$
such that $g(e^-_i,e^+_j) = \delta_{ij}$.
Thus,  $\zeta(e_1^\sg,e^\sg_2)e_2^\msg = e_1^\sg$
and $\zeta(u^\sg,e_1^\sg)e_1^\msg = u^\sg - g(e_1^\msg,u^\sg) e_1^\sg$ for $u^\sg \in U^\sg$,  so $U' = U$ and $U'' = 0$.
Hence we are done by (ii).
(See also Corollary \ref{cor:notJordan}.)
\end{proof}

It turns out the Kantor pairs in (iii) of the Proposition make up one of the four classes of central simple Kantor pairs that appear in the structure theorem in   \cite{AS}.

\begin{remark} \label{ex:skewfd} Suppose $\bR$ is a field.
If $f = (f^-,f^+)\in \Vm\times \Vp$ is another pair of vectors satisfying
$g(f^-,f^+)=1$, then it is easy to see using
Lemma \ref{lem:fd}  that there exists an \emph{isometry} $(\varphi^-,\varphi^+) \in \GL(\Vm) \times \GL(\Vp)$
of $g$ (satisfying $g(\varphi^- (v^-),\varphi^+ (v^+)) = g(v^-,v^+)$) such
that $\varphi^\sg(e^\sg) = f^\sg$ for $\sg = \pm$.  Using this fact it is also
easy to see that \emph{the $\BCtwo$-graded Lie algebra $\fo(\tg)$ and the Kantor pairs $\FSk(g)$ and $\FSk(g)\brv$
described in this subsection are independent
up to graded isomorphism of the choice~of~$e$}.
\end{remark}

\begin{example}[The finite dimensional case]  \label{ex:Dnfd} Suppose that $\bR$ is a field and $\Vs$ has finite dimension  $n$, where
$n\ge 3$.  Then $\ffo(\tg) = \fo(\tg)$ and
$\FSk(g) = \Sk(g)$. Choose dual basis $\set{v_i^-}_{i=1}^n$ and
$\set{v_i^+}_{i=1}^n$ for $\Vm$ and $\Vp$ relative to $g$ with
$v_1^- = e^-$ and $v_1^+ = e^+$, and use these bases to identify
$\End(\tV)^{\sg,\tau}$ with $M_n(\bR)$ for  $\sg,\tau = \pm$.  Then one sees directly using
Proposition \ref{prop:JPskew}(ii) and
Corollary \ref{cor:SPskew} that
$\FSk(g)$ is the double  of the Jordan triple system $\operatorname{A}_n(\bR)$ of
alternating $n\times n$-matrices with product $ABC + CBA$, and that
\[\textstyle \FSk(g)_0^\sg =  \sum_{i,j = 2}^n \bR (E_{ij} - E_{ji}) \andd
\FSk(g)_1^\sg = \sum_{i = 2}^n \bR (E_{i1} - E_{1i}).\]
(Here $E_{ij}$ is the $(i,j)$-matrix unit.)
Also,  $\ffo(\tg)$ is the split central simple Lie algebra of type $\type Dn$
(interpreting $\type D3$ as $\type A3$) \cite[\S IV.3]{Sel}.  So $\FSk(g)$ and
$\FSk(g)\brv$ are split central simple
Kantor pairs of
type  $\type Dn$ by Proposition~\ref{prop:simplereflect}.  Note that
$\FSk(g)$ (being Jordan) has  balanced $2$-dimension $0$;
whereas, by Proposition
Theorem \ref{thm:reflectskew}(iii),  $\FSk(g)\brv$ has  balanced $2$-dimension
$n-1$.
\end{example}

If $n$ is odd in Example \ref{ex:Dnfd}, the SP-grading of $\FSk(g)$
arises as in Example \ref{ex:JordSP} from an idempotent $c$ with trivial Peirce $0$-component
\cite[\S8.16]{L}.

\begin{remark}
If $\bR$ is algebraically closed of characteristic 0,  the Kantor pair
$\FSk(g)\brv$  in Example \ref{ex:Dnfd} is the double of  a Kantor triple system
that appears in Kantor's classification (mentioned in the introduction), where it is represented by
the notation
{
 \setlength{\unitlength}{5pt}
\begin{picture}(17,1)
\thicklines
\put(0,0){$\type C{n-1,n-1}$}
\drawline(8,.5)(11,.5)
\put(12,0){$\type A{n-1,1}$}
\end{picture}
}
\cite[(5.27)]{K1}.
In fact this notation displays the SP-grading on
$\FSk(g)\brv$.
\end{remark}

\begin{remark}\label{rem:fgpskew}   Example \ref{ex:Dnfd} can be formulated
somewhat more generally.  Indeed suppose we have the assumptions and notation of
Theorem \ref{thm:reflectskew},  suppose each $V^\sg$ is a finitely generated projective module of rank $n$ over $\bR$ with $n \ge 3$,
and suppose $g: V^- \times V^+ \to \bR$ is non-singular.  (See Subsection
\ref{subsec:fgp} below to recall this terminology.)
Then $Z(\ffo(\tg)) = 0$
and  $\ffo(\tg)$ tightly envelops $\FSk(g)$. Furthermore, $\FSk(g)$ is a form
of an SP-graded split Jordan pair of type $\type{D}{n}$ which is split if
each $V^\sg$ is free;
$\FSk(g)\brv$ is a form of a split SP-graded Kantor pair of
type $\type{D}{n}$ which is split if
each $V^\sg$ is free; and finally  the Jordan obstruction
of $\FSk(g)\brv$ is isomorphic to $U^\op$.  The verifications
of these facts, which we omit, follow the methods
(faithfully flat base ring extension to the free case) used in the next section, where
an   example of type $\type{E}{6}$ is treated in detail.
\end{remark}

\section{Construction of $\Esix$ and Kantor pairs from the exterior
algebra}

\label{sec:KPE6}

It is well known that over an algebraically closed field of characteristic $0$,
the simple Lie algebra $\cE$ of type $\type{E}{6}$ has a $5$-grading
with components $\Wg^6(V^*)$,  $\Wg^3(V^*)$,
$\operatorname{sl}(V) \oplus \frak \bR h$,
$\Wg^3(V)$  and $\Wg^6(V)$ as $\bR$-spaces, and with natural actions of
$\operatorname{sl}(V)$ on each component, where
$V$ is a six-dimensional space \cite[\S V.18, pp.89--90]{C}, \cite[\S 3.3.5]{GOV}.
Recently this fact has been used in work on gradings of  $\cE$ \cite[\S 3.4]{ADG}, \cite[\S 6.4]{EK}.

In this section, we take this point of view in a more general context to construct $5$-graded Lie algebras, SP-graded Kantor pairs and their reflections.
Throughout the section,   \emph{we assume  only that $\bR $ is a unital commutative associative  ring
(not necessarily containing $\frac 16$).  We will add further assumptions on $\bR$ as needed.}

\subsection{F.g.p.~modules and non-singular forms} \label{subsec:fgp}
 Recall that a module $M$ is \emph{finitely generated projective (f.g.p.)} if and only
 if $M$ is a
direct summand of a free module of finite rank.  If $M$ and $N$ are f.g.p.,
then $M^*$ and $M\otimes
N$ are f.g.p.  Moreover, we may identify $M$ with $M^{\ast\ast}$
where $m(\ph)=\ph(m)$ for $m\in M$ and $\ph\in M$, and the linear map
$M\otimes M^*\rightarrow\End(M)$
with $m\otimes\ph\rightarrow m\ph$ where $(m\ph)(m')=\ph
(m')m$ is bijective.  The \emph{trace} function $\tr$ on
$\End(M)$ is the unique linear map with $\tr(m\ph)=\ph(m)$.
If $A,B\in\End(M)$, then $\tr(AB)=\tr(BA)$.  For
more details see \cite[II.4.3]{B1}  and \cite[\S 2]{F}.

If $\fp \in \Spec(\bR)$, the set of   prime ideals of $\bR$, let $\bR _\fp =(\bR \backslash \fp)^{-1}\bR $ be the localization of $\bR $ at
$\fp$.  If $M$ is a finitely generated projective module, then $M_{\bR _\fp}$ is a free $\bR_\fp$-module of finite rank  \cite[II.5.2]{B2}.  If $M_{\bR _\fp}$
has rank $n$ for
all $\fp\in \Spec(\bR)$, we say $M$
has  \emph{rank} $n$, in which case
we have     $\tr(\id_M) = n 1_\bR$.

\subsection{Assumptions and notation}
\label{subsec:E6assume}
From  now on, we assume that
\begin{equation}  \label{eq:condg}
\parbox{.9\linewidth}{\centering $g:M^-\times M^+\rightarrow\bR$ is a non-singular bilinear form,   where\\
$M^-$ and $M^+$ are f.g.p.~modules of rank $n$.}
\end{equation}
(To check this condition, it is easy to see using \cite[II.5.3, (4)]{B2}
and \cite[II.2.7, Cor.~4]{B1} that it suffices to show that $M^+$ is f.g.p.~of rank $n$ and the map $v^- \rightarrow g(v^-,\  \ )$ from $V^-$ into $(V^+)^*$ is  bijective.)
If  $v^+\in M^+$ and $v^-\in M^-$, we again
set $g(v^+,v^-)=g(v^-,v^+)$ for convenience.

\begin{remark}  \label{rem:onemodule} (i) If $(M'^-, M'^+,g')$
is another triple satisfying \eqref{eq:condg}, an \emph{isomorphism}
of $(M^-, M^+,g)$ onto $(M'^-, M'^+,g')$ is a pair
$\theta=(\theta^-,\theta^+)$ where
$\theta^{\sg
}:M^\sg\rightarrow M'^\sg$ is a linear isomorphism and  $g'(\theta^-v^-,\theta^+v^+)=g(v^-,v^+)$ for $v^\sg\in
M^\sg$.

(ii)  If $N^+$ is an f.g.p.~module of rank $n$
and $\text{can} : (N^+)^* \times N^+ \to \bR$  is the canonical map given by
$\text{can}(x^-,y^+) = x^-(y^+)$, then
$((N^+)^* ,N^+,\text{can})$ satisfies \eqref{eq:condg}; and any triple
satisfying \eqref{eq:condg} is isomorphic to one obtained in this way.
So in this sense we are really just starting with an f.g.p.~module
of rank $n$. However, the more symmetric point of view taken here
is very convenient for us.

(iii) The reference  \cite[III]{B1} works with triples
$((N^+)^* ,N^+,\text{can})$ as in (ii).  In view of (ii), we can use facts from that reference.
\end{remark}

If  $A^\sg\in\End(M^\sg)$, then  since $g$ is non-singular  there exists a unique $(A^\sg)^* \in \End(M^{-\sg})$, called the \emph{adjoint} of $A^\sg$, such that
$g(v^\sg,(A^\sg)^*v^{-\sg})=g(A^\sg v^\sg,v^{-\sg})$ for $v^\sg\in M^\sg$, $v^{-\sg}\in
M^{-\sg}$

For $\sg=\pm$, we  form the exterior algebra $\Wg(M^\sg)$
with the natural $\bbZ$-grading
\[\textstyle
\Wg(M^\sg)=\bigoplus_{k\geq0}\Wg_{k}(M^\sg),
\]
where $\Wg_{k}(M^\sg)=0$ if $k<0$. For convenience, we write
the product in $\Wg(M^\sg)$ as $uv$ (rather than the usual
$u\wedge v$).

In the next three subsections, we record the facts about $\Wg(M^\sg)$
that we will need for our constructions. For more details, the
reader can consult \cite[III.7, III.10 and III.11]{B1} or \cite[\S 2]{F}.

\subsection{The $\cdot$ action of $\Wg(M^{-\sg})$ on
$\Wg(M^\sg)$}

\label{subsec:dotact}

Recall that an endomorphism $D$ of the graded algebra $\Wg(M^\sg)$ is called an \emph{anti-derivation of degree $-1$} of
$\Wg(M^\sg)$ if $D(\Wg_{k}(M^\sg
))\subseteq\Wg_{k-1}(M^\sg)$ and
\[
D(x^\sg y^\sg)=D(x^\sg)y^\sg+(-1)^{k}x^\sg D(y^{\sg })
\]
for $k\geq0$, $x^\sg\in\Wg_{k}(M^\sg)$ and $y^{\sg
}\in\Wg(M^\sg)$ \cite[III.10.3]{B1}.

For $v^{-\sg}\in M^{-\sg}$, there is a unique anti-derivation
$\Delta_{v^{-\sg}}$ of $\Wg(M^\sg)$ of degree $-1$ with
\[
\Delta_{v^{-\sg}}(v^\sg)=g(v^{-\sg},v^\sg)
\]
for $v^{\gamma}\in M^{\gamma}$ \cite[III.10.9, Ex.~2]{B1}. Since
$\Delta_{v^{-\sg}}^2=0$, the map $v^{-\sg}\rightarrow\Delta
_{v^{-\sg}}$ extends to a homomorphism $a\rightarrow\Delta_{a}$ of
$\Wg(M^{-\sg})$ into the associative algebra $\End(\Wg(M^\sg))$,
and hence we can view $\Wg(M^\sg)$ as a left module for the associative algebra $\Wg(M^{-\sg})$. We write the action as
\[
a\cdot x=\Delta_{a}(x)
\]
for $a\in\Wg(M^{-\sg})$, $x\in\Wg(M^\sg)$.

\begin{remark}
\label{rem:Bourbaki} If we identify $M^{-\sg}$ with
$(M^\sg)^*$ via the pairing $g$, then it follows from \cite[III.11.9,(68)]{B1} and \cite[III.11.8, Prop.~10]{B1} that $a\cdot x$ is the
\emph{left inner product} of $x$ by $a$ that is studied in
\cite[III.11.9]{B1}.
\end{remark}

Note that
$\Wg_{k}(M^{-\sg})\cdot\Wg_{\ell
}(M^\sg)
\subseteq\Wg_{\ell-k}(M^\sg)$. In particular if
$a\in\Wg_{k}(M^{-\sg})$ and $x\in\Wg_{k}(M^{\sg
})$, then $a\cdot x$ is a scalar in $\bR $ and, by \cite[Lemma 3(i)]{F},
we have
\[x\cdot a=a\cdot x.\]
Note also that
\begin{equation} \label{eq:nondegen}
x\in \Wg_k(M^\sg),\ x\cdot \Wg_k(M^\msg) = 0 \implies x = 0
\end{equation}
by \cite[III.11.5, Prop.~7]{B1}.
Furthermore, if $v^\sg\in M^\sg$ and $v^{-\sg}\in M^{-\sg}$,
then $v^\sg\cdot v^{-\sg}=g(v^\sg,v^{-\sg})$, so there is
no confusion of notation in case $M^-=M^+=\bR ^{n}$ and $g $ is the
usual dot product $\cdot$.

If $p\in\Wg_{n}(M^\sg)$, $q\in\Wg_{n}(M^{-\sg
})$ (recall that $n$ is the rank of $M^\sg$) and $a\in\Wg(M^{-\sg})$, then
\begin{equation}
(a\cdot p)\cdot q=(p\cdot q)a,\label{eq:apq}
\end{equation}
by \cite[III.11.11, Prop.~12(i)]{B1} or \cite[Lemma 4(i)]{F}.

\subsection{The Lie algebra $\cS$ and its $\circ$ action on $\Wg(M^\sg)$}

\label{subsec:SV} We consider the following Lie subalgebra
\begin{align*}
\cS  & =\cS(M^-,M^+,g)\\
& =\{(A^-,A^+)\in\End(M^-)\oplus\End(M^+):g(A^-
v^-,v^+)+g(v^-,A^+v^+)=0\}.
\end{align*}
of $\End(M^-)\oplus\End(M^+)$.
Note  that for $\sg = \pm$ we have the  Lie algebra isomorphism
$\iota^\sg: = \iota^\sg(M^-,M^+,g) :  \End(M^\sg) \to \cS$
given by
\[\iota^+(A^+) = (-(A^+)^*, A^+) \andd \iota^-(A^-) = (A^-,-(A^-)^*),\] in which case the inverse of $\iota^\sg$ is  projection onto the $\sg$-factor restricted to $\cS$.

If $A^\sg\in\End(M^\sg)$, there is a unique extension of
$A^\sg$ to a derivation $D_{A^\sg}$ of $\Wg(M^\sg)$
stabilizing each $\Wg_{k}(M^\sg)$, and it is easy to see that
$[D_{A^\sg},D_{B^\sg}]=D_{[A^\sg,B^\sg]}$
\cite[III.10.9, Example 1]{B1}. So each $\Wg_{k}(M^\sg)$,
and therefore also $\Wg(M^\sg)$, is a module.
for the Lie algebra $\End(M^\sg)$ with
\[
A^\sg\circ x^\sg=D_{A^\sg}(x^\sg)
\]
for $x^\sg\in\Wg_{k}(M^\sg)$.  We also view
$\Wg(M^\sg)$ as a module for $\cS$ with
\[
A\circ x^\sg:=A^\sg\circ x^\sg
\]
for  $A=(A^-,A^+)\in\cS$.

The $\cdot$ action and the $\circ$ action are related by the identity
\begin{equation}
A\circ(a\cdot x)=(A\circ a)\cdot x+a\cdot(A\circ x)\label{eq:SVmod}
\end{equation}
which, by \cite[Eq.~(2)]{F}, holds for $A\in\cS$, $a\in\Wg(M^{-\sg})$, $x\in\Wg(M^\sg)$, $\sg=\pm$.

Note that   if $A^\sg\in\End(M^\sg)$ and $p^\sg\in\Wg_{n}(M^\sg)$, we have
\begin{equation}
A^\sg\circ p^\sg=\tr(A^\sg)p^\sg.\label{eq:circn}
\end{equation}
Indeed, this holds if $M^\sg$ is free by \cite[III.10.9, Prop.~15]{B1};
and hence it holds in general by an easy localization argument (as in
\cite[\S 2, \P3]{F}).

\subsection{The elements $E(x^\sg,x^{-\sg})$ in $\cS$}

\label{subsec:BA} For $\sg=\pm$, $x\in\Wg_{k}(M^\sg)$,
$a\in\Wg_{k}(M^{-\sg})$ and $1\leq k\leq n$, we define
$E^\sg(x,a)\in\End(M^\sg)$ by
\[
E^\sg(x,a)(v^\sg)=(v^\sg\cdot a)\cdot x.
\]
In that case we have
\begin{equation}
E^\sg(x,a)^*=E^{-\sg}(a,x)\text{ and }\tr(E^{\sg
}(x,a))=k\,x\cdot a,\label{eq:Bsgprop}
\end{equation}
by parts (ii) and (iv) of \cite[Lemma 3]{F} (where $E^\sg(x,a)$ is
denoted by $e(x,a)$).

If $x^-\in\Wg_{k}(M^-)$, $x^+\in\Wg_{k}(M^+)$
and $1\leq k\leq n$, we define
\[
E(x^-,x^+)=(E^-(x^-,x^+),-E^+(x^+,x^-))\text{ and }
E(x^+,x^-)=-E(x^-,x^+).
\]
Then by the first equation in \eqref{eq:Bsgprop}, $E(x^\sg,x^{-\sg})\in\cS$ for $\sg=\pm$. Observe also that
\begin{equation}
E(x^\sg,x^{-\sg})\circ z^\sg=E^\sg(x^\sg,x^{-\sg})\circ z^\sg\label{eq:Bact}
\end{equation}
for $z^\sg\in\Wg(M^\sg)$.

\subsection{The $5$-graded Lie algebra $\tE$ }
\label{subsec:tildeE6Lie}
\emph{From now on suppose that~$n=6$.}  So
\begin{equation}  \label{eq:condg6}
\parbox{.9\linewidth}{\centering $g:M^-\times M^+\rightarrow\bR$ is a non-singular bilinear form,
where\\ $M^-$ and $M^+$ are f.g.p.~modules of rank $6$.}
\end{equation}

Let $\tS$ be the Lie algebra
direct sum
\[
\tS =\tS (M^-,M^+,g)=\cS(M^-
,M^+,g)\oplus\bR h_0,
\]
where $\bR h_0$ is free with basis $h_0 = h_0(M^-,M^+,g)$ and
$\bR h_0$ has trivial product.  We extend the action of
$\cS$ on the subalgebra
\[
\Wg_{(3)}(M^\sg):=\bR 1\oplus\Wg_{3}(M^\sg)\oplus\Wg_{6}(M^\sg)\]
of $\Wg(M^\sg)$ generated by
$\Wg_{3}(M^\sg)$ to an  action of $\tS $ on $\Wg_{(3)}(M^\sg)$ by letting
\[h_0\circ p=\sg kp \quad \text{for } p\in\Wg_{3k}(M^\sg). \]
Thus,
$\tS $ acts as derivations of $\Wg_{(3)}(M^{\sg
})$.

We define a $\bbZ$-graded module  $\tE = \tE(M^-,M^+,g) := \bigoplus_{i\in \bbZ} \tE_i$, where the modules
$\tE_i = \tE_i(M^-,M^+,g)$ are given by
\[\tE_0 :=\tS,\quad \tE_{\sg k} :=\Wg_{3k}(M^\sg) \text{  for }
k=1,2, \text{ and } \tE_{\sg k} :=0 \text{ for } k>2.\]
Define a
$\bbZ$-graded skew-symmetric product on $\tE$ by
\begin{equation}
\begin{tabular}
[c]{ll}
$[ A,B]$ is the above product in $\tS$, & $[A,p_{i}]=A\circ p_{i}$, for $i=\pm1,\pm2,$\\
$[ p_{-1},q_1]=E(p_{-1},q_1)+(p_{-1}\cdot q_1)h_0,$ &
$[p_{-2},q_2]=-(p_{-2}\cdot q_2)(h_M-2h_0),$\\
$[ p_{i},q_{i}]=p_{i}q_{i}$ for $i=\pm1,$ & $[p_{i},q_{-2i}]=p_{i}\cdot
q_{-2i}$ for $i=\pm1,$
\end{tabular}
\label{eq:Eproduct}
\end{equation}
for $A,B\in\tS$, $p_{j},q_{j}\in\tE_{j}$,
where
\[h_M := \iota^+(\id_{M^+}) = (-\id_{M^-},\id_{M^+}) \in \cS.\]

\begin{remark}
If $i\in \bbZ$ and $p_{i}\in\tE_i$, then
\begin{equation}\label{eq:hVact}
[ h_0,p_{i}]=ip_{i} \andd [h_M,p_i] = 3ip_i,
\end{equation}
so $h_M-3h_0\in Z(\tE)$.
\label{eq:centre}
Also,  if $x^\sg\in$
$\tE_{\sg1}$ and $y^{-\sg}\in$ $\tE$
$_{-\sg1}$, then
\begin{equation} \label{eq:1prod2}
[ x^\sg,y^{-\sg}]=E(x^\sg,y^{-\sg})-\sg(x^\sg\cdot y^{-\sg})h_0;
\end{equation}
so, if $z^\sg\in$ $\tE_{\sg i}$ with $i=1,2$, we
have
\begin{equation}\label{eq:3prod}
[[ x^\sg,y^{-\sg}],z^\sg]=
E^\sg(x^\sg,y^{-\sg})\circ z^\sg-i(x^\sg\cdot y^{-\sg})z^\sg
\end{equation}
by \eqref{eq:Bact} and \eqref{eq:hVact}. Moreover, if $p^\sg\in\tE_{\sg2}$ and $q^{-\sg}\in\tE_{-\sg2}$, then
\begin{equation}\label{eq:2prod2}
[ p^\sg,q^{-\sg}]=\sg(p^\sg\cdot q^{-\sg
})(h_M-2h_0);
\end{equation}
so, if $r^\sg\in\tE_{\sg2}$, we have, using \eqref{eq:hVact} and \eqref{eq:apq},
\begin{equation}
[[ p^\sg,q^{-\sg}],r^\sg]= 2(p^\sg \cdot q^{-\sg}) r^\sg  = (p^\sg \cdot q^{-\sg}) r^\sg + (r^\sg \cdot q^{-\sg}) p^\sg.
\label{eq:3prod2}
\end{equation}
\end{remark}

\begin{theorem}
\label{thm:EtildeLie} Suppose $\bR $ is a unital commutative
ring and $(M^-,M^+,g)$ satisfies \eqref{eq:condg6}.
Then $\tE=\tE(M^-,M^+,g)$  is a $5$-graded Lie algebra.
\end{theorem}

\begin{proof}
It suffices to check the Jacobi identity
\[
J(z_1,z_2,z_{3}):=[[z_1,z_2],z_{3}]+[[z_2,z_{3}],z_1
]+[[z_{3},z_1],z_2]=0
\]
for homogeneous $z_1\in$ $\tE_{d_1},z_2\in$
$\tE_{d_2},z_{3}\in$ $\tE_{d_{3}}$,
where $|d_{i}|\leq2$. Moreover, since the product is skew-symmetric,
$J(z_1,z_2,z_{3})=0\text{ implies }J(z_{\pi1},z_{\pi2},z_{\pi3})=0$ for
any $\pi\in S_{3}$. Also, since $\tE_{\sg k}=0$ for
$k>2$, we can assume $|d_1+d_2+d_{3}|\leq2$. With these observations, we
are reduced to considering the following cases for $(d_1,d_2,d_{3})$:
\begin{gather*}
(0,0,0),\ (0,0,\sg1),\ (0,0,\sg2),\ (0,\sg1,\sg
1),\ (0,2,-2),\ (0,\sg1,-\sg2),\ (0,-1,1),\\
(\sg 2,-\sg 2,\sg2),\ (\sg1,\sg1,-\sg2),(\sg1,\sg1,-\sg
1),\ (\sg1,\sg2,-\sg2),\ (\sg1,\sg2,-\sg1),
\end{gather*}
where in each case $\sg=\pm$.

Now the case $(0,0,0)$ holds since $\tS$ is a Lie algebra; the
cases $(0,0,\sg1)$, $(0,0,\sg2)$ hold since $\Wg(M^{\sg
})$ is an $\tS$-module under the $\circ$ action; the case
$(0,\sg1,\sg1)$ holds since $\tS$ acts by derivations on
$\Wg(M^\sg)$ under $\circ$; and the cases $(0,2,-2)$ and
$(0,\sg1,-\sg2)$ follow from (\ref{eq:SVmod}). This leaves the following cases:

\smallskip\emph{Case $(0,-1,1)$}: For $A\in\tS$, $x^-
\in\Wg_{3}(M^-)$ and $x^+\in\Wg_{3}(M^+)$, we
have
\begin{align*}
J(A,x^-,x^+) &  =[[A,x^-],x^+]-[A,[x^-,x^+]]+[x^-,[A,x^+]]\\
&  =E(A\circ x^-,x^+)+((A\circ x^-)\cdot x^+)h_0-[A,E(x^-
,x^+)+(x^-\cdot x^+)h_0]\\
&  \qquad\qquad+E(x^-,A\circ x^+)+(x^-\cdot(A\circ x^+))h_0\\
&  =E(A\circ x^-,x^+)+E(x^-,A\circ x^+)-[A,E(x^-,x^+)],
\end{align*}
since, by (\ref{eq:SVmod}), $(A\circ x^-)\cdot x^++x^-\cdot(A\circ
x^+)=A\circ(x^-\cdot x^+)=0$. For $v\in M^-$,
\begin{multline*}
(E(A\circ x^-,x^+)+E(x^-,A\circ x^+)
-[A,E(x^-,x^+)])\circ v
 =(v\cdot x^+)\cdot(A\circ x^-)
 \\+(v\cdot(A\circ x^+))\cdot x^-
-A\circ((v\cdot x^+)\cdot x^-)+((A\circ v)\cdot x^+)\cdot x^-
\end{multline*}
This is $0$ by (\ref{eq:SVmod}), so $J(A,x^-,x^+)=0$.

\smallskip\smallskip\emph{Case $(\sg 2,-\sg 2,\sg2)$}:
For $p,r\in\tE_{\sg2}$ and $q\in$ $\tE_{-\sg2}$, we have
$J(p,r,q)= 0 + [[r,q],p] - [[p,q],r]$, which is $0$ by \eqref{eq:3prod2}.

\smallskip\emph{Case $(\sg1,\sg1,-\sg2)$}: For $x,y\in\tE_{\sg1}$,
and $q\in$ $\tE_{-\sg2}$, we have, using
\eqref{eq:2prod2},
\begin{align}
J(x,y,q) &  =[[x,y],q]+[[y,q],x]-[[x,q],y]\nonumber\\
&  =\sg((xy)\cdot q)(h_M-2h_0)  +E(y\cdot q,x)+\sg(x\cdot(y\cdot q))h_0\nonumber\\
&  \qquad\qquad-E(x\cdot q,y)-\sg(y\cdot(x\cdot q))h_0\nonumber\\
&  =\sg((xy)\cdot q)h_M+E(y\cdot q,x)-E(x\cdot
q,y)\nonumber
\end{align}
For $v\in M^{-\sg}$, we have
\begin{align*}
(E(y\cdot q,x)-E(x\cdot q,y))\circ v  & =(v\cdot x)\cdot(y\cdot q)
-(v\cdot y)\cdot(x\cdot q)\\
& =((v\cdot x)y)\cdot q-((v\cdot y)x)\cdot q\\
& =(v\cdot(xy))\cdot q =((xy)\cdot q)v
\end{align*}
by (\ref{eq:apq}).
Thus, $E(y\cdot q,x)-E(x\cdot q,y)=-\sg((xy)\cdot q)h_M$,
so $J(x,y,q)=0$.

\smallskip\emph{Case $(\sg1,\sg1,-\sg1)$}: For $x,y\in\tE_{\sg1}$ and $a\in$ $\tE_{-\sg1}$, we have
\begin{align*}
J(x,y,a) &  =-[a,[x,y]]+[[y,a],x]-[[x,a],y]\\
&  =-a\cdot(xy)+E^\sg(y,a)\circ x-(y\cdot a)x-E^\sg(x,a)\circ
y+(x\cdot a)y
\end{align*}
using \eqref{eq:3prod}. This is $0$ by \cite[Lemma 3(vi)]{F}.

\smallskip\emph{Case $(\sg1,\sg2,-\sg2)$}: For $x\in$
$\tE_{\sg1}$, $p\in$ $\tE_{\sg2}$,
$q\in$ $\tE_{-\sg2}$, we have
\begin{align*}
J(x,p,q) &  =[[x,p],q]-[[q,p],x]-[[x,q],p] &  & \\
&  =0+\sg[(q\cdot p)(h_M-2h_0),x]-(x\cdot
q)\cdot p &  &  \text{by \eqref{eq:2prod2}}\\
&  =(q\cdot p)x-(x\cdot q)\cdot p &  &  \text{by \eqref{eq:hVact}.}
\end{align*}
This is $0$ by \eqref{eq:apq}.

\smallskip\emph{Case $(\sg1,\sg2,-\sg1)$}: For $x\in$
$\tE_{\sg1}$, $p\in$ $\tE_{\sg2}$,
$a\in$ $\tE_{-\sg1}$, we have
\begin{align*}
J(x,p,a) &  =[[x,p],a]-[[a,p],x]-[[x,a],p] &  & \\
&  =0-(a\cdot p)x-E^\sg(x,a)\circ p+2(x\cdot a)p &  &  \text{by
\eqref{eq:3prod}}\\
&  =-(a\cdot p)x-\tr(E^\sg(x,a))p+2(x\cdot a)p &  &  \text{by
\eqref{eq:circn}}\\
&  =-(a\cdot p)x-3(x\cdot a)p+2(x\cdot a)p &  &  \text{by \eqref{eq:Bsgprop}}
\\
&  =-(a\cdot p)x-(x\cdot a)p. &  &
\end{align*}
But if $q\in$ $\tE_{-\sg2}$, we have
\begin{align*}
-((a\cdot p)x)\cdot q &  =(x(a\cdot p))\cdot q=x\cdot((a\cdot p)\cdot
q)=x\cdot((p\cdot q)a)\quad\text{by \eqref{eq:apq}}\\
&  =(x\cdot a)(p\cdot q)=((x\cdot a)p)\cdot q,
\end{align*}
so $-(a\cdot p)x=(x\cdot a)p$ by (\ref{eq:nondegen}).
\end{proof}

\begin{remark}\label{rem:ranktE}   For $i=-2,-1,,0,1,2$,
the module $\tE_i$  is f.g.p.~of rank
$1$, $20$, $37$, $20$, $1$ respectively.
Indeed this holds
since $\Wg_k(M^\sg)$ is f.g.p.~of rank $6 \choose k$ for $0 \le k \le 6$ and $\End(M^\sg)$ is f.g.p.~of rank $36$ \cite[II.5.3]{B1}.
So $\tE$ is f.g.p.~of rank $79$.
\end{remark}

\begin{remark}\label{rem:extendE}
Suppose that  $\bbF\in\Kalg$.
One sees using
\cite[II.5.3]{B2}, \cite[II.7.5]{B2} and \cite[II.5.3]{B1}
that $(M^-_\bbF,M^+_\bbF. g_\bbF)$ satisfies  \eqref{eq:condg6} (over $\bbF$).
We now show that \emph{there exists a canonical $\bbZ$-graded $\bbF$-algebra isomorphism $\omega : \tE_\bbF \to \tE\tripF$}.  We define $\omega$ by defining its restriction $\omega_i$ to the
$i$-th graded component $(\tE_i)_\bbF$ of   $\tE_\bbF$ for $-2\le i \le 2$.
First
\[
(\tE_0)_\bbF = \cS_\bbF \oplus \bbF (1\otimes h_0)
\text{ and } \tE_0\tripF= \cS\tripF \oplus h_0\tripF. \]
We define
$\omega_0$ on $\cS_\bbF$ as the
composite $\bbF$-algebra isomorphism
$\cS_\bbF \to \End(M^+)_\bbF \to \End(M_\bbF^+) \to \cS\tripF$,
where the first isomorphism is induced by $(\iota^+)^{-1}$ (see Subsection \ref{subsec:SV}), the second
is canonical \cite[II.5.4]{B1} and the third is $\iota^+\tripF$; and we define $\omega_0 ( 1\otimes h_0 ) = h_0\tripF$.
Lastly we have the canonical $\bbF$-module isomorphism
\[\omega_{\sg k} : (\tE_{\sg k})_\bbF = (\Wg_{3k}(M^\sg))_\bbF
\to \Wg_{3k}(M^\sg_\bbF) = \tE_{\sg k}\tripF\]
for $\sg = \pm$, $k = 1,2$ \cite[III.7.4, Prop.~8]{B1}.  One checks that
the direct sum $\omega$ of these maps is in fact an $\bbF$-algebra isomorphism as desired.
\end{remark}

\subsection{The $5$-graded Lie algebra $\cE$}
\label{subsec:E6Lie}
We now introduce an ideal $\cE$ of $\tE$  which we will see in Proposition
\ref{prop:E0fgp} is actually the derived algebra of $\tE$.
For this we define a linear map $\lambda = \lambda(M^-,M^+,g) : \tE_0 \to \bR$ by
\[\lambda((A^-,A^+)+ah_0) = \tr(A^+)+3a\]
for $(A^-,A^+)\in \cS$ and $a\in \bR$.
Let
 \[\cE= \cE(M^-,M^+,g) :=
 \textstyle \bigoplus_{i\in \bbZ} \cE_i \ \text{ in } \ \tE,\]
where the submodules $\cE_i = \cE_i(M^-,M^+,g)$ of $\tE$ are given by
\[\cE_{i} :=\tE_{i} \text { for } i \neq 0,  \andd \cE_0 :=\{X\in\tE_0:\lambda(X)=0\}.\]
One checks easily using
\eqref{eq:Eproduct}, $\tr(\id_{M^+}) = 6(1_\bR)$
and $\tr([A^+,B^+])=0$, that
$\lambda([\tE_{-i},\tE_{i}])=0$ for $i=0,1,2$.
So $\cE$ contains the derived algebra $[\tE,\tE]$ and is hence a $5$-graded ideal of $\tE$.

If $\bbF\in \Kalg$, one checks that the $\bbF$-algebra isomorphism
$\omega_0 : (\tE_0)_\bbF \to \tE_0\tripF$ in Remark \ref{rem:extendE} satisfies
\begin{equation} \label{eq:extendlm}
\lambda(M^-,M^+,g)_\bbF = \lambda(M^-_\bbF,M^+_\bbF,g_\bbF) \circ \omega_0.
\end{equation}

We have a (non-canonical) direct sum decomposition of $\tE_0$:

\begin{lemma} \label{lem:lmsurject}
The map $\lambda : \tE_0 \to \bR$ is surjective, and if $X_0\in \tE_0$ is chosen so that $\lambda(X_0) = 1$, then
$\tE_0  = \cE_0 \oplus \bR X_0$ and $\bR X_0$ is free of rank 1 with basis $X_0$.
\end{lemma}

\begin{proof}
For $\fp\in \Spec(\bR)$, $M_{\bR_\fp}^+$ is free of rank $6$, so there is $A^+\in
\End(M_{\bR_\fp}^+)$ with $\tr(A^+)=1$.  Hence $\lambda(M^-_{\bR_\fp},M^+_{\bR_\fp},M^-_{\bR_\fp})$
is surjective,
so $\lambda_{\bR_\fp}$ is surjective by \eqref{eq:extendlm} (with $\bbF = \bR_\fp$).
Since this holds for all $\fp\in \Spec(\bR)$ and since $\tE_0$
is f.g.p., it follows that $\lambda: \tE_0 \rightarrow
\bR $ is surjective \cite[II.3.3]{B2}.  The rest is clear.
\end{proof}

\begin{remark}\label{rem:rankcE}  It follows from Remark
\ref{rem:ranktE} and Lemma \ref{lem:lmsurject} that for $i=-2,-1,,0,1,2$,
the module $\cE_i$  is f.g.p.~of rank
$1$, $20$, $36$, $20$, $1$ respectively.   So
$\cE$ is f.g.p.~of rank $78$.
\end{remark}

If $\bbF\in \Kalg$,  it follows from Lemma \ref{lem:lmsurject} that
the canonical homomorphism $\cE_\bbF \to \tE_\bbF$ induced by inclusion
is injective.  Using this map \emph{we will identify
$\cE_\bbF$ as a $\bbZ$-graded subalgebra of $\tE_\bbF$}.

\begin{lemma}\label{lem:extendE} If $\bbF\in \Kalg$,  the restriction
of the isomorphism $\omega$ in Remark
\ref{rem:extendE}
is a $\bbZ$-graded $\bbF$-algebra isomorphism from $\cE_\bbF$ onto $\cE\tripF$.
\end{lemma}

\begin{proof}  One checks using the direct sum decomposition $\tE_0  = \cE_0 \oplus \bR X_0$ in Lemma \ref{lem:lmsurject}, that
$(\cE_0)_\bbF = \{X\in(\tE_0)_\bbF :\lambda_\bbF(X)=0\}$.  It follows from this and \eqref{eq:extendlm} that
$\omega_0 ((\cE_0)_\bbF) = \cE_0\tripF$ as needed.
\end{proof}

\begin{remark} \label{rem:freenotation}  Suppose
$M^\sg$, $\sg=\pm$, is free of rank $6$.  We now introduce some
notation that will be useful in our calculations.
Since $g$ is non-singular, we can choose   bases
$B^\sg=\{v_1^\sg,\ldots,v_{6}^\sg\}$ for $M^\sg$, $\sg = \pm$,
that are dual relative to $g$.
We use these bases to identify
$\End(M^-)$ and $\End(M^+)$ with $\mathrm{M}_{6}(\bR )$,  in which case
$(A^\sg)^*$ is the transpose of $A^\sg$ for $A^\sg \in \End(M^\sg)$.  Then
$E(v_{i}^+,v_{j}^-)=\iota^+(E_{ij})$ for $i\neq j$, where $E_{ij}$ is the
standard matrix unit in $\mathrm{M}_{6}(\bR )$ and $\iota^+ = \iota^+(M^-,M^+,g)$.
For $S=\{i_1 < \cdots < i_{\ell}\}\subseteq[1,6]:=\{1,2,3,4,5,6\}$, let
\[
v_{S}^-  =v_{i_{\ell}}^-\cdots v_{i_1}^- \andd
v_{S}^+  =v_{i_1}^+\cdots v_{i_{\ell}}^+.
\]
so $\{v_{S}^+:\left\vert S\right\vert =k\}$ and
$\{v_{S}^+:\left\vert S\right\vert =k\}$  are dual bases for
$\Wg_{k}(M^-)$ and $\Wg_{k}(M^+)$ relative to the
bilinear form~$\cdot$.  Note that if $S=\{i<j<k\}$ and $\left\vert T\right\vert =3$,
then $E^+(v_{S}^+,v_{T}^-)   =E^+(v_{i}^+,(v_{j}^+v_{k}^+)\cdot
v_{T}^-)+E^+(v_{j}^+,(v_{k}^+v_{i}^+)\cdot v_{T}^-)+E^+
(v_{k}^+,(v_{i}^+v_{j}^+)\cdot v_{T}^-)$
by \cite[Lemma 3(v)]{F},  so by \eqref{eq:1prod2} we have
\begin{equation}
[v_S^+,v_T^-] =\left\{
\begin{tabular}
[c]{ll}
$\iota^+(E_{ii}+E_{jj}+E_{kk}) - h_0$ & if $T=S,$\\
$\pm \iota^+(E_{pq})$ & if $S\backslash T=\{p\},\ T\backslash S=\{q\},$\\
$0$ & if $\left\vert T\cap S\right\vert \leq1$.
\end{tabular}
\right. \label{eq:EvSvT}
\end{equation}
Finally let
\begin{equation*}
h_{i}   =\iota^+(E_{ii}-E_{i+1,i+1})\text{ for }1\leq i\leq5 \andd
h_{6}   =\iota^+(E_{44}+E_{55}+E_{66})-h_0
\end{equation*}
in $\cE_0$.
\end{remark}

\begin{lemma} Suppose that $M^\sg$, $\sg=\pm$, is free.  With the above notation, the set
\[
B_{\cE_0}=\{h_{i}:1\leq i\leq6\}\cup\{\iota^+(E_{ij}):1\leq i\neq
j\leq6\}
\]
is a basis for $\cE_0(M^-,M^+,g)$, and
\[
B_{\cE}=B_{\cE_0}\cup\{v_{S}^\sg:\sg=\pm,\left\vert
S\right\vert =3\}\cup\{v_{[1,6]}^\sg:\sg=\pm\}
\]
is a basis for $\cE(M^-,M^+,g)$, where  $[1,6] := \set{1,2,3,4,5,6}$.
\end{lemma}

\begin{proof}
Since $\{v_{S}^\sg:\left\vert S\right\vert =k\}$ is a basis for
$\Wg_{k}(M^\sg)$, it suffices to show $B_{\cE_0}$
is a basis for $\cE_0$.  Since  $\tE_0=\cS
\oplus\bR h_0=\cS\oplus\bR h_{6}$ and $B_{\cE
_0}\backslash\{h_{6}\}$ is independent in $\cS$, we see that
$B_{\cE_0}$ is independent.  If $X = (A^-,A^+)+ah_0
\in\cE_0$ with $(A^-,A^+) \in \cS$ and $a\in \bR$,
then $X+ah_6 \in \cE_0$
 and $X+ah_6 = (B^-,B^+)\in \cS$ for some $B^\sg \in M_6(\bR)$.
Thus $\tr(B^+) = 0$ and hence
$(B^-,B^+) \in \spann_\bbZ\set{h_1,\dots,h_5}$.
\end{proof}

\begin{proposition} \label{prop:E0fgp}
$\cE_0=[\cE_{-1},\cE_1]$;\ \  $\cE$ is
generated by $\cE_{-1}\cup\cE_1$; \  and
\[\cE=[\tE,\tE]=[\cE,\cE].\]
\end{proposition}

\begin{proof}
It suffices to show $[\cE_{i},\cE_{j}
]=\cE_{i+j} $ for $(i,j)=(-1,1)$, $(\sg1,\sg1)$, and
$(-\sg1,\sg2)$.

First we assume that $M^\sg$ is free of rank $6$,
$\sg=\pm$, and we use the notation of Remark
\ref{rem:freenotation}.
Then \eqref{eq:EvSvT} and \eqref{eq:1prod2}  show
that $\iota^+(E_{ij})$
and $\iota^+(E_{ii} + E_{jj} + E_{jj})-h_0$ are in $[\cE_{- 1},\cE_1]$ for distinct $i,j,k$.
Thus $B_{\cE_0}
\subseteq[\cE_{-1},\cE_1]$, so $\cE
_0=[\cE_{-1},\cE_1]$.  Also, if $\left\vert S\right\vert
=\left\vert T\right\vert =3$, then
\begin{align} \label{eq:[vS,vT]}
& [ v_{S}^\sg,v_{T}^\sg]
 =v_{S}^\sg v_{T}^\sg=\left\{
\begin{tabular}
[c]{ll}
$\pm v_{[1,6]}^\sg$ & $S\cup T=[1,6]$\\
$0$ & $S\cup T\neq[1,6]$
\end{tabular}
\right.,\\
&[ v_{S}^{-\sg},v_{[1,6]}^\sg]
 =v_{S}^{-\sg}\cdot v_{[1,6]}^\sg=\pm v_{[1,6]\backslash S}^\sg,\nonumber
\end{align}
so $[\cE_{\sg1},\cE_{\sg1}]=\cE_{\sg2}$ and
$[\cE_{-\sg1},\cE_{\sg2}]=\cE_{\sg1}$.

In
the general case, it follows from Lemma \ref{lem:extendE}  and the preceding paragraph
that for $\fp\in \Spec(\bR)$ we have $[(\cE_{i})_{\bR_\fp},(\cE_{j})_{\bR _\fp}]=(\cE_{i+j})_{\bR _\fp}$
for $(i,j)=(-1,1)$, $(\sg1,\sg1)$, and
$(-\sg1,\sg2)$.
But by Remark \ref{rem:rankcE},   $\cE_{i}$, $\cE_{j}$, $\cE_{i+j}$
and $\cE_{i}\otimes\cE_{j}$ are f.g.p.~with $[\cE
_{i},\cE_{j}]\subseteq\cE_{i+j}$.
Then a localization argument
using the multiplication map $\cE_{i}\otimes\cE_{j}
\rightarrow\cE_{i+j}$ shows that
$[\cE_{i},\cE_{j}]=\cE_{i+j}$.
\end{proof}

\begin{remark}  \label{rem:isoE}
For use in the next proof
(and again in Proposition \ref{prop:BC2E6}), we describe here two isomorphisms  involving the
Lie algebras $\tE$ and $\cE$. We omit some of the details which are easy to fill in.

(i)  First suppose $(M'^-, M'^+,g')$
is another triple satisfying \eqref{eq:condg6} and
$\theta=(\theta^-,\theta^+)$ is an isomorphism of
 $(M^-, M^+,g)$ onto $(M'^-, M'^+,g')$.
It is easy to see that $\theta^\sg$ extends to a graded algebra
isomorphism $\Wg(M^\sg)\rightarrow\Wg
(M^{\prime\sg})$ and induces by conjugation an isomorphism
$\End(M^\sg)\rightarrow\End(M^{\prime\sg})$ such that the ``products''
$\circ$ and $\cdot$ are preserved.  Thus \emph{$\theta$ induces a graded isomorphism}
\begin{equation*}
\psi_{\theta}:\tE(M^-,M^+,g)\rightarrow\tE (M'^-,M'^+,g')
%\label{eq:psitheta}.
\end{equation*}
which maps $\tS(M^-,M^+,g)$ onto $\tS(M'^-,M'^+,g')$
and  $h_0$ to $h_0 (M'^-,M'^+,g')$. It is  clear that
$\psi_{\theta}$
maps $\cE(M^-,M^+,g)$ onto $\cE(M'^-,M'^+,g')$.

(ii) We consider now the opposite triple $(M^-,M^+,g)^{op}=(M^+,M^-,g)$.  It is
easy to see that for $x\in\Wg_{3}(M^\sg)$ and $a\in
\Wg_{3}(M^{-\sg})$, we have $(E^{op})^{-\sg}(x,a)=E^\sg(x,a)$, so
\begin{align*}
E^{op}(x_{-1},x_1) & =(-E^+(x_1,x_{-1}),E^-(x_{-1},x_1)),
\end{align*}
for $x_{i}\in\cE_{i}$.  Then \emph{the map
$\zeta:\tE(M^-,M^+,g)\rightarrow\tE(M^+,M^-,g)$ defined by
\begin{equation*}
\zeta(x_{-2}+x_{-1}+(A^-,A^+)+ah_0+x_1+x_2) = x_2
+x_1+(A^+,A^-)-ah_0^{op}+x_{-1}+x_{-2}
%\label{eq:zeta}
\end{equation*}
for $x_{i}\in\cE_{i}$, $(A^-,A^+)\in\cS(M^-,M^+,g)$   and
$a\in\bR $, is a graded algebra isomorphism}.  For the  proof of this, we check two cases (the others being easily checked).  We have
\begin{align*}
[\zeta(x_{-1}),\zeta(x_1)]^{op}  & =[x_{-1},x_1]^{op}=E^{op}(x_{-1},x_1)-(x_{-1}\cdot x_1)h_0^{op}\\
& =\zeta(-E(x_1,x_{-1})+(x_{-1}\cdot x_1)h_0) =\zeta([x_{-1},x_1])
\end{align*}
and
\begin{align*}
[\zeta(x_{-2}),\zeta(x_2)]^{op}  & =[x_{-2},x_2]^{op}=(x_{-2}\cdot
x_2)((-\id_{M^+},\id_{M^-})-2h_0^{op})\\
& =\zeta((x_{-2}\cdot x_2)((\id_{M^-},-\id_{M^+})+2h_0) =\zeta([x_{-2},x_2]).
\end{align*}
Evidently  $\zeta$   maps
$\cE(M^-,M^+,g)$ onto $\cE(M^+,M^-,g)$.
\end{remark}

\begin{remark} \label{rem:Galois}  By \cite[II.5, Exercise 8]{B2},
there exists a faithfully flat $\bbF\in \Kalg$ such that $M^\sg$
is free of rank $6$ for $\sg = \pm$.  So, using Remarks \ref{rem:extendE},
\ref{rem:freenotation} and \ref{rem:isoE}(i), we have $5$-graded isomorphisms
$\cE_\bbF \simeq \cE(\bbF^6,\bbF^6,\cdot) \simeq \cE(\bR^6,\bR^6,\cdot)_\bbF$.
In other words, $\cE$ is an \emph{$\bbF/\bR$-form of $\cE(\bR^6,\bR^6,\cdot)$}
as defined in
\cite[III.1.1]{Ser}.
A similar remark, which we leave to the reader, holds for the
Kantor pairs $\Wg_3$ and the SP-graded Kantor pairs $\Wgref$
described in Subsections \ref{subsec:Wg3} and \ref{subsec:BCtwoE} below.
\end{remark}

\begin{proposition}
\label{prop:complexE6}  If $\bR = \bbC$, then $\cE=\cE(\bbC ^{6},\bbC
^{6},\cdot)$ is a simple Lie algebra of type $\Esix$ and
$B_{\cE}$ is a Chevalley basis of $\cE$.
\end{proposition}

\begin{proof}  We use the  notation of Remark \ref{rem:freenotation}
relative to the standard dual bases.

We first note that $\eta:(A^-,A^+)+
ah_0\rightarrow A^++\frac1{3}a\id_{\bbC ^{6}}$
is a homomorphism $\eta:\tS
=\tS(\bbC^{6},\bbC ^{6},\cdot)\rightarrow \End(M^+)=\mathrm{M}_{6}(\bbC )$
with $\eta(X)\circ x=X\circ x$
for $X\in\tS$, $x\in\cE_{k}$, $k=1,2$.  Since
$\tr(\eta(h_{6}))=1$, $\eta$ maps $B_{\cE_0}$ to a basis for
$\mathrm{M}_{6}(\bbC )$, so $\eta$ restricts to an isomorphism
$\cE_0\rightarrow\mathrm{M}_{6}(\bbC )$.

Using \eqref{eq:[vS,vT]} we see that if $I$ is an ideal of $\cE$ and
$x\in I$ with  nonzero component in
$\cE_{\sg1}\oplus \cE_{\sg2}$, then some
$y\in[ x,\cE_{\sg1}]\cup\bbC x$ has component $v_{[1,6]}^\sg$ in $\cE_{\sg
2}$; so (fixing $S \subseteq [1,6]$ with $\left\vert S\right\vert =3$) some $z\in[ y,\cE_{-\sg1}]$ has component
$v_{S}^\sg$ in $\cE_{\sg1}\oplus\cE_{\sg2}$, and  hence
$[[[z,\cE_{-\sg2}],\cE_{-\sg1}],\cE_{\sg1}]=\cE_{-\sg1}$.  Thus, an ideal
$I$ is either contained in $\cE_0$ or contains $\cE
_{\sg1}$, $\sg = \pm$, and hence $\cE$.  If $I\subseteq
\cE_0$, then $\eta(I)$ is an ideal of the Lie algebra $\mathrm{M}_{6}(\bbC )$,
so $\eta(I)=0$, $\bbC \id_{\bbC ^{6}}$ or $\mathrm{M}_{6}(\bbC
)$. On the other hand, $\eta(I)\circ\cE_1
=[I,\cE_1]\subseteq I\cap\cE_1=0$, so $I=0$.  Thus, $\cE$ is simple.

Since $\eta$ maps $\cH:=\spann_\bbC\set{h_1,\ldots,h_{6}}$ to the set of
diagonal matrices, $\cH$ is an abelian Cartan subalgebra of
$\cE_0$.  Also,
$t=\iota^+(\id_{\bbC ^{6}})-2h_0\in\cH$ has $\ad(t)= \ad(h_0)$ on $\cE$, so the
normalizer of $\cH$ in $\cE$ is contained in $\cE_0$ and
hence equals $\cH$. Thus, $\cH$ is an abelian Cartan
subalgebra of $\cE$.  It is clear that $\mathrm{ad}(h)$ is
diagonalizable  on $\cE$ for $h\in\cH$, so we have $\cE$
$=\bigoplus_{\mu\in\cH^*}$ $\cE(\mu)$, where
\[
\cE(\mu)=\{x\in\cE:[h,x]=\mu(h)x \text{ for } h\in\cH\}
\]
for $\mu\in\cH^*$. Let
$\Sigma=\{\mu\in\cH^*:\mu\neq0,\cE(\mu)\neq0\}$, so
$\cE$ $=\bigoplus_{\mu\in\Sigma\cup\{0\}}$ $\cE(\mu)$ with
$\cE(0)=\cH$.  Now let $\ep_{i}
\in\cH^*$ with $\ep_{i}(h)=a_{i}$ where
$\eta(h)=\diag(a_1,\ldots,a_{6})$. It is easy to see that the
elements $\mu\in\Sigma$ and the corresponding root spaces $\cE
(\mu)=\bR x_{\mu}$ are:
\begin{align}
\mu &  =\ep_{i}-\ep_{j},\ i\neq j, &
\text{with }x_{\mu} &  =\iota^+(E_{ij});\label{eq:E6roota}\\
\mu &  =\sg(\ep_{i}+\ep_{j}
+\ep_{k}),\ i< j < k,\ \sg=\pm, &
\text{with }x_{\mu} &  =v_{\{i,j,k\}}^\sg;\label{eq:E6rootb}\\
\mu &  =\sg(\ep_1+\ep_2
+\ep_{3}+\ep_{4}+\ep
_{5}+\ep_{6}),\ \sg=\pm, & \text{with }x_{\mu} &
=v_{[1,6]}^\sg.\label{eq:E6rootc}
\end{align}

To show that  $\cE$ has type $\Esix$, let $\mu_{i}
=\ep_{i}-\ep_{i+1}\text{ for }1\leq
i\leq5$ and $\mu_{6}=\ep_{4}+\ep
_{5}+\ep_{6}$; let $\Pi=\{\mu_1,\ldots,\mu_{6}\}$; and let
$A_{ij} $ be the Cartan integer of the pair $(\mu_{i},\mu_{j})$ for $1\leq
i,j\leq6$. An examination of the $\mu_{j}$-string through $\mu_{i}$ shows that
$A=(A_{ij})$ is the Cartan matrix of type $\Esix$.

Let $h_\mu = [x_\mu,x_{-\mu}]$ for
$\mu \in \Sigma$.   To show that $B_{\cE}$ is a Chevalley basis, we need to show
\cite[p.147]{H}
\begin{description}
\item[(a)] $[h_{\mu},x_{\mu}]=2x_{\mu}$  for $\mu\in \Sigma$,
\item[(b)] $h_{\mu_{i}}=h_{i}$, $i=1,\dots,6$,
\item[(c)] the linear map with $x_{\mu}\rightarrow-x_{-\mu}$, $h_{i}
\rightarrow-h_{i}$ is an automorphism of $\cE$.
\end{description}
By (\ref{eq:EvSvT}), we have
\begin{align*}
h_{\mu} &  =\iota^+(E_{ii}-E_{jj}) &
\text{for $\mu$ as in \eqref{eq:E6roota}};\\
h_{\mu} &  =\sg\iota^+(E_{ii}+E_{jj}+E_{kk})-\sg h_0 &
\text{for $\mu$ as in \eqref{eq:E6rootb}};\\
h_{\mu} &  =\sg\iota^+(\id_{\mathrm{M}_{6}(\bbC)})-\sg 2h_0 &
\text{for $\mu$ as in \eqref{eq:E6rootc}},
\end{align*}
and (a) and (b) follow.  For (c) let $\theta=(\theta^-,\theta^+)$ where
$\theta^\sg:M^\sg\rightarrow M^{-\sg}$ with $\theta^{\sg
}(v_{i}^\sg)=v_{i}^{-\sg}$, so $\psi_{\theta}$ described in
Remark \ref{rem:isoE}(i) and
$\zeta$  described in Remark \ref{rem:isoE}(ii) are isomorphisms
$\tE(M^-,M^+,g)\rightarrow
\tE(M^+,M^-,g)$.  For $A\in\mathrm{M}_{6}(\bbC )$, we have
$\zeta^{-1}
\psi_{\theta}(-A^*,A)=\zeta^{-1}(-A^*,A)=(A,-A^*)$
and
$\zeta^{-1}\psi_{\theta}(h_0)=\zeta^{-1}(h_0^{op})=-h_0$.  Thus,
$\zeta^{-1}\psi_{\theta}(x_{\mu})=-x_{-\mu}$ for $\mu=\ep
_{i}-\ep_{j}$ and $\zeta^{-1}\psi_{\theta}(h_{i})=-h_{i}$.  Also, if
$S=\{i<j<k\}$, then $\psi_{\theta}$ interchanges $v_{S}^+=v_{i}^+v_{j}
^+v_{k}^+$ with
$v_{i}^-v_{j}^-v_{k}^-=-v_{k}^-v_{j}^-v_{i}^-=-v_{S}^-$, so $\zeta^{-1}\psi_{\theta}(v_{S}^\sg)=-v_{S}^{-\sg}$.  Similarly, $\zeta^{-1}\psi_{\theta}(v_{[1,6]}^\sg)=-v_{[1,6]}
^{-\sg}$ and (c) holds.
\end{proof}

 \begin{theorem}\label{thm:E6}
Suppose $\bR $ is a unital commutative
ring and $(M^-,M^+,g)$ satisfies \eqref{eq:condg6}. Let $\cE = \cE(M^-,M^+,g)$.  Then
 \begin{itemize}
\item[(i)]  $\cE$ is a form of the Chevalley algebra of type $\Esix$ and, if
$M^-$ and $M^+$ are free, $\cE$ is the Chevalley algebra of type $\Esix$.
\item[(ii)] $\cE/Z(\cE)$ is a form of the split Lie algebra of type $\Esix$ which is
split if $M^-$ and $M^+$ are free.
\item[(iii)] If $\frac 13\in \bR$, then $Z(\cE) = 0$.
\item[(iv)] If $\bR$ is a field of characteristic $\ne 3$,
$\cE$ is central simple.
\end{itemize}
\end{theorem}

\begin{proof} (i):  By Remarks \ref{rem:freenotation} and  \ref{rem:Galois}, we can assume that  $(M^-,M^+,g)=(\bR ^{6},\bR ^{6},\cdot)$ and show that $\cE$ is the Chevalley algebra of type $\type{E}{6}$.
Then it is clear that the $\bbZ$-linear map
taking elements of the basis $B_\cE$ in $\cE(\bbZ^{6},\bbZ^{6},\cdot)$
to the corresponding
elements of  the basis $B_\cE$ in
$\cE(\bbC ^{6},\bbC ^{6},\cdot)$
is an injective  $\bbZ$-algebra homomorphism.
We use this map to view
$\cE(\bbZ^{6},\bbZ^{6},\cdot)$
as the $\bbZ$-span of the Chevalley basis $B_{\cE}$
of $\cE(\bbC ^{6},\bbC ^{6},\cdot)$.  Since
$\cE(\bR ^{6},\bR^{6},\cdot)\cong\cE(\bbZ^{6},\bbZ^{6},\cdot)_{\bR }$ by Lemma \ref{lem:extendE},
$\cE(\bR ^{6},\bR^{6},\cdot)$ is the Chevalley algebra.

(ii) follows from (i) and Remark \ref{rem:chevalleyform}.

(iii):  By Lemma \ref{lem:form} and Remark \ref{rem:Galois},
we can assume that $(M^-,M^+,g)=(\bR ^{6},\bR ^{6},\cdot)$.  Then a direct calculation
which we leave to the reader shows that
$Z(\cE) = \set{c\, h_M : c\in \bR,\ 3c = 0}$, which is 0 when $\frac 13\in \bR$.

(iv): By (i) and (iii), $\cE$ is the Chevalley algebra of type $\Esix$
and $Z(\cE) = 0$.  Hence, by \cite[2.6(5)]{St},
$\cE$ is simple.  Moreover, if $\bbF$ is a field containing $\bR$,
then $\cE_\bbF \simeq \cE\tripF$ by Lemma \ref{lem:extendE}, so $\cE_\bbF$ is simple
over $\bbF$.  Thus, $\cE$ is central simple
by \cite[Thm.~II.1.7.1]{Mc}.
\end{proof}

\begin{remark}
The article \cite{F} also uses exterior  algebras to construct forms of Chevalley algebras of exceptional type.
However, the Lie algebras obtained there have natural $\bbZ_3$-gradings rather than the $5$-gradings
and $\BCtwo$-gradings that we need in order to construct Kantor pairs.
\end{remark}

\subsection{The Kantor pair $\Wg_3$}
\label{subsec:Wg3}
We now use the results of Subsection \ref{subsec:E6Lie} to construct a Kantor pair $\Wg_3$, with simply described underlying modules and products.

\begin{theorem}
\label{thm:Wg3}
Suppose $\bR $ is a unital commutative ring containing $\frac 16$ and $(M^-,M^+,g)$ satisfies \eqref{eq:condg6}.
Let
$\Wg_{3}=\Wg_{3}(M^-,M^+,g):=(\Wg
_{3}(M^-),\Wg_{3}(M^+))$
with trilinear products
\[
\{x^\sg y^{-\sg}z^\sg\}^\sg=E^\sg(x^\sg,y^{-\sg
})\circ z^\sg-(x^\sg\cdot y^{-\sg})z^\sg.
\]
(See Subsections \ref{subsec:dotact}--\ref{subsec:BA} for the notation used
here.)
Then
\begin{itemize}
\item[(i)] $\Wg_{3}$ is a form of a split Kantor pair of type $\Esix$, which is split if
    each $M^\sg$ is free.  Also $\Wg_3$ is tightly enveloped by the $5$-graded
Lie algebra $\cE = \cE(M^-,M^+,g)$, so
$ \Kan(\Wg_3)  \simeq \cE$
as $5$-graded Lie algebras.
\item[(ii)] The Jordan obstruction of $\Wg_3$ is isomorphic to $(\Wg^6(M^-), \Wg^6(M^+))$
with products $\set{p,q,r}^\sg = 2 (p \cdot q) r = (p \cdot q) r+  (r \cdot q) p$.

\item[(iii)]   If
$\bR $ is a field, $\Wg_{3}$ is a central simple split Kantor pair of type
$\Esix$ of balanced dimension $20$ and balanced $2$-dimension $1$.
\end{itemize}
\end{theorem}

\begin{proof}
(i): It follows from \eqref{eq:3prod} (with $i=1$) that the trilinear pair
$\Wg_3$ is the Kantor pair enveloped by $\cE$, and in particular
it is a Kantor pair.  The fact that $\cE$ tightly envelops
$\Wg_3$ follows from Proposition \ref{prop:E0fgp} and Theorem \ref{thm:E6}(ii).
So $\Kan(\Wg_3) \simeq \cE$ as $5$-graded Lie algebras by Corollary \ref{cor:Kanchar}.
Hence by Lemma \ref{lem:splitform} and Theorem \ref{thm:E6}(ii),
$\Wg_{3}$ is a form of a split Kantor pair of type $\Esix$;
and this form is split if  each $M^\sg$ is free by
Theorem \ref{thm:E6}(i).

(ii):
By (i) and \eqref{eq:JPL},
$J(\Wg_3) \simeq (\cE_{-2}, \cE_2) = (\Wg_6(M^-),\Wg_6(M^+))$  under the products $[[X,Y],Z]$ in $\cE$.  The formulas for the products now  follow from
\eqref{eq:3prod2}.

(iii)  $\Wg_3$ is central simple by Theorem \ref{thm:E6}(iii) and Theorem
\ref{thm:LdetP}(iii).
The dimension statements follow from the definition of $\Wg_3$ and (ii).
\end{proof}

\begin{remark}  \label{rem:E6example}
Suppose  that $\bR$ is an algebraically closed  field of characteristic $0$.

(i)  The weighted Dynkin diagram corresponding to $\Wg_3$ in Kantor's classification
(see Remark \ref{rem:fdclass}) is the first diagram below, whereas the one corresponding to its reflection $\Wgref$
(described in Subsection \ref{subsec:BCtwoE}) is the second  diagram.
%%%%%%%Sine weighted Dynkin diagrams
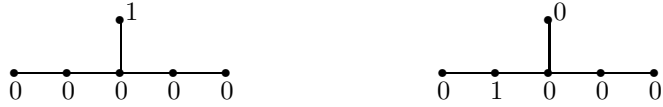
\newcommand\dynkindot{\circle*{.14}}
\setlength{\unitlength}{20pt}
\begin{minipage}{.45\textwidth}
\centering
\begin{picture}(4.2,2.2)(0,-.8)
%Lines
\put(-.2,0){ \line(1,0){4} }  %Adjusted for apparent horizontal bug
\put(1.85,0){ \line(0,1){1} } %Adjusted for apparent horizontal bug
%%Points
\put(0,0){\dynkindot}
\put(1,0){\dynkindot}
\put(2,0){\dynkindot}
\put(3,0){\dynkindot}
\put(4,0){\dynkindot}
\put(2,1){\dynkindot}
%Labels
	\put(-.10,-.5){$0$} %mu1
	\put(.9,-.5){$0$}  %mu2
	\put(1.9,-.5){$0$} %mu3
	\put(2.9,-.5){$0$} %mu4
	\put(3.9,-.5){$0$} %mu5
	\put(2.1,1){$1$}  %mu6
%\label{fig:one}
\end{picture}
%%%%Reflected Dynkin diagram
\end{minipage}%
\begin{minipage}{.45\textwidth}
\centering
\begin{picture}(4.2,2.2)(0,-.8)
%Lines
\put(-.2,0){ \line(1,0){4} }   %Adjusted for apparent horizontal bug
\put(1.85,0){ \line(0,1){1} }  %Adjusted for apparent horizontal bug
%%Points
\put(0,0){\dynkindot}
\put(1,0){\dynkindot}
\put(2,0){\dynkindot}
\put(3,0){\dynkindot}
\put(4,0){\dynkindot}
\put(2,1){\dynkindot}
%Labels
	\put(-.1,-.5){$0$} %mu1
	\put(.9,-.5){$1$}  %mu2
	\put(1.9,-.5){$0$} %mu3
	\put(2.9,-.5){$0$} %mu4
	\put(3.9,-.5){$0$} %mu5
	\put(2.1,1){$0$}  %mu6
%\label{fig:one}
\end{picture}
\end{minipage}%

(ii)  The Kantor pair $\Wg_{3}$ is a simplified and basis free version
of the double of the KTS that was described by Kantor without full proofs in \cite[(6.11) and
\S 6.6]{K1}.

$\Wg_{3}$  is also isomorphic to the signed double of a $(1,1)$-Freudenthal-Kantor triple system.  Indeed,
Elduque and Kochetov have defined the structure of a symplectic triple system on
$\cT = \Wg^3(V)$ (although it appears that
the scalar -24 used in this definition should be replaced by -2), where $V$ is a $6$-dimensional space \cite[\S 6.4]{EK} . The signed double of the
corresponding $(1,1)$-Freudenthal triple system  is a Kantor pair $\cP(\cT)$ (see Example \ref{ex:symplectic}) that can be shown directly
to be isomorphic to $\Wg_3$.

Finally, it can be seen from Kantor's classification that
$\Wg_{3}$  is isomorphic to
the double of the structurable algebra
$\smat{\bR  & J\\
J & \bR }$,
where $J$ the Jordan algebra of $3\times3$-matrices over $\bR$
(see \cite[\S 2]{K2} and \cite[\S 8]{A1}).
However from this matrix point of view a natural non-trivial SP-grading is not  apparent to us.
\end{remark}

\subsection{The Kantor pair $\Wgref$} \label{subsec:BCtwoE}
Suppose from now on that $(M^-,M^+,g,e)$ is a quadruple satisfying
\begin{equation}  \label{eq:condge}
\parbox{4truein}{\centering $g:M^-\times M^+\rightarrow\bR$ is a non-singular bilinear form, \\
$M^-$ and $M^+$ are f.g.p.~modules of rank $6$, and\\
$e=(e^-,e^+)\in M^-\times M^+$ satisfies $g(e^-,e^+)=1$.}
\end{equation}
As  in Subsection \ref{prop:BC2skew}, we now use $e$ to define a $\BCtwo$-grading on $\cE$.

Once again, we have $M^\sg=\bR e^\sg\oplus
U^\sg$, where $U^\sg=(e^{-\sg})^\perp$ relative to $g$ in $M^\sg$.
We write  $M^\sg$ as
$\smat{\bR e^\sg\\
U^\sg}$, and correspondingly identify
\[\End(M^\sg) =
\mat{  \End(\bR e^\sg) & \Hom(U^\sg,\bR e^\sg)\\
\Hom(\bR e^\sg,U^\sg) & \End(U^\sg)  }.\]

\begin{proposition}
\label{prop:BC2E6}
$\cE=\bigoplus_{(i_1,i_2)\in\bbZ^2}~\cE
_{i_1,i_2}$ is a $\BCtwo$-grading of the Lie algebra $\cE$, where
\begin{equation}\label{eq:BC2E6}
\begin{gathered}
\cE_{\sg1,0}
=\Wg_{3}(U^\sg),\quad \cE_{\sg1,\sg1}
=\Wg_2(U^\sg)e^\sg,\\
\quad  \cE_{\sg2,\sg1}
=\Wg_{6}(M^\sg),\\
\cE_{0,0}
=\cE_0\cap\big(\iota^+(\End(\bR e^\sg)\oplus\End(U^\sg))+\bR h_0\big),\\
\cE_{0,1}
=\iota^+(\Hom(U^+,\bR e^+)),\quad  \cE_{0,-1}
=\iota^+(\Hom(\bR e^+,U^+)),
\end{gathered}
\end{equation}
and $\cE_{i_1,i_2}=0$ for all other pairs $(i_1,i_2) \in\bbZ^2$.
Moreover the first component grading of this $\BCtwo$-grading
is the $5$-grading of $\cE$ in Subsection \ref{subsec:E6Lie}.
\end{proposition}

\begin{proof}  We follow the proof of Proposition \ref{prop:BC2skew} with $\ffo(\tg)$ replaced by $\cE$.  As in the last paragraph of that proof, we can (by a base ring extension argument using Lemma \ref{lem:extendE})
assume that there exists $t\in \bR$ such that $(t^i-t^j)x = 0,\ x \in \cE \implies x = 0 \text{ or } i = j.$
Let $\theta^\sg \in \GL(M^\sg)$ be left multiplication by
$\smat{t^{\sg1} & 0\\0 & 1}$.
Since $\theta=(\theta^-,\theta^+)$ is an automorphism of $(M^-,M^+,g)$, it induces an automorphism $\psi_{\theta}$ of
$\cE(M^-,M^+,g)$
as in Remark  \ref{rem:isoE}(i).  The proof now proceeds as
in \ref{prop:BC2skew}.
\end{proof}

By Proposition \ref{prop:BC2E6} and Remark \ref{rem:BC2vs5gr}, we have the following result which describes
an SP-grading on the Kantor pair
$\Wg_3$ constructed in Theorem \ref{thm:Wg3}.

\begin{proposition}
\label{prop:SPWg3}  Suppose $\frac 16\in \bR$, and let
\begin{equation}
(\Wg_{3})_0^\sg=\Wg_{3}(U^\sg)\text{ and }(\Wg_{3})_1^\sg=\Wg_2(U^\sg)e^{\sg
}.\label{eq:Wg3SP1}
\end{equation}
Then $\Wg_{3}=(\Wg_{3})_0\oplus
(\Wg_{3})_1$ is an SP-graded Kantor pair, which is tightly enveloped
by the $\BCtwo$-graded Lie algebra $\cE$ described in Proposition \ref{prop:BC2E6}.
\end{proposition}

We now have the main result about Kantor pairs in this section.

\begin{theorem}
\label{thm:E6brv}
 Suppose $\bR $ is a unital commutative
ring containing $\frac 16$ and $(M^-,M^+,g,e)$ satisfies \eqref{eq:condge}.  Then
\begin{itemize}  \item[(i)]
The reflection $\Wgref$ of the SP-graded Kantor pair $\Wg_3$ described
in Proposition \ref{prop:SPWg3}  is an SP-graded form of a split Kantor pair of type $\Esix$ which is split if each $M^\sg$ is free.
\item[(ii)] The Jordan obstruction $J$ of $\Wgref$ is isomorphic
to $U^\op$, where $U = (U^-,U^+)$ with products
$\set{u^\sg,v^\msg,w^\sg}^\sg = g(u^\sg,v^\msg) w^\sg + g(w^\sg,v^\msg) u^\sg$.
\item[(iii)] If
$\bR $ is a field, $\Wgref$ is a central simple split Kantor pair of type
$\Esix$ of balanced dimension $20$ and balanced $2$-dimension $5$.
\end{itemize}
\end{theorem}

\begin{proof}

(i) follows from Theorem \ref{thm:Wg3} and Proposition \ref{prop:simplereflect}.

(ii):  By Proposition \ref{prop:obstruct}, Theorem \ref{thm:Wg3}(i) and
Proposition \ref{prop:BC2E6},
\[J^\op \simeq (\cE_{0,1}, \cE_{0,-1}) = \big(\iota^+(\Hom(U^+,\bR e^+)),\iota^+(\Hom(\bR e^+,U^+))\big)\]
under the products $[[X,Y],Z]$ in $\cE_{0,*}$.
Since  there are natural module isomorphisms
$U^- \simeq \Hom(U^+,\bR e^+) $ and $U^+ \simeq \Hom(\bR e^+,U^+)$ (induced by $g$ in the first case),
our conclusion is easily checked.

(iii):  $\Wgref$ is central simple split Kantor pair of type
$\Esix$ by Corollary \ref{thm:Wg3}(iii) and
Proposition \ref{prop:simplereflect}.
Also since $(\Wgref)^\sg = \Wg_{3}(U^\msg) \oplus \Wg_2(U^\sg)e^\sg$,
$\Wgref$ has balanced dimension $10 + 10$; while $\Wgref$ has balanced
$2$-dimension $5$ by (ii).
\end{proof}

If $(M'^-,M^+,g',e')$ is another quadruple satisfying \eqref{eq:condge}, an \emph{isomorphism}
of  $(M^-,M^+,g,e)$ onto $(M'^-,M'^+,g',e')$
is an isomorphism of $(M^-,M+,g)$ onto $(M'^-,M^+,g')$ which maps $e$ onto $e'$.
If such an isomorphism exists, one sees using Remark \ref{rem:isoE}(i)
that the $\BCtwo$-graded Lie algebras $\cE$ and $\cE'$ constructed above
are graded-isomorphic; and, if $\frac 16\in \bR$, the SP-graded Kantor pairs
$\Wg_3$ and  $\Wg'_3$ are graded isomorphic, as are the
SP-graded pairs $\Wgref$ and  $(\Wg'_3)\brv$.

\begin{remark}
\label{rem:BC2E} Suppose $\bR$ is a field. Then one easily checks that
any two quadruples satisfying \eqref{eq:condge}  are isomorphic.
Hence \emph{the $\BCtwo$-graded Lie algebra $\cE$ constructed above
is independent up to graded isomorphism of the choice of
$(M^-,M^+,g,e)$;
and, if $\frac 16\in \bR$, so too are the SP-graded Kantor pairs
$\Wg_3$ and $\Wgref$}.
\end{remark}

\begin{remark}
\label{rem:Wg3refb}  Suppose  $\bR $ is an algebraically closed field
of characteristic 0. The construction of $\Wgref$
given above is a simple new basis free construction of the double of the KTS
$C_{55}^2$ constructed by Kantor without full proofs in \cite[\S 4]{K2} and \cite[\S 6.6]{K1}.   The pair $\Wgref$ is of
particular interest since it is one of the two  split simple Kantor pairs of
exceptional type that does not arise by doubling a structurable algebra---the
other being the famous Jordan pair $(M_{1,2}(C),M_{1,2}(C^{op}))$ determined
by a Cayley algebra~$C$ \cite[\S 8.15]{L}.
\end{remark}

\subsection{An example using rank 1 modules}
\label{subsec:Dedekind}
If $I$ is an  f.g.p.~module of rank 1, set
\[M_I = \bR^5 \oplus I \andd e_I = (e^-_I,e^+_I) \in (M_I^*,M_I),\]
where $e^+_I = (1,0,0,0,0)$ , $e^-_I(a_1,\dots,a_5) = a_1$ and $e^-_I|_I = 0$.
Then $(M_I^*,M_I,\text{can})$ satisfies \eqref{eq:condg6} and  $(M_I^*,M_I,\text{can},e_I)$ satisfies \eqref{eq:condge}  (see Remark \ref{rem:onemodule}(ii)).
Let $\cE_I$ be the $\BCtwo$-graded Lie algebra
constructed   from $(M_I^*,M_I,\text{can},e_I)$ in Proposition \ref{prop:BC2E6}.
We also regard  $\cE_I$ as a $5$-graded Lie algebra with the first component grading.
If $\frac 16 \in \bR$, let $\Wg_{3,I}$ be the SP-graded Kantor pair
constructed  from
$(M_I^*,M_I,\text{can},e_I)$ in Proposition \ref{prop:SPWg3}.

Suppose that $I$ and $I'$ are f.g.p.~modules of rank $1$.  We now show that
\begin{equation}
\label{eq:isoI1}
\cE_I \simeq_{\BCtwo} \cE_{I'} \iff \cE_I \simeq_{\text{5-gr}} \cE_{I'}  \iff I \simeq I',
\end{equation}
where $\simeq_{\BCtwo}$ (resp.~$\simeq_{\text{5-gr}}$) indicates isomorphism as
$\BCtwo$-graded (resp.~$5$-graded) Lie algebras.
Indeed, denoting these statements by (a), (b) and (c) in order, it is clear that
(a) implies (b).  Suppose (b) holds.  Then $(\cE_I)_2 \simeq (\cE_{I'})_2$ as $\bR$-modules
so $\Wg_6(M_I) \simeq  \Wg_6(M_{I'})$.  But
$\Wg_6(M_I) = \Wg_6(\bR^5 \oplus I) \simeq \Wg_5(\bR^5)\otimes \Wg_1( I) \simeq R \otimes I \simeq I$ using \cite[III.7.7,Prop.~10]{B2}, and we have (c).  Suppose finally that (c) holds.  Then the quadruples $(M_I^*,M_I,\text{can},e_I)$ and
$(M_{I'}^*,M_{I'},\text{can},e_{I'})$ are isomorphic, so (a) holds (as noted after Theorem \ref{thm:E6brv}).

Suppose next that $\frac 16\in \bR$ and
$I,I'$ are f.g.p.~modules of rank $1$.
We now show that
\newcommand\compactiff{\hspace{-.5em}\iff\hspace{-.1em}}
\begin{align}
\label{eq:isoI2}
\begin{split}
(\Wg_{3,I}) \brv \simeq (\Wg_{3,I'})\brv
\compactiff
(\Wg_{3,I})\brv \simeq_{\text{SP}} (\Wg_{3,I'})\brv
&\compactiff
\Wg_{3,I} \simeq_{\text{SP}} \Wg_{3,I'},\\
&\compactiff
\Wg_{3,I} \simeq\Wg_{3,I'}
\compactiff
I \simeq I',
\end{split}
\end{align}
where $\simeq_{\text{SP}}$  indicates isomorphism as
SP-graded Kantor pairs.
To see this, we denote these statements by ($\alpha$), ($\beta$), ($\gamma$),  ($\delta$) and ($\varepsilon$) in order.
Then ($\gamma$) (resp.~($\delta$)) is equivalent  to (a) (resp.~(b)) in \eqref{eq:isoI1}. Also the equivalence of ($\beta$) and ($\gamma$) is clear (for example from Proposition \ref{prop:reflect}). So ($\beta$), ($\gamma$), ($\delta$) and ($\varepsilon$) are equivalent.  Since $(\beta)$ implies $(\alpha)$, it is enough now to check that $(\alpha)$ implies $(\varepsilon)$. But if $(\alpha)$ holds, then
$J((\Wg_{3,I}) \brv) \simeq J((\Wg_{3,I'}) \brv)$, so $\bR^4 \oplus I \simeq \bR^4 \oplus I'$  by Theorem \ref{thm:E6brv}(ii). Hence, $\Wg_5(\bR^4 \oplus I) \simeq \Wg_5(\bR^4 \oplus I')$, so $I \simeq I'$ as above.

\emph{Suppose from now on that $\bR$ is a Dedekind domain}. Let
 $\Pic(\bR)$ (the \emph{Picard group} of $\bR$) be the group
of all isomorphism classes of f.g.p.~modules of rank $1$ under the product induced from the tensor product.
Recall that
$\Pic(\bR)$ is naturally isomorphic to the \emph{ideal class group} $\mathfrak C(\bR)$  of $\bR$   \cite[II.5.7, Prop.~12]{B2}.

Recall further that   any f.g.p.~module of rank $m\ge 1$ is isomorphic to $\bR^{m-1}\oplus I$
for   some f.g.p.~module $I$ of rank $1$ (see for example \cite[\S 1.3]{Na}). Using this fact with $m=6$ (resp.~$m=5$) it is easy to see that any triple satisfying  \eqref{eq:condg6}
(resp.~ any quadruple satisfying  \eqref{eq:condge})
is isomorphic to $(M_I^*,M_I,\text{can})$
(resp.~$(M_I^*,M_I,\text{can},e_I)$) for some $I$.

Therefore all of the  following algebraic structures are obtained from some
$I$ as above:
(i) $5$-graded Lie algebras $\cE$ in Theorem \ref{thm:E6};
(ii) (ungraded) Kantor pairs $\Wg_3$ in Theorem \ref{thm:Wg3};
(iii)
$\BCtwo$-graded Lie algebras $\cE$ in Proposition \ref{prop:BC2E6};
(iv) SP-graded Kantor pairs $\Wg_3$ in Proposition \ref{prop:SPWg3};
(v) SP-graded pairs $\Wgref$ in Theorem \ref{thm:E6brv}; and
(vi) (ungraded) Kantor pairs $\Wgref$ in Theorem \ref{thm:E6brv}.
So, by \eqref{eq:isoI1} and \eqref{eq:isoI2},
the sets of graded-isomorphism classes for each of the families (i), (iii), (iv) and (v), as well as the
sets of isomorphism classes for the families (ii) and (vi), are each in 1-1 correspondence with $\Pic(\bR)$.

Since any abelian group arises as the ideal class group of some Dedekind domain \cite[Thm.~1.4]{L-G},
we see that we have many examples of the indicated structures.

\end{document}